\documentclass[final,hidelinks,onefignum,onetabnum]{siamart220329}

\usepackage{lipsum}
\usepackage{amsfonts}
\usepackage{graphicx}
\usepackage{epstopdf}
\usepackage{algpseudocode}
\usepackage{enumerate}

\ifpdf
  \DeclareGraphicsExtensions{.eps,.pdf,.png,.jpg}
\else
  \DeclareGraphicsExtensions{.eps}
\fi

\newsiamremark{remark}{Remark}
\newsiamremark{hypothesis}{Hypothesis}
\crefname{hypothesis}{Hypothesis}{Hypotheses}
\newsiamthm{claim}{Claim}

\title{Neumann Control Problem
}

\usepackage{amsopn}

\makeatletter
\newcommand*{\addFileDependency}[1]{% argument=file name and extension
  \typeout{(#1)}% latexmk will find this if $recorder=0 (however, in that case, it will ignore #1 if it is a .aux or .pdf file etc and it exists! if it doesn't exist, it will appear in the list of dependents regardless)
  \@addtofilelist{#1}% if you want it to appear in \listfiles, not really necessary and latexmk doesn't use this
  \IfFileExists{#1}{}{\typeout{No file #1.}}% latexmk will find this message if #1 doesn't exist (yet)
}
\makeatother

\usepackage[textsize=small]{todonotes}
\usepackage{commath}
\usepackage{bm}
\usepackage{multirow}
\usepackage{booktabs}
\usepackage{graphicx}
\usepackage{stmaryrd}
\usepackage{soul}
\usepackage{subcaption}
\usepackage{empheq}
\usepackage{enumitem}
\usepackage{amssymb}
\usepackage{amsmath}
\usepackage{cleveref}

\usepackage[textsize=small]{todonotes}
\usepackage{soul}
\usepackage{longtable}
\usepackage{array}

\newcolumntype{S}{>{$\displaystyle}p{0.28\linewidth}<{$}}
\newcolumntype{D}{p{0.66\linewidth}}
\renewcommand{\arraystretch}{1.25}
\definecolor{ForestGreen}{rgb}{0.13,0.55,0.13}
\allowdisplaybreaks

\begin{document}

\author{Harbir Antil\thanks{Department of Mathematical Sciences and Center for Mathematics and Artificial Intelligence, George Mason University, Fairfax, VA 22030 (\email{hantil@gmu.edu})}\and
Keegan L.A. Kirk\thanks{Department of Mathematics, Louisiana State University, Baton Rouge, LA 70803(\email{Keegan.Kirk@lsu.edu})}\and
Umarkhon Rakhimov\thanks{Department of Mathematical Sciences and Center for Mathematics and Artificial Intelligence, George Mason University, Fairfax, VA 22030 (\email{urakhimo@gmu.edu})}}

\title{A Function-Space Framework for BDDC Preconditioning in Control- and State-Constrained Sparse Optimal Control \thanks{This work is partially supported by the Office of Naval Research (ONR) under Award No. N00014-24-1-2147, the
National Science Foundation (NSF) under Grant DMS-2408877, and the Air Force Office of Scientific Research (AFOSR)
under Award No. FA9550-22-1-0248.}}

\maketitle 

\begin{abstract}
We develop a Balancing Domain Decomposition by Constraints (BDDC) preconditioner for the interface systems arising from active-set semi-smooth Newton linearizations of elliptic optimal control problems with box-constrained controls, $L^1$-sparsity, and Moreau--Yosida regularized state constraints. The domain decomposition and BDDC construction are formulated directly at the infinite-dimensional level. Eliminating the interior variables yields a Schur complement for the state and adjoint traces assembled from local subdomain operators. We prove a global--local equivalence result, establish well-posedness of the local and partially assembled problems under verifiable conditions on the primal constraints. The BDDC operator admits a two-level additive Schwarz representation with independent local solves and a finite-dimensional coarse solve. We implement the preconditioner using conforming finite elements and test it for distributed and Neumann boundary control. For a fixed ratio of subdomain diameter $H$ to mesh size $h$, the GMRES iteration counts are nearly independent of mesh refinement and grow only moderately as this ratio increases. Enriching the coarse space substantially improves convergence and reduces sensitivity to the control regularization and sparsity weights. For the state-constrained problem, scaling the Moreau–Yosida penalty parameter proportionally to $h^2$ yields nearly constant GMRES iteration counts as the mesh is refined.
\end{abstract}

\medskip

\section{Introduction}

We develop a Balancing Domain Decomposition by Constraints (BDDC)
preconditioner~\cite{Dohrmann} for the linear systems generated by
active-set semi-smooth Newton methods for elliptic optimal control.  The
model class combines box-constrained controls, an
$L^1$ sparsity-promoting term, and pointwise state constraints.  Let
$\Omega\subset\mathbb R^d$ be a Lipschitz domain and let
$\mathcal U=\Omega$ for distributed control or
$\mathcal U=\partial\Omega$ for Neumann boundary control.  Given
$y_d\in L^2(\Omega)$, bounds $y_a,y_b\in L^\infty(\Omega)$ and
$u_a,u_b\in L^\infty(\mathcal U)$, and parameters $\alpha>0$,
$\beta\ge 0$, and $\sigma\ge 0$, we consider
\begin{subequations}\label{eq:ocp}
\begin{gather}
  \min_{(y,u)} \mathcal{J}(y,u) :=
    \tfrac{1}{2}\|y-y_d\|_{L^2(\Omega)}^2
  + \tfrac{\alpha}{2}\|u\|_{L^2(\mathcal U)}^2
  + \beta\|u\|_{L^1(\mathcal U)}
  \label{eq:ocp_obj} \\[-0.25em]
  \text{subject to} \notag \\[0.1em]
  y_a \le y \le y_b
    \quad \text{a.e.\ in }\Omega,
  \label{eq:ocp_state_bounds}\\
  u_a \le u \le u_b
    \quad \text{a.e.\ in }\mathcal U,
  \label{eq:ocp_control_bounds}
\end{gather}
\vspace{-7mm}
\begin{equation}\label{eq:ocp_control_PDE}
  \left\{
  \begin{aligned}
    -\Delta y+ \sigma y - u &= 0 &&\text{in }\mathcal{U}:=\Omega,\\
    y &= 0 &&\text{on }\partial\Omega,
  \end{aligned}
  \right.
  \; \; \;\text{or}\; \; \;
  \left\{
  \begin{aligned}
    -\Delta y + \sigma y &= 0 &&\text{in }\Omega,\\
    \partial_n y  - u &= 0 &&\text{on }\mathcal U:=\partial\Omega.
  \end{aligned}
  \right.
\end{equation}
\end{subequations}
These problems serve as model problems for a large class of PDE constrained optimization problems arising in optimal control, inverse problems, data assimilation, and optimal design. Their finite element discretization
and semi-smooth Newton solution require the repeated solution of large,
active-set-dependent KKT systems.  This motivates a nonoverlapping
domain decomposition method that permits parallel subdomain solves
while preserving the coupled optimality-system structure.

%----------------------------------------------------------------------
\subsection{Semi-smooth Newton methods and preconditioning}
\label{ss:intro_ssn}
%----------------------------------------------------------------------

The $L^1$ term and the control bounds give rise to a max--min
complementarity operator that is semi-smooth from
$L^q(\mathcal U)\times L^q(\mathcal U)$ to $L^2(\mathcal U)$ for
$q>2$, yielding locally superlinear semi-smooth Newton convergence.
Because this convergence analysis is carried out at the function-space
level, it transfers uniformly to suitable conforming finite element
discretizations, yielding mesh-independent Newton iteration counts
\cite{Ulbrich2003,UlbrichBook2011,HintermuuellerUlbrich2004}.
Pointwise state constraints, however, do not fit this framework directly.  Their
multipliers are generally regular Borel measures rather than $L^2$
functions~\cite[Theorem~6.5]{troltzsch2010optimal}, so the associated
complementarity relation cannot be formulated as a semi-smooth
Nemytskii equation on the $L^p$-based spaces used for the control
constraints.  We therefore replace the state constraints by the
Moreau--Yosida regularization
\cite{HintermuellerHinze2009}.
\begin{align}\label{eq:ocp_regularized_obj}
  \mathcal{J}_\gamma(y,u) :={}
  & \tfrac{1}{2}\|y - y_d\|_{L^2(\Omega)}^2
  + \tfrac{\alpha}{2}\|u\|_{L^2(\mathcal{U})}^2
  + \tfrac{1}{2\gamma}\|\max\{0,y-y_b\}\|_{L^2(\Omega)}^2
  \notag\\
  & + \tfrac{1}{2\gamma}\|\min\{0,y-y_a\}\|_{L^2(\Omega)}^2
  + \beta\|u\|_{L^1(\mathcal{U})},
\end{align}
and consider the penalized problem of minimizing
$\mathcal{J}_\gamma$ subject to~\eqref{eq:ocp_control_bounds} and the
state equation~\eqref{eq:ocp_control_PDE}. 
As
$\gamma\to0^+$, the regularized solutions converge to solutions of the
original state-constrained problem
\cite{HintermuellerHinze2009}.  Alternatively, one may use a
Lavrentiev-type regularization
\cite{MeyerRoeschTroeltzsch2006,HintermuuellerTroeltzschYousept2008};
the present framework extends to this setting with minor
modifications of the active-set structure
\cite{PorcelliSimonciniTani2015}. The parameters $\alpha$, $\beta$, and $\gamma$ affect the conditioning
and active-set structure of the Newton systems, and their influence on
the proposed preconditioner is examined in the numerical experiments.

The semi-smooth Newton method provides superlinear convergence, but
each iteration requires solving a large, sparse, symmetric but
indefinite $4\times 4$ double saddle-point KKT system whose structure
changes with the active sets. First-order methods are inexpensive per
iteration but converge slowly; the tradeoff favors Newton precisely
when the linear subproblem solver is efficient, which is the problem
this paper addresses.
Block preconditioners for these systems have been developed
in~\cite{BattermannSachs2001,ReesDollarWathen2010,PearsonStollWathen2012,SchoeberlZulehner2007,PorcelliSimonciniTani2015,Porcelli2017Preconditioning,HerzogSachs2010};
see~\cite{BenziGolubLiesen2005} for a survey. When
designed at the operator level, the framework of operator
preconditioning~\cite{MardalWinther2011,SchielaUlbrich2014,MardalNielsenNordaas2017}
can yield mesh-independent spectral equivalences. The present work
adopts a similar operator-level perspective, though owing to the indefinite nature of the problem  

%----------------------------------------------------------------------
\subsection{Domain decomposition preconditioners}
\label{ss:intro_dd}
%----------------------------------------------------------------------

Nonoverlapping domain decomposition eliminates subdomain-interior
unknowns and iterates on an interface Schur complement.  One-level
Neumann--Neumann methods lack global coupling, whereas Balancing Neumann Neumann (BNN)
\cite{Mandel1993,DryjaWidlund1995} and BDDC
\cite{Dohrmann,MandelDohrmann2003} add a coarse problem.  For symmetric
positive-definite elliptic problems, BDDC admits polylogarithmic condition
number bounds in the ratio $H/h$~\cite{MandelDohrmannTezaur2005}; related
extensions cover several indefinite and nonsymmetric systems
\cite{LiWidlund2006,LiTu2009,LiuZhang2025b,TosellWidlund2005}.

In PDE-constrained optimization, domain decomposition is often used inside
a block preconditioner for the state and adjoint PDEs
\cite{LangerSteinbach2021,ReesDollarWathen2010}.  Alternatively, one may
decompose the full KKT system, so that each local solve remains a coupled
optimal control problem.  Heinkenschloss and
Nguyen~\cite{heinkenschloss2006neumann} introduced this approach for
unconstrained elliptic control, with subsequent extensions
\cite{BartlettHeinkenschlossRidzalWaanders2006,HeinkenschlossHerty2007}.  
Liu and Zhang~\cite{LiuZhang2025a,LiuZhang2025b} recently developed
BDDC preconditioners for HDG discretizations of unconstrained problems,
proving polylogarithmic GMRES bounds after reducing to a positive
definite interface problem. In the constrained setting
considered here, the active-set-dependent Schur complement is generally indefinite, so that reduction is unavailable.

%----------------------------------------------------------------------
\subsection{Contributions}
\label{ss:intro_contributions}
%----------------------------------------------------------------------

The main contributions of the present work are as follows:
\begin{itemize}
\item We develop a variational nonoverlapping domain decomposition
framework for the coupled symmetrized KKT systems arising at each
iteration of an active-set semi-smooth Newton method. This extends the results of
\cite{heinkenschloss2006neumann} from unconstrained linear-quadratic
problems to elliptic optimal control with control constraints,
$L^1$-sparsity, and Moreau--Yosida regularized state constraints. We
prove a global--local equivalence result, derive the assembly of the global
interface Schur complement, and establish well-posedness of the local
subdomain problems.

\item We construct a BDDC preconditioner for the state--adjoint
interface Schur complement. We give verifiable conditions on the
primal constraints ensuring well-posedness of the partially assembled
and local dual-subspace problems, prove injectivity of the
preconditioner under a compatibility condition on the averaging operator, and
derive its two-level additive Schwarz form with independent local
solves and a finite-dimensional coarse solve.

\item We establish several abstract results on the theory of block-operators used in the
analysis that may also be of independent interest (\emph{cf}. \Cref{s:block_theory}). These
include sufficient conditions for double saddle-point operators to be
Fredholm of index zero, together with preservation of the Fredholm
index under Schur complementation and restrictions to finite-codimensional subspaces.

\item Numerical experiments for distributed and Neumann boundary
control show nearly mesh-independent GMRES iteration counts at fixed
$H/h$ and moderate growth with $H/h$. Edge-average enrichment
substantially reduces the counts and their sensitivity to $\alpha$ and
$\beta$, while continuation with $\gamma=h^2$ yields nearly constant
counts for the state-constrained problem.
\end{itemize}

The interface Schur complement is self-adjoint but generally
indefinite, while left BDDC preconditioning produces a generally
nonsymmetric operator; we therefore use GMRES. A rigorous GMRES convergence
theory for the active-set-dependent preconditioned systems remains
open. In particular, the benign-subspace technique for Stokes-type
problems~\cite{LiWidlund2006} is not applicable because the coupling
operator is elliptic, whereas the reduction to a positive-definite
interface problem used in~\cite{LiuZhang2025a,LiuZhang2025b} is
unavailable in the present setting due to the non-trivial active-set structure.

%----------------------------------------------------------------------
\subsection{Outline}
\label{ss:intro_outline}
%----------------------------------------------------------------------

Section~2 formulates the infinite-dimensional optimal control problem and its semi-smooth Newton linearization, and Section~3 develops the domain decomposition framework. Section~4 constructs an abstract infinite-dimensional BDDC preconditioning operator. Section~5 presents numerical experiments. Section~6 draws conclusions. \Cref{s:block_theory} collects a number of results on the Fredholm theory of block operators required for our analysis.

%======================================================================
\section{The continuous optimal control problem}\label{sec:continuous}
%======================================================================

In this section, we formulate the first-order optimality system
for the Moreau--Yosida regularized
problem~\eqref{eq:ocp_regularized_obj},
\eqref{eq:ocp_control_bounds}--\eqref{eq:ocp_control_PDE} as a
nonsmooth operator equation, identify its semi-smoothness properties,
and derive the semi-smooth Newton linearization that forms the basis
for the domain decomposition framework developed in subsequent
sections.

%----------------------------------------------------------------------
\subsection{Function spaces and operators}\label{ss:spaces_ops}
%----------------------------------------------------------------------

The state and adjoint are sought in
\begin{align*}
  V :=
  \begin{cases}
    H_0^1(\Omega), & \text{distributed control},\\
    H^1(\Omega),   & \text{Neumann control},
  \end{cases}
\end{align*}
while the control and multiplier belong to
$U := L^2(\mathcal{U})$, identified with its dual via the Riesz
isomorphism.  We introduce the following bounded linear operators:
\begin{itemize}[leftmargin=2em]
\item $A : V \to V^\star$ denotes the weak Laplacian,
  $\langle Av, w \rangle_{V} := \int_\Omega \nabla v \cdot \nabla w
  \,\mathrm{d}x$;
\item $M : V \hookrightarrow L^2(\Omega)$ denotes the canonical
  (compact) embedding;
\item $B : U \to V^\star$ denotes the control operator, defined by
  $\langle Bu, v \rangle_{V}
   := \int_{\mathcal{U}} u\,(Tv)\,\mathrm{d}\sigma$,
  where $T := M$ in the distributed-control
  case (so that $B = M^\star$) and
  $T := \iota \circ \operatorname{tr}$ in the Neumann-control case,
  with $\operatorname{tr} : V \to H^{\frac{1}{2}}(\partial\Omega)$ the trace
  operator and
  $\iota : H^{\frac{1}{2}}(\partial\Omega) \hookrightarrow U$ the canonical
  embedding.
\end{itemize}
For notational brevity we set
$\widetilde{A} := A + \sigma\,M^\star M : V \to V^\star$.
In the distributed-control case, $B^\star = M$; in the
Neumann-control case, $B^\star = \iota  \operatorname{tr}$. For the remainder of the paper, we will refer to the cases of distributed control and Neumann control by (DC) and (NC), respectively.

%----------------------------------------------------------------------
\subsection{First-order optimality conditions}\label{ss:KKT}
%----------------------------------------------------------------------

The first-order optimality conditions for the regularized
problem~\eqref{eq:ocp_regularized_obj}, \eqref{eq:ocp_state_bounds}-- \eqref{eq:ocp_control_PDE} require the
existence of $(p,\mu) \in V \times U$ such that the quadruplet
$x= (y,p,u,\mu) \in Z:= V \times V \times U \times U$ satisfies
\begin{align} \label{eq:cts_KKT}
  \mathcal{F}(x)
  :=
  \begin{bmatrix}
    \widetilde{A}^\star\,p
      + M^\star\bigl(M\,y - y_d
        + \Phi_\gamma(M\,y)\bigr) \\[3pt]
    \widetilde{A}\,y - B\,u \\[3pt]
    \alpha\,u - B^\star p + \mu \\[3pt]
    \Psi(u,\mu)
  \end{bmatrix}
  = 0
    \quad \text{in } Z^\star.
\end{align}
The nonsmooth state-constraint and control-constraint operators are
the Nemytskii operators
$\Phi_\gamma : L^2(\Omega) \to L^2(\Omega)$ and
$\Psi : U \times U \to U$, defined pointwise a.e.\ by
\begin{align*}
  \bigl[\Phi_\gamma(v)\bigr](x)
    &:= \gamma^{-1}\max\{0,\,v(x) - y_b(x)\}
       + \gamma^{-1}\min\{0,\,v(x) - y_a(x)\}, \\[4pt]
  \bigl[\Psi(q,\eta)\bigr](x)
    &:= \psi\bigl(x,q(x),\eta(x)\bigr),
\end{align*}
where the scalar complementarity function
$\psi : \mathcal{U} \times \mathbb{R} \times \mathbb{R} \to \mathbb{R}$
is given by
\begin{align}\label{eq:psi_def}
  \psi(x,s,r)
    &:= s
       - \max\{0,\, s + c(r-\beta)\}
       - \min\{0,\, s + c(r+\beta)\} \notag\\
    &\quad
       + \max\{0,\, (s - u_b(x)) + c(r-\beta)\}
       + \min\{0,\, (s - u_a(x)) + c(r+\beta)\}.
\end{align}
The constant $c > 0$ is a free parameter whose choice influences the
performance of the semi-smooth Newton method~\cite{HintermuellerItoKunisch2002PDAS}.

\begin{remark}[Derivation of the complementarity function]
\label{rem:psi_derivation}
The fourth equation in~\eqref{eq:cts_KKT} encodes the optimality
condition for the control.  Stationarity of the
Lagrangian with respect to~$u$ yields the inclusion
\begin{align}\label{eq:subdiff_inclusion}
  0 \in \alpha\,u - B^\star p
      + \beta\,\partial\|u\|_{L^1(\mathcal{U})}
      + N_{K}(u),
\end{align}
where $K := \{v \in U : u_a \le v \le u_b\}$ is the admissible set
and $N_K(u)$ denotes its normal cone at~$u$.  Setting
$\mu := \beta\,\xi + \lambda$ with
$\xi \in \partial\|u\|_{L^1(\mathcal{U})}$ and $\lambda \in N_K(u)$,
the third equation in~\eqref{eq:cts_KKT} becomes
$\alpha\,u - B^\star p + \mu = 0$.
As shown in~\cite{Stadler2009Elliptic}, the condition
$\mu \in \beta\,\partial\|u\|_{L^1(\mathcal{U})} + N_K(u)$
is equivalent to $\Psi(u,\mu) = 0$ with $\Psi$ defined
via~\eqref{eq:psi_def}. 
\end{remark}

%----------------------------------------------------------------------
\subsection{Semi-smoothness, active sets, and generalized Jacobians}\label{ss:active_sets}
%----------------------------------------------------------------------

The Nemytskii operators $\Phi_\gamma$ and $\Psi$ are semi-smooth in the
sense of~\cite{Ulbrich2003}, provided they are regarded as mappings between
$L^q$ and $L^2$ spaces with $q >2$. More precisely, $\Phi_\gamma$ is semi-smooth as a
mapping $L^q(\Omega) \to L^2(\Omega)$ and $\Psi$ is semi-smooth as a
mapping $L^q(\mathcal{U}) \times L^q(\mathcal{U}) \to L^2(\mathcal{U})$
for any $q > 2$.  Composing with the Sobolev embedding
$V \hookrightarrow L^q(\Omega)$ (which holds for some $q > 2$ whenever
$d \le 3$), the full KKT operator $\mathcal{F}$ is semi-smooth as a
mapping from
\begin{align*}
  Z_q := \cbr{
    (y,p,u,\mu) \in Z :
    My \in L^q(\Omega),\;
    u,\mu \in L^q(\mathcal{U})}
\end{align*}
to $V^\star \hspace{-0.2mm} \times \hspace{-0.2mm}V^\star \hspace{-0.2mm}\times\hspace{-0.2mm} U \hspace{-0.2mm} \times  \hspace{-0.2mm} U$.  Near a solution
$z_\star \in Z_q$ at which the generalized Jacobian is
non-singular, the semi-smooth Newton method converges
$q$-superlinearly~\cite{Ulbrich2003}.  By the mesh-independence principle
of~\cite{HintermuuellerUlbrich2004}, a conforming finite element discretization of the KKT
system inherits this convergence behavior, so that the number of Newton
iterations required to reach a given tolerance is independent of the
mesh size. It is our goal to develop a BDDC preconditioner for the generalized Jacobian.

\subsubsection{Active sets}
The generalized Jacobians of $\Phi_\gamma$ and $\Psi$ are piecewise
constant and determined by the following partition of the domain
into active and inactive regions. For the \emph{state constraint}, we define
$\mathcal{A}_y := \mathcal{A}_{y,a} \cup \mathcal{A}_{y,b}$, where
\begin{align*}
  \mathcal{A}_{y,a}
    &:= \bigl\{x \in \Omega : y(x) < y_a(x)\bigr\}, &
  \mathcal{A}_{y,b}
    &:= \bigl\{x \in \Omega : y(x) > y_b(x)\bigr\}.
\end{align*}
For the \emph{control constraint}, the active set is
$\mathcal{A}_{u,\mu}
  := \mathcal{A}_{u,\mu,a}
   \cup \mathcal{A}_{u,\mu,b}
   \cup \mathcal{A}_{u,\mu,0}$,
with
\begin{alignat*}{1}
  \mathcal{A}_{u,\mu,a}
    &:= \bigl\{x \in \mathcal{U} :
           (u-u_a)(x) + c\bigl(\mu(x)+\beta\bigr) < 0 \bigr\},\\
  \mathcal{A}_{u,\mu,b}
    &:= \bigl\{x \in \mathcal{U} :
           (u-u_b)(x) + c\bigl(\mu(x)-\beta\bigr) > 0 \bigr\},\\
  \mathcal{A}_{u,\mu,0}
    &:= \bigl\{x \in \mathcal{U} :
           u(x)+c\bigl(\mu(x)+\beta\bigr) \ge 0
           \;\text{and}\;
           u(x)+c\bigl(\mu(x)-\beta\bigr) \le 0 \bigr\},
\end{alignat*}
and the inactive set is
$\mathcal{I}_{u,\mu}
  := \mathcal{I}_{u,\mu,+} \cup \mathcal{I}_{u,\mu,-}$,
with
\begin{alignat*}{1}
  \mathcal{I}_{u,\mu,+}
    &:= \bigl\{x \in \mathcal{U} :
           u(x)+c\bigl(\mu(x)-\beta\bigr) > 0
           \;\text{and}\;
           (u-u_b)(x)+c\bigl(\mu(x)-\beta\bigr) \le 0 \bigr\},\\
  \mathcal{I}_{u,\mu,-}
    &:= \bigl\{x \in \mathcal{U} :
           u(x)+c\bigl(\mu(x)+\beta\bigr) < 0
           \;\text{and}\;
           (u-u_a)(x)+c\bigl(\mu(x)+\beta\bigr) \ge 0 \bigr\}.
\end{alignat*}
The five sets
$\mathcal{A}_{u,\mu,a}$,
$\mathcal{A}_{u,\mu,b}$,
$\mathcal{A}_{u,\mu,0}$,
$\mathcal{I}_{u,\mu,+}$,
$\mathcal{I}_{u,\mu,-}$
form a partition of $\mathcal{U}$ (up to a set of measure zero), with
$\mathcal{A}_{u,\mu} \cap \mathcal{I}_{u,\mu} = \emptyset$.

\subsubsection{Generalized Jacobians}

We define the operator that restricts to the state constraint active set 
$\mathcal{P}_{\mathcal{A}_y} : L^2(\Omega) \to~L^2(\mathcal{A}_y)$ 
by $\mathcal{P}_{\mathcal{A}_y}(v) := v|_{\mathcal{A}_y}$, whose 
adjoint
$\mathcal{P}_{\mathcal{A}_y}^\star : L^2(\mathcal{A}_y) \to L^2(\Omega)$
is the extension by zero operator.
Similarly, for the control constraint, let
$U_\mathcal{A} := L^2(\mathcal{A}_{u,\mu})$ and
$U_\mathcal{I} := L^2(\mathcal{I}_{u,\mu})$.  We define the
restriction operators
$\mathcal{P}_\mathcal{A} : U \to U_\mathcal{A}$,\;
$\mathcal{P}_\mathcal{I} : U \to U_\mathcal{I}$
by $\mathcal{P}_\mathcal{A}(v) := v|_{\mathcal{A}_{u,\mu}}$ and
$\mathcal{P}_\mathcal{I}(v) := v|_{\mathcal{I}_{u,\mu}}$,
respectively.  Their adjoints
$\mathcal{P}_\mathcal{A}^\star : U_\mathcal{A} \to U$ and
$\mathcal{P}_\mathcal{I}^\star : U_\mathcal{I} \to U$
are the extension by zero operators.  Since $\mathcal{U} =\mathcal{A}_{u,\mu}\,  \dot{\cup} \, \mathcal{I}_{u,\mu} $, it holds that
\begin{align}\label{eq:PA_PI_identity}
  \mathcal{P}_\mathcal{A}^\star \mathcal{P}_\mathcal{A}
  + \mathcal{P}_\mathcal{I}^\star \mathcal{P}_\mathcal{I}
  = I_U.
\end{align}
Elements of the generalized Jacobian of $\Phi_\gamma$ and $\Psi$ are
given by
\begin{align}\label{eq:gen_jacobians}
  G_{\Phi_\gamma}(y)\,\delta y
    &= \gamma^{-1}\,\mathcal{P}_{\mathcal{A}_y}^\star\,
       \mathcal{P}_{\mathcal{A}_y}\,\delta y, \\[3pt]
  G_\Psi(u,\mu)\,[\delta u,\delta\mu]
    &= \mathcal{P}_\mathcal{A}^\star\,\mathcal{P}_\mathcal{A}\,\delta u
     - c\,\mathcal{P}_\mathcal{I}^\star\,\mathcal{P}_\mathcal{I}\,\delta\mu. \notag
\end{align}

\begin{remark}\label{rem:proj_characteristic}
In each case, the composition of a restriction operator with its
adjoint acts as multiplication by the corresponding characteristic
function.  For instance,
$\mathcal{P}_{\mathcal{A}_y}^\star\,\mathcal{P}_{\mathcal{A}_y}
 = \chi_{\mathcal{A}_y}\,\cdot$
on $L^2(\Omega)$, and
$\mathcal{P}_\mathcal{A}^\star\,\mathcal{P}_\mathcal{A}
 = \chi_{\mathcal{A}_{u,\mu}}\,\cdot$,\;
$\mathcal{P}_\mathcal{I}^\star\,\mathcal{P}_\mathcal{I}
 = \chi_{\mathcal{I}_{u,\mu}}\,\cdot$
on $U$.

\end{remark}

%----------------------------------------------------------------------
\subsection{Semi-smooth Newton linearization}\label{ss:newton}
%----------------------------------------------------------------------

At each Newton iteration, we solve the linearized system
$G_\mathcal{F}(x^k)\,\delta x^k
 = -\mathcal{F}(x^k)$
for the update
$\delta x^k
 = (\delta y^k,\delta p^k,\delta u^k,\delta\mu^k) \in X$,
then set $x^{k+1} = x^k + \delta x^k$.
Writing
$L^k := M^\star
        (I + \gamma^{-1}\,\mathcal{P}_{\mathcal{A}_{y^k}}^\star\,
        \mathcal{P}_{\mathcal{A}_{y^k}})M : V \to V^\star$
for the active set-dependent operator and denoting the residual
components by
$\mathcal{F}_i^k := [\mathcal{F}(x^k)]_i$ for
$i = 1,\dots,4$, the linearized system reads: find
$(\delta y^k, \delta p^k, \delta u^k, \delta\mu^k) \in Z$ such that
\begin{subequations}\label{eq:newton_raw}
\begin{align}
  L^k\,\delta y^k
    + \widetilde{A}^\star\,\delta p^k
    &= -\mathcal{F}_1^k, \quad \text{in } V^\star,
  \label{eq:newton_raw_a} \\
  \widetilde{A}\,\delta y^k - B\,\delta u^k
    &= -\mathcal{F}_2^k, \quad \text{in } V^\star,
  \label{eq:newton_raw_b} \\
  -B^\star\,\delta p^k + \alpha\,\delta u^k + \delta\mu^k
    &= -\mathcal{F}_3^k, \quad \text{in } U,
  \label{eq:newton_raw_c} \\
  \mathcal{P}_\mathcal{A}^\star\,\mathcal{P}_\mathcal{A}\,\delta u^k
    - c\,\mathcal{P}_\mathcal{I}^\star\,\mathcal{P}_\mathcal{I}\,\delta\mu^k
    &= -\mathcal{F}_4^k,  \quad \text{in } U,
  \label{eq:newton_raw_d}
\end{align}
\end{subequations}
with residual components
\begin{subequations} \label{eq:residual_components}
\begin{align}
 \hspace{15mm} \mathcal{F}_1^k
    &:= \widetilde{A}^\star p^k
       + M^\star\del[1]{M y^k - y_d
         + \Phi_\gamma(M y^k)}, \label{eq:residual_components_1} \\
  \mathcal{F}_2^k
    &:= \widetilde{A} y^k - B u^k, \label{eq:residual_components_2}\\
  \mathcal{F}_3^k
    &:= \alpha u^k - B^\star p^k + \mu^k,  \label{eq:residual_components_3}\\
  \mathcal{F}_4^k
    &:= \Psi(u^k,\mu^k). \label{eq:residual_components_4}
\end{align}
\end{subequations}
Although the block operator in
\eqref{eq:newton_raw_a}--\eqref{eq:newton_raw_d} is not self-adjoint in
general, the complementarity equation can be used to eliminate the
inactive multiplier and recast the Newton step as a self-adjoint
indefinite system.
%----------------------------------------------------------------------
\subsection{Symmetrized system}\label{ss:symmetrized}
%----------------------------------------------------------------------

For the remainder of this section and the rest of the paper, we
simplify the notation for the linearized
system~\eqref{eq:newton_raw}.  We drop the prefix~$\delta$ and the
iteration index~$k$ from the Newton increments, writing
$(y,p,u,\mu)$ in place of
$(\delta y^k, \delta p^k, \delta u^k, \delta\mu^k)$.
The previous iterate at which the system is linearized is denoted by
$\bar{x} = (\bar{y},\bar{p},\bar{u},\bar{\mu})$, 
the $i^\text{th}$ component of the residual evaluated at this iterate by
$\bar{\mathcal{F}}_i$, and the active and inactive sets evaluated at
this iterate by
$\mathcal{A}_{\bar{y}}$,
$\mathcal{A}_{\bar{u},\bar{\mu}}$, and
$\mathcal{I}_{\bar{u},\bar{\mu}}$.

The fourth equation~\eqref{eq:newton_raw_d} couples $u$ and $\mu$
through the active and inactive sets in a pointwise
fashion.  Restricting~\eqref{eq:newton_raw_d} to the active set
determines $u|_{\mathcal{A}_{\bar{u},\bar{\mu}}}$, and restricting
to the inactive set determines
$\mu|_{\mathcal{I}_{\bar{u},\bar{\mu}}}$.  Eliminating the inactive
multiplier from the third equation yields a symmetric reduced system
whose fourth unknown is the active multiplier
$\mu_\mathcal{A} := \mathcal{P}_\mathcal{A}\,\mu
 \in U_\mathcal{A}$. 
 
We denote the product space
$X := V \times V \times U \times U_{\mathcal{A}}$
and write a generic element of $X$ as $x := (y,p,u,\mu_\mathcal{A})$,
where $y$ is the state, $p$ the adjoint, $u$ the control, and
$\mu_\mathcal{A}$ the multiplier restricted to the active control set.

\begin{proposition}[Symmetrized system]\label{prop:symmetrized}
Denoting
$\mu_\mathcal{A}:=\mathcal{P}_\mathcal{A}\mu$, the linearized KKT
system~\eqref{eq:newton_raw_a}--\eqref{eq:newton_raw_d} is
equivalent to the following symmetrized system: find
$(y,p,u,\mu_\mathcal{A})
 \in X
 := V \times V \times U \times U_\mathcal{A}$
such that
\begin{subequations}\label{eq:symmetrized_KKT}
\begin{align}
  L y + \widetilde{A}^\star\,p
    &= -\bar{\mathcal{F}}_1,
  \label{eq:symmetrized_KKT_a} \\
  \widetilde{A} y - B u
    &= -\bar{\mathcal{F}}_2,
  \label{eq:symmetrized_KKT_b} \\
  -B^\star p + \alpha u
    + \mathcal{P}_\mathcal{A}^\star \mu_\mathcal{A}
    &= -\bar{\mathcal{F}}_3
       -c^{-1}\mathcal{P}_\mathcal{I}^\star
         \mathcal{P}_\mathcal{I} \bar{\mathcal{F}}_4,
  \label{eq:symmetrized_KKT_c} \\
  \mathcal{P}_\mathcal{A}u
    &= -\mathcal{P}_\mathcal{A}\bar{\mathcal{F}}_4.
  \label{eq:symmetrized_KKT_d}
\end{align}
\end{subequations}
The full multiplier is recovered from $\mu_\mathcal{A}$ by
\begin{align}\label{eq:mu_recovery}
  \mu
  = \mathcal{P}_\mathcal{A}^\star\,\mu_\mathcal{A}
    +c^{-1}\mathcal{P}_\mathcal{I}^\star\,
      \mathcal{P}_\mathcal{I}\,\bar{\mathcal{F}}_4.
\end{align}
\end{proposition}
\begin{proof}
Since~\eqref{eq:newton_raw_d} holds pointwise a.e.\ in $\mathcal{U}$,
we may restrict it separately to $\mathcal{A}_{\bar{u},\bar{\mu}}$
and $\mathcal{I}_{\bar{u},\bar{\mu}}$.
On the active set, $\chi_{\mathcal{I}} = 0$, so
$\mathcal{P}_\mathcal{A}\,u = -\mathcal{P}_\mathcal{A}\,
\bar{\mathcal{F}}_4$,
which is~\eqref{eq:symmetrized_KKT_d}.
On the inactive set, $\chi_{\mathcal{A}} = 0$, so
$-c\,\mathcal{P}_\mathcal{I}\,\mu
 = -\mathcal{P}_\mathcal{I}\,\bar{\mathcal{F}}_4$,
i.e.,
$\mathcal{P}_\mathcal{I}\,\mu
 = c^{-1}\mathcal{P}_\mathcal{I}\,\bar{\mathcal{F}}_4$.
Using~\eqref{eq:PA_PI_identity}, we decompose
\begin{align*}
  \mu
  &= \mathcal{P}_\mathcal{A}^\star\,\mathcal{P}_\mathcal{A}\,\mu
   + \mathcal{P}_\mathcal{I}^\star\,\mathcal{P}_\mathcal{I}\,\mu
  = \mathcal{P}_\mathcal{A}^\star\,\mu_\mathcal{A}
   + c^{-1}\mathcal{P}_\mathcal{I}^\star\,
     \mathcal{P}_\mathcal{I}\,\bar{\mathcal{F}}_4.
\end{align*}
Substituting into~\eqref{eq:newton_raw_c} and rearranging
yields~\eqref{eq:symmetrized_KKT_c}.
\end{proof}
We write the symmetrized
system~\eqref{eq:symmetrized_KKT_a}--\eqref{eq:symmetrized_KKT_d}
in block operator form as
\begin{align}\label{eq:block_KKT}
  \underbrace{\begin{bmatrix}
    L & \widetilde{A}^\star & 0 & 0 \\
    \widetilde{A} & 0 & -B & 0 \\
    0 & -B^\star & \alpha\,I_U
      & \mathcal{P}_\mathcal{A}^\star \\
    0 & 0 & \mathcal{P}_\mathcal{A} & 0
  \end{bmatrix}}_{\mathcal{K}}
  \underbrace{\begin{bmatrix} y \\ p \\ u \\ \mu_\mathcal{A} \end{bmatrix}}_{x}
  =
  \underbrace{\begin{bmatrix}
    -\bar{\mathcal{F}}_1 \\[2pt]
    -\bar{\mathcal{F}}_2 \\[2pt]
    -\bar{\mathcal{F}}_3
      -c^{-1}\mathcal{P}_\mathcal{I}^\star\,
        \mathcal{P}_\mathcal{I}\,\bar{\mathcal{F}}_4 \\[2pt]
    -\mathcal{P}_\mathcal{A}\,\bar{\mathcal{F}}_4
  \end{bmatrix}}_{f}.
\end{align}
System~\eqref{eq:block_KKT} is the point of departure for the domain
decomposition framework developed in the subsequent sections.  The
operator $\mathcal{K} : X \to X^\star$ is
self-adjoint (with respect to the  duality pairing on $X$) but indefinite.

\begin{remark}[Weak formulation]
    Note that the linearized KKT system \eqref{eq:block_KKT} can be written in weak form as follows: find $(y,p,u,\mu_{\mathcal{A}}) \in V \times V \times U \times U_\mathcal{A}$ such that
\overfullrule=0pt    %
\begin{align*}
  \int_\Omega y\,v\,\mathrm{d}x
  + \frac{1}{\gamma}\int_{\mathcal{A}_{\bar y}} y\,v\,\mathrm{d}x
  + \int_\Omega \nabla p\cdot\nabla v\,\mathrm{d}x
  + \sigma\int_\Omega p\,v\,\mathrm{d}x
  &= -\langle\bar{\mathcal{F}}_1, v\rangle_{V},
  \\
  \int_\Omega \nabla y\cdot\nabla q\,\mathrm{d}x
  + \sigma\int_\Omega y\,q\,\mathrm{d}x
  - \int_{\mathcal{U}} u\,q\,\mathrm{d}\sigma
  &= -\langle\bar{\mathcal{F}}_2, q\rangle_{V},
  \\
  -\int_{\mathcal{U}} p\,w\,\mathrm{d}\sigma
  + \alpha\int_{\mathcal{U}} u\,w\,\mathrm{d}\sigma
  + \int_{\mathcal{A}_{\bar u,\bar\mu}} \mu_{\mathcal{A}}\,w\,\mathrm{d}\sigma
  &= -\int_{\mathcal{U}} \bar{\mathcal{F}}_3\,w\,\mathrm{d}\sigma
     - \frac{1}{c}\int_{\mathcal{I}_{\bar u,\bar\mu}} \bar{\mathcal{F}}_4\,w\,\mathrm{d}\sigma,
  \\
  \int_{\mathcal{A}_{\bar u,\bar\mu}} u\,\eta\,\mathrm{d}\sigma
  &= -\int_{\mathcal{A}_{\bar u,\bar\mu}} \bar{\mathcal{F}}_4\,\eta\,\mathrm{d}\sigma,
\end{align*}
for all $(v,q,w,\eta)\in V\times V\times U\times U_{\mathcal{A}}$.
  \end{remark}

\begin{theorem}[Well-posedness of linearized KKT problems]
\label{thm:global_well_posed}
The global KKT operator $\mathcal K:X\to X^\star$ is Fredholm with
$\operatorname{ind}(\mathcal K)=0$ for every $\sigma\ge 0$. Moreover:
\begin{enumerate}[label=\textup{(\alph*)},leftmargin=2.4em]
\item In case \textup{(DC)}, $\mathcal K$ is an isomorphism for every
$\sigma\ge 0$.

\item In case \textup{(NC)}, the following alternatives hold:
  \begin{enumerate}[label=\textup{(\roman*)},leftmargin=2.6em]
  \item If $\sigma>0$, or if $\sigma=0$ and
  $|\mathcal I_{\bar u,\bar\mu}| > 0$, then $\mathcal K$ is an
  isomorphism.

  \item If $\sigma=0$ and
  $|\mathcal I_{\bar u,\bar\mu}| = 0$, then
  \begin{align*}
    \ker(\mathcal K)
    =\cbr[1]{
      c\cdot\del{0,1|_\Omega,0,1|_{\mathcal A_{\bar u,\bar\mu}}}
      :c\in\mathbb R
    }.
  \end{align*}
  \end{enumerate}
\end{enumerate}
\end{theorem}
\begin{proof} 
  \emph{Step~1: $\mathcal{K}$ is Fredholm of index zero.}
Note that $\alpha\,I_{U}$ is an isomorphism and
$\mathcal{P}_{\mathcal{A}_u}
 (\alpha\,I_{U})^{-1}
 \mathcal{P}_{\mathcal{A}_u}^\star
 = \alpha^{-1}I_{U_{\mathcal{A}}}$.
The operator $L$ is compact since
$M : V \hookrightarrow L^2(\Omega)$ is compact by
Rellich--Kondrachov and $\mathcal{P}_{\mathcal{A}_{\bar{y}}}$ is bounded.
Likewise, $B$ is compact: in case (DC),
$B = M^\star$, while in case (NC), $B$ is the
adjoint of the compact trace embedding into $L^2(\mathcal{U})$. Since $\widetilde{A} = A + \sigma M^\star M $ with
%with $A_i$ the Laplacian which satisfies 
$\ker(A) \subseteq \operatorname{span}\{1\}$ and
$\operatorname{ran}(A)
 = \{f \in V^\star : \langle f, 1 \rangle_{V} = 0\}$,
 $\widetilde{A}$ is Fredholm with $\operatorname{ind}(\widetilde{A}) = 0$.
The fact that $\mathcal{K}$ is Fredholm with
$\operatorname{ind}(\mathcal{K}) = 0$, follows from \Cref{prop:KKT_fredholm} with
$\mathcal{X}_1 = \mathcal{X}_2 = V$,
$\mathcal{X}_3 = U$,
$\mathcal{X}_4 = U_{\mathcal{A}}$,
$\mathsf{A} = \widetilde{A}$,
$\mathsf{L} = L$,
$\mathsf{B} = B$,
$\mathsf{D} = \alpha\,I_{U}$, and
$\mathsf{P} = \mathcal{P}_{\mathcal{A}}$.

\medskip
\emph{Step~2: Kernel of $\mathcal{K}$.}
Suppose $(y, p, u, \mu_{\mathcal{A}})
 \in \ker(\mathcal{K})$, so that
 \begin{subequations}\label{eq:K_kernel}
\begin{align}
  L y
    + \widetilde{A}^\star p &= 0, \quad \text{ in } V^\star,
    \label{eq:K_kernel_a} \\
   {\widetilde{A}}
y
    - Bu &= 0, \quad \text{ in } V^\star,
    \label{eq:K_kernel_b} \\
  -B^\star p
    + \alpha u
    + \mathcal{P}_{\mathcal{A}}^\star \mu_{\mathcal{A}} &= 0, \quad \text{ in } U,
    \label{eq:K_kernel_c} \\
  \mathcal{P}_{\mathcal{A}} u &= 0, \quad \text{ in } U_{\mathcal{A}}.
    \label{eq:K_kernel_d}
\end{align}
\end{subequations}
Testing \eqref{eq:K_kernel_a}--\eqref{eq:K_kernel_c} against $(y,-p,u) \in V \times V \times U$ and summing yields
\begin{align}\label{eq:energy_Ki}
  \alpha\,\|u\|_{U}^2
  + \|M\,y\|_{L^2(\Omega)}^2
  + \gamma^{-1}\|M\,y\|_{L^2(\mathcal{A}_{\bar{y}})}^2
  = 0,
\end{align}
so $u = 0$ and $M y = 0$. Since $M y = 0$ in $L^2(\Omega)$, we have $y = 0$ in $V$. 
Next, \eqref{eq:K_kernel_b} gives
$\widetilde{A} y = B u = 0$, and \eqref{eq:K_kernel_c}
gives $\mathcal{P}_{\mathcal{A}_u}^\star \mu_{\mathcal{A}} = B^\star p $. Applying the operator $\mathcal{P}_{\mathcal{A}}$  to both sides, we find $\mu_{\mathcal{A}}  = \mathcal{P}_{\mathcal{A}} B^\star p$. Therefore,
\begin{align*}
    (y,p,u,\mu) \in \text{ker}(\mathcal{K}) \quad \Longrightarrow \quad (y, u, \mu) = (0, 0,  \mathcal{P}_{\mathcal{A}_u} B^\star p).
\end{align*}
Thus, $\text{ker}(\mathcal{K})$ is determined by solutions of the homogeneous  equation
$\widetilde{A}^\star p = 0$. 

\medskip
\textbf{Case (DC).} $\widetilde{A}$ is an isomorphism for all $\sigma \ge 0$ owing to the Poincar\'{e} inequality and Lax--Milgram theorem, which proves (a). 

\medskip
\textbf{Case (NC).}
In case (NC), the situation is more complex. If $\sigma > 0$, again $\widetilde{A}$ is an isomorphism by Lax--Milgram. However, if $\sigma = 0$, $\widetilde{A} = A$ is the  Neumann Laplacian with
$\ker(A) = \operatorname{span}\{1|_{\Omega}\}$ and
$\operatorname{ran}(A)
 = \{f \in V^\star : \langle f, 1|_{\Omega} \rangle_{V} = 0\}$. As a result, $p = c \cdot 1|_{\Omega}$ for $c \in \mathbb{R}$, and $\operatorname{ker}(\mathcal{K})$ depends on the structure of the active control set. If $|\mathcal I_{\bar u,\bar\mu}| > 0$, we can apply $\mathcal{P}_{\mathcal{I}}$ to both sides of \eqref{eq:K_kernel_c} to find
\begin{align*}
 0 =  \mathcal{P}_{\mathcal{I}} B^\star p = c \cdot \mathcal{P}_{\mathcal{I}}(1|_{\mathcal{U}}) = c \cdot 1|_{\mathcal{I}},
\end{align*}
from which it follows that $c = 0$.
In each of the aforementioned cases, $(p,\mu) = (0,0)$ and it follows that $\mathcal{K}$ is an isomorphism. On the other hand, if
$\mathcal I_{\bar u,\bar\mu}=\emptyset$,  then $B^\star p = c \cdot 1|_{\mathcal{U}}$, and $\mu_{\mathcal{A}} = c \cdot 1|_{\mathcal{A}_{\bar u,\bar\mu}}$ for arbitrary $c \in \mathbb{R}$, from which (b) follows. 
\end{proof}

%======================================================================
\section{Domain decomposition}\label{sec:dd}
%======================================================================

A natural structure that can be exploited in PDE-constrained
optimization is the geometry of the underlying physical domain.
Domain decomposition methods partition \(\Omega \subset \mathbb{R}^d\)
into subdomains and reformulate the global problem as local
subdomain problems coupled through interface conditions. At the
interface level, this leads to a generalized Steklov--Poincar\'e
operator
\cite[Section 1.1]{quarteroni-valli-1999-domain-decomposition}: given
state and adjoint traces on the subdomain interfaces, one solves
local linearized KKT systems in the interiors and obtains the
corresponding interface residuals. This approach is due to Heinkenschloss and
Nguyen~\cite{heinkenschloss2006neumann} in the linear-quadratic
setting; we extend it to problems with control constraints,
$L^1$-sparsity, and regularized state constraints. The local interior
problems are well posed by \Cref{thm:KII_fredholm}, and
\Cref{thm:equivalence} identifies the global KKT system with local
subdomain problems coupled by trace conformity and interface flux
balance. 

We first fix notation for products of Hilbert spaces, their duals, and block-diagonal operators between them, all of which are used throughout this section and the remainder of the article.

\begin{remark}[Products, duals, and block operators]\label{rem:notation}
Let $X_1,\dots,X_N$ be Hilbert spaces and equip $\widetilde{X} := \prod_{i=1}^N X_i$ with
the norm $\|x\|_{\widetilde{X}}^2 := \sum_{i=1}^N \|x_i\|_{X_i}^2$. We denote by
$\widetilde{\pi}_i : \widetilde{X} \to X_i$ and $\widetilde{\jmath}_i : X_i \to \widetilde{X}$ the coordinate projections and
injections, which satisfy
\begin{align} \label{eq:resolution}
    \sum_{i=1}^N \widetilde{\jmath}_i \widetilde{\pi}_i = I_{\widetilde{X}}, \qquad \widetilde{\pi}_j \widetilde{\jmath}_i = \delta_{ij} I_{X_i}.
\end{align}
We identify $\widetilde{X}^\star$ with $\prod_{i=1}^N X_i^\star$ via the isometric
isomorphism $\Phi : \prod_{i=1}^N X_i^\star \to \widetilde{X}^\star$ defined by
$\langle \Phi(\ell_1,\dots,\ell_N), x \rangle_{\widetilde{X}}
 := \sum_{i=1}^N \langle \ell_i, \widetilde \pi_i x \rangle_{X_i}$, whose inverse is
$\Phi^{-1} f = (\widetilde\jmath_1^{\, \star} f, \dots, \widetilde \jmath_N^{\, \star} f)$
(\emph{cf.} \cite[Section~1.10]{Megginson1998}). Given bounded linear operators
$T_i : X_i \to X_i^\star$, we write
\begin{align} \label{eq:block_diag_sum}
    \bigoplus_{i=1}^N T_i := \sum_{i=1}^N \widetilde \pi_i^\star T_i \widetilde\pi_i : \widetilde{X} \to \widetilde{X}^\star.
\end{align}
This operator is an isomorphism if and only if every $T_i$ is, in which case
\begin{align} \label{eq:block_diag_inv}
    \bigg(\bigoplus_{i=1}^N T_i\bigg)^{-1} = \sum_{i=1}^N \widetilde \jmath_i T_i^{-1} \widetilde \jmath_i^\star,
\end{align}
as follows from \eqref{eq:resolution} by direct computation.
\end{remark}

\subsection{Global interior--interface decomposition}
\label{subsec:global_decomposition}
Suppose that
$\overline{\Omega} = \bigcup_{i=1}^N \overline{\Omega}_i$, where the
subdomains $\Omega_i$ are open, connected, and mutually disjoint.  We
denote the subdomain interface by
$\Gamma := \bigl(\bigcup_{i=1}^N \partial\Omega_i\bigr)
           \setminus \partial\Omega$
and write $\Gamma_i := \partial\Omega_i \cap \Gamma$ for the portion
of the interface belonging to~$\Omega_i$. Let $\operatorname{tr}_\Gamma: V \to \prod_{i=1}^N H^{\frac{1}{2}}(\Gamma_i)$ denote the interface trace operator and define the interface trace space
\begin{align}
    V_\Gamma := \operatorname{ran}(\operatorname{tr}_\Gamma), \quad \|v_\Gamma\|_{V_\Gamma} := \inf{\cbr{\|v\|_{V} \, : \, \operatorname{tr}_\Gamma(v) = v_\Gamma}},
\end{align}
and the subspace 
$V_0 := \operatorname{ker}(\operatorname{tr}_\Gamma)$ of $V$
with vanishing trace on the interface. We equip $V_0$ with the $H^1(\Omega)$ semi-norm, which is a norm by the Poincar\'e--Friedrichs inequality.
In general, controls and active multipliers have no interface trace component due to a lack of regularity. Thus, we define the product interface trace space corresponding to the state and adjoint traces:
\begin{align} \label{eq:def_global_interface_trace_space}
    X_\Gamma := V_\Gamma \times V_\Gamma.
\end{align}
We define the following component-wise trace operator on the product space:
\begin{align}
   \tau :X \to X_\Gamma, \quad \tau: (y,p,u,\mu_{\mathcal{A}}) \mapsto (\operatorname{tr}_\Gamma(y),\operatorname{tr}_\Gamma(p)),
\end{align}
and the subspace of tuples in $X$ whose state and adjoint components vanish on $\Gamma$:
\begin{align} \label{eq:def_global_interior_space}
    X_I := \operatorname{ker}(\tau) = V_0 \times V_0 \times U \times U_{\mathcal{A}}.
\end{align}
We denote by $\iota_I : X_I \hookrightarrow X$ the canonical injection of $X_I$ into $X$. We will henceforth refer to $X_I$ as the global interior space and $X_\Gamma$ as the global interface space.

Owing to the surjectivity of the trace operator $\tau: X \to X_\Gamma$, there exists a (non-unique) bounded lifting $E : X_\Gamma \to X$ satisfying $\tau  E = I_{X_\Gamma}$. Moreover, any such choice of $E$ induces a decomposition of $X$ into the topological direct sum:
\begin{align} \label{eq:X_decomp_E}
    X = X_I \oplus E(X_\Gamma).
\end{align}
In other words, any $x \in X$ can be decomposed as $x = \iota_I x_I + Ex_\Gamma$, for some $(x_I,x_\Gamma) \in X_I \times X_\Gamma$. This decomposition is unique only up to the choice of lifting. 
Note that the mapping $J_E: X_I \times X_\Gamma \to X_I \oplus E(X_\Gamma)$ defined by $J_E(x_I,x_\Gamma) := \iota_I x_I + Ex_\Gamma$ is an isomorphism, and moreover that
\begin{align}
    \Phi^{-1} J_E^\star \mathcal{K} J_E (x_I,x_\Gamma) = \begin{bmatrix}
        \iota_I^\star \mathcal{K} \iota_I & \iota_I^\star \mathcal{K} E \\
        E^\star\mathcal{K} \iota_I & E^\star\mathcal{K} E
    \end{bmatrix} \begin{bmatrix}
        x_I \\
        x_\Gamma
    \end{bmatrix},
\end{align}
where $\Phi : X_I^\star \times X_\Gamma^\star \to (X_I \times X_\Gamma)^\star$ is the canonical identification (\emph{cf}. \Cref{rem:notation}).
For notational brevity, we define the operators $\mathcal{K}_{II}: X_I \to X_I^\star$, $\mathcal{K}_{I\Gamma}^E: X_\Gamma \to X_I^\star$, $\mathcal{K}_{\Gamma I}^E : X_I \to X_\Gamma^\star$, and $\mathcal{K}_{\Gamma \Gamma}^E: X_\Gamma \to X_\Gamma^\star$ via
\begin{align} \label{eq:KKT_blocks_definition}
    \mathcal{K}_{II} := \iota_I^\star \mathcal{K} \iota_I, \quad \mathcal{K}_{I\Gamma}^E := \iota_I^\star \mathcal{K} E, \quad \mathcal{K}_{\Gamma I}^E := E^\star\mathcal{K} \iota_I, \quad \mathcal{K}_{\Gamma \Gamma}^E := E^\star\mathcal{K} E.
\end{align}
Note that the operator $\mathcal{K}_{II}$, which we will refer to as the \emph{global interior KKT operator}, is independent of the choice of lifting. Moreover, since the symmetrized KKT operator is self-adjoint, $(\mathcal{K}_{I\Gamma}^E)^\star = (\mathcal{K}_{\Gamma I}^E)$.
The following result shows that once a lifting has been fixed, the symmetrized KKT system \eqref{eq:block_KKT}  can be recast 
\begin{lemma} \label{lem:equivalent_block_product_form}
  Let $E:X_\Gamma \to X$ be a fixed lifting of the trace operator $\tau: X \to X_\Gamma$. Given $x \in X$, let $x = x_I^E + Ex_\Gamma$ be its unique decomposition corresponding to \eqref{eq:X_decomp_E}. Then, $x \in X$ satisfies the symmetrized KKT system \eqref{eq:block_KKT} if and only if $(x_I^E, x_\Gamma) \in X_I \times X_\Gamma$ satisfy the block operator equation
  \begin{align} \label{eq:KKT_block_form_E}
  \begin{bmatrix}
        \mathcal{K}_{II} & \mathcal{K}_{I\Gamma}^E \\
        \mathcal{K}_{\Gamma I}^E  & \mathcal{K}_{\Gamma \Gamma}^E
    \end{bmatrix} 
    \begin{bmatrix}
        x_I^E \\
        x_\Gamma
    \end{bmatrix} = \begin{bmatrix}
        f_I\\
        f_\Gamma^E
    \end{bmatrix}, \quad \text{in } X_I^\star \times X_\Gamma^\star,
\end{align}
where $f_I := \iota_I^\star f$ and $f_\Gamma^E := E^\star f$. 
\end{lemma}
\begin{proof}
  Suppose $x \in X$ solves the symmetrized KKT system \eqref{eq:block_KKT}, and let $(x_I^E,x_\Gamma) \in X_I \times X_\Gamma$ be the unique pair such that $x = J_E(x_I^E,x_\Gamma)$ guaranteed by the decomposition \eqref{eq:X_decomp_E}. Then,
  \begin{align} \label{eq:KKT_block_E.1}
      \mathcal{K}  J_E(x_I^E,x_\Gamma) = f.
  \end{align}
 Applying the operator $\Phi^{-1} J_E^\star$ to  both sides of \eqref{eq:KKT_block_E.1}   yields \eqref{eq:KKT_block_form_E}.
Conversely, suppose the pair $(x_I^E,x_\Gamma) \in X_I \times X_\Gamma$ solves \eqref{eq:KKT_block_form_E}. Equivalently,
\begin{align*}
   \Phi^{-1} J_E^\star \del{ \mathcal{K} J_E (x_I^E,x_\Gamma) - f} = 0,
\end{align*}
which implies that $ \mathcal{K} J_E (x_I^E,x_\Gamma) - f = 0$ since $\Phi^{-1} : (X_I \times X_\Gamma)^\star \to X_I^\star \times X_\Gamma^\star$ and $J_E^\star: X^\star \to (X_I \times X_\Gamma)^\star$ are isomorphisms.
Therefore, $x := J_E (x_I^E,x_\Gamma)$ solves \eqref{eq:block_KKT}.
\end{proof}
\subsection{Global Schur complement problem}
As a consequence of \Cref{lem:equivalent_block_product_form}, given a solution of $(x_I^E,x_\Gamma) \in X_I \times X_\Gamma$ of \eqref{eq:KKT_block_form_E}, a solution $x \in X$ of \eqref{eq:block_KKT} can be recovered by setting $x:= J_E(x_I^E,x_\Gamma)$. Thus, we turn our attention to solving \eqref{eq:KKT_block_form_E}, which can be reduced to a single equation for $x_\Gamma \in X_\Gamma$ via static condensation. 
More precisely, assuming $\mathcal{K}_{II}: X_I \to X_I^\star$ is an isomorphism, the first row of \eqref{eq:KKT_block_form_E} can be rearranged to find
\begin{align} \label{eq:xI_solve}
    x_I^E = \mathcal{K}_{II}^{-1} \del[1]{ f_I - \mathcal{K}_{I\Gamma}^E x_\Gamma}.
\end{align}
Substituting \eqref{eq:xI_solve} into the second row of \eqref{eq:KKT_block_form_E} and rearranging yields
\begin{align} \label{eq:Schur_comp_system_E}
   \del[1]{ \mathcal{K}_{\Gamma \Gamma}^E - \mathcal{K}_{\Gamma I}^E \mathcal{K}_{II}^{-1} \mathcal{K}_{I \Gamma}^E} x_\Gamma = g_\Gamma^E,
\end{align}
where we have defined the reduced right-hand side $g_\Gamma^E:= f_\Gamma^E - \mathcal{K}_{\Gamma I}^E\mathcal{K}_{II}^{-1} f_I \in X_\Gamma^\star$.
Thus, once \eqref{eq:Schur_comp_system_E} has been solved for $x_\Gamma \in X_\Gamma$, $x_I^E \in X_I$ can be recovered via back-substitution in \eqref{eq:xI_solve}.
\subsubsection{The Schur complement operator}
The considerations of the previous subsection motivate the \emph{Schur complement operator} $\mathcal{S}_{\Gamma}^E: X_\Gamma \to X_\Gamma^\star$, defined via the left-hand side of \eqref{eq:Schur_comp_system_E}:
\begin{align} \label{eq:Schur_complement_E}
    \mathcal{S}_{\Gamma}^E := \mathcal{K}_{\Gamma \Gamma}^E - \mathcal{K}_{\Gamma I}^E \mathcal{K}_{II}^{-1} \mathcal{K}_{I \Gamma}^E.
\end{align}
As the following result shows, the Schur complement operator is invariant to the choice of lifting $E$. We will therefore suppress the superscript in the remainder of the article and simply write $\mathcal{S}_\Gamma: X_\Gamma \to X_\Gamma^\star$ and $g_\Gamma \in X_\Gamma^\star$, so that the global Schur complement problem reads: find $x_\Gamma \in X_\Gamma$ satisfying
\begin{align}\label{eq:global_schur_eq}
\mathcal{S}_\Gamma\, x_\Gamma = g_\Gamma, \quad \text{in } X_\Gamma^\star.
\end{align}
\begin{lemma} \label{lem:schur_independent}
    The Schur complement operator $\mathcal{S}_\Gamma^E : X_\Gamma \to X_\Gamma^\star$ defined in \eqref{eq:Schur_complement_E}, as well as the reduced right-hand side $g_\Gamma^E$ of \eqref{eq:Schur_comp_system_E}, are independent of the choice of lifting $E: X_\Gamma \to X$.
\end{lemma}
\begin{proof}
    Let $E_1, E_2 : X_\Gamma \to X$ be any two bounded linear operators satisfying $\tau  E_1 = \tau  E_2 = I_{X_\Gamma}$. For notational brevity, we define an operator $\mathcal{K}_0 : X \to X^\star$ via $\mathcal{K}_0  := \mathcal{K}  -  \mathcal{K} \iota_I (\iota_I^\star \mathcal{K} \iota_I)^{-1} \iota_I^\star \mathcal{K}$ and a functional $f_0 \in X^\star$ via $f_0 := f - \mathcal{K} \iota_I (\iota_I^\star \mathcal{K} \iota_I)^{-1} \iota_I^\star f$ so that for $i=1,2$,
    \begin{align*}
        \mathcal{S}_\Gamma^{E_i} = E_i^\star \mathcal{K}_0  E_i, \quad 
        g_\Gamma^{E_i} = E_i^\star f_0.
    \end{align*}
    Observe that the following identities hold:
    \begin{align}
        \iota_I^\star \mathcal{K}_0 &= \iota_I^\star \mathcal{K}  -  \iota_I^\star \mathcal{K} \iota_I (\iota_I^\star \mathcal{K} \iota_I)^{-1} \iota_I^\star \mathcal{K} = \iota_I^\star \mathcal{K} - \iota_I^\star \mathcal{K} = 0,  \label{eq:K_0_property.1}\\
        \mathcal{K}_0 \iota_I &= \mathcal{K} \iota_I  -   \mathcal{K} \iota_I (\iota_I^\star \mathcal{K} \iota_I)^{-1} \iota_I^\star \mathcal{K} \iota_I = \mathcal{K} \iota_I - \mathcal{K} \iota_I = 0, \label{eq:K_0_property.2} \\
        \iota_I^\star f_0 &= \iota_I^\star f  - \iota_I^\star\mathcal{K} \iota_I (\iota_I^\star \mathcal{K} \iota_I)^{-1} \iota_I^\star f = \iota_I^\star f - \iota_I^\star f = 0. \label{eq:f_0_property}
    \end{align}
    Since $\text{ran}(E_1 - E_2) \subset \text{ker}(\tau) = X_I$, for each $x_\Gamma \in X_\Gamma$ there exists a unique $w_I \in X_I$ such that 
    $(E_1 - E_2) x_\Gamma = \iota_I w_I$.
    Consequently, \eqref{eq:K_0_property.1} yields, for all $x_\Gamma \in X_\Gamma$,
    \begin{align} \label{eq:K_0_left}
        \mathcal{K}_0 (E_1 - E_2) x_\Gamma = 0.
    \end{align}
    Moreover, for all $x \in X$ and $x_\Gamma \in X_\Gamma$, \eqref{eq:K_0_property.2} and \eqref{eq:f_0_property} yield
    \begin{align} \label{eq:K_0_right}
        \langle (E_1 - E_2)^\star \mathcal{K}_0 x, x_\Gamma \rangle_{X_\Gamma} &=  \langle \mathcal{K}_0 x, (E_1 - E_2)x_\Gamma \rangle_{X} = \langle \mathcal{K}_0 x, \iota_I w_I \rangle_{X} = 0, \\ \label{eq:f_0_left}
      \langle (E_1 - E_2)^\star f_0 , x_\Gamma \rangle_{X_\Gamma}  &=  \langle f_0 , (E_1 - E_2)x_\Gamma \rangle_{X} = \langle f_0 , \iota_I w_I \rangle_{X_\Gamma} = 0.
    \end{align}
    We deduce from \eqref{eq:K_0_left}--\eqref{eq:f_0_left} that
    \begin{align} \label{eq:K_0_properties.3,4}
    \mathcal{K}_0 (E_1 - E_2) = 0, \quad (E_1 - E_2)^\star \mathcal{K}_0 = 0, \text{ and } (E_1 - E_2)^\star f_0 = 0.
    \end{align}
    Subtracting $\mathcal{S}_\Gamma^{E_2}$ from $\mathcal{S}_\Gamma^{E_1}$, we therefore have
       \begin{align*}
    \mathcal{S}_\Gamma^{E_1} -  \mathcal{S}_\Gamma^{E_2} &= E_1^\star \mathcal{K}_0 E_1 - E_2^\star \mathcal{K}_0 E_2 
    = (E_1 - E_2)^\star \mathcal{K}_0 E_1 + E_2^\star \mathcal{K}_0 (E_1 - E_2) = 0.
    \end{align*}
    Similarly, subtracting $g_\Gamma^{E_2}$ from $g_\Gamma^{E_1}$,
    \begin{align*}
        g_\Gamma^{E_1} - g_\Gamma^{E_2} = (E_1 - E_2)^\star f_0 = 0.
    \end{align*}
    The proof is now complete.
\end{proof}

\color{black}

\subsubsection{Generalized Harmonic extension}
We next introduce a particular lifting $ \mathcal{H} : X_\Gamma \to X$ that has the desirable property of diagonalizing \eqref{eq:KKT_block_form_E}.
The choice $E:= \mathcal{H}$ yields what we will henceforth refer to as the canonical decomposition of $X$:
\begin{align}\label{eq:X_decomp_H}
X = X_I \oplus \mathcal{H}(X_\Gamma).
\end{align}
Note that the lifting constructed below is a  generalization of the familiar harmonic lifting from elliptic PDE theory (see, e.g., \cite{quarteroni-valli-1999-domain-decomposition}). An analogous operator was introduced in \cite{heinkenschloss2006neumann} for unconstrained linear-quadratic elliptic optimal control problems.
\begin{theorem}[Generalized harmonic extension] \label{thm:global_harmonic_extension}
 There exists an operator $\mathcal{H}: X_\Gamma \to X$, defined as as follows: given $x_\Gamma \in X_\Gamma$, $\mathcal{H}(x_\Gamma) \in X$ is the unique solution of the following system:
\begin{subequations}\label{eq:harmonic_extension}
\begin{align}
    \iota_I^\star \mathcal{K} \mathcal{H}(x_\Gamma) &= 0, \quad \text{in } X_I^\star, \label{eq:harmonic_extension_a}\\
    \tau \mathcal{H}(x_\Gamma) &= x_\Gamma. \label{eq:harmonic_extension_b}
\end{align}
\end{subequations}
We call the operator $\mathcal{H}$ the generalized harmonic extension of $x_\Gamma$. Moreover, the following closed-form expression holds (independently of the choice of lifting $E$):
\begin{align}\label{eq:H_closed_form}
    \mathcal{H} = (I_{X} - \iota_I \mathcal{K}_{II}^{-1} \iota_I^\star \mathcal{K})E.
\end{align}
In particular, $\mathcal{H}$ is linear and bounded.
\end{theorem}
\begin{proof}
  We begin by showing existence. To this end, we let $z \in X$ be any fixed element satisfying $\tau z = x_\Gamma$ and search for a solution of system \eqref{eq:harmonic_extension} of the form 
\begin{align}\label{eq:harmonic_homogeneous.1}
      \mathcal{H}(x_\Gamma) =  z + \iota_I x_I, \quad x_I \in X_I.
  \end{align}
 Clearly, $\mathcal{H}(x_\Gamma)$ satisfies \eqref{eq:harmonic_extension_b}. Substituting into \eqref{eq:harmonic_extension_a} and rearranging, we find that $x_I \in X_I$ must satisfy
\begin{align}\label{eq:harmonic_homogeneous.2}
      \iota_I^\star \mathcal{K} \iota_I x_I = - \iota_I^\star \mathcal{K} z.
  \end{align}
  Since $\mathcal{K}_{II} := \iota_I^\star \mathcal{K} \iota_I$, there exists a unique $x_I \in X_I$ satisfying \eqref{eq:harmonic_homogeneous.2} by \Cref{thm:KII_fredholm} below, ensuring the existence of at least one $\mathcal{H}(x_\Gamma) \in X$ of the form \eqref{eq:harmonic_homogeneous.1} satisfying \eqref{eq:harmonic_extension}.
  Next, we show that any solution to the system \eqref{eq:harmonic_extension} is unique. Indeed, let $x_1,x_2 \in X$ be any two solutions. Then, by \eqref{eq:harmonic_homogeneous.2}, $x_1 - x_2 \in \operatorname{ker}(\tau)$, so that $x_1 - x_2 = \iota_I w_I$ for some $w_I \in X_I$. It follows from \eqref{eq:harmonic_homogeneous.1} that
   $0 = \iota_I^\star \mathcal{K} \iota_I w_I = \mathcal{K}_{II} w_I$
  and thus $w_I = 0$ since $\mathcal{K}_{II}$ is injective.
Finally, the representation \eqref{eq:H_closed_form} of the harmonic extension can be verified by direct computation. Linearity and boundedness immediately follow.
\end{proof}
  \begin{remark}\label{rem:harmonic_extension}
      Note that the system \eqref{eq:harmonic_extension} is equivalent to the following weak formulation: given $x_\Gamma := (y_\Gamma,p_\Gamma)$, find $x:= (y,p,u,\mu) \in V \times V \times U \times U_\mathcal{A}$ such that $\operatorname{tr}_\Gamma y = y_\Gamma$, $\operatorname{tr}_\Gamma p = p_\Gamma$, and 
\begin{align*}
  \int_\Omega y\,v\,\mathrm{d}x
  + \frac{1}{\gamma}\int_{\mathcal{A}_{\bar y}} y\,v\,\mathrm{d}x
  + \int_\Omega \nabla p\cdot\nabla v\,\mathrm{d}x
  + \sigma\int_\Omega p\,v\,\mathrm{d}x
  &= 0,
  \\
  \int_\Omega \nabla y\cdot\nabla q\,\mathrm{d}x
  + \sigma\int_\Omega y\,q\,\mathrm{d}x
  - \int_{\mathcal{U}} u\,q\,\mathrm{d}\sigma
  &= 0,
  \\
  -\int_{\mathcal{U}} p\,w\,\mathrm{d}\sigma
  + \alpha\int_{\mathcal{U}} u\,w\,\mathrm{d}\sigma
  + \int_{\mathcal{A}_{\bar u,\bar\mu}} \mu_{\mathcal{A}}\,w\,\mathrm{d}\sigma
  &= 0, \\
 \int_{\mathcal{A}_{\bar u,\bar\mu}} u\,\eta\,\mathrm{d}\sigma &= 0,
\end{align*}
for all $(v,q,w,\eta)\in V_0\times V_0\times U\times U_{\mathcal{A}}$.
  \end{remark}
\begin{corollary}
\label{cor:diagonalized_block_product_form}
Suppose that $\mathcal{K}_{II}$ and $\mathcal{S}_\Gamma$ are isomorphisms.
  Given $x \in X$, let $x = x_I + \mathcal{H}(x_\Gamma)$ be its unique decomposition corresponding to \eqref{eq:X_decomp_H}. Then, $x \in X$ satisfies the symmetrized KKT system \eqref{eq:block_KKT} if and only if $(x_I, x_\Gamma) \in X_I \times X_\Gamma$ satisfy the decoupled block operator equation
  \begin{align} \label{eq:KKT_block_form_H}
  \begin{bmatrix}
        \mathcal{K}_{II} & 0\\
        0  & \mathcal{S}_\Gamma
    \end{bmatrix} 
    \begin{bmatrix}
        x_I \\
        x_\Gamma
    \end{bmatrix} = \begin{bmatrix}
        f_I\\
        g_\Gamma
    \end{bmatrix}, \quad \text{in } X_I^\star \times X_\Gamma^\star,
\end{align}
where $f_I := \iota_I^\star f$ and $g_\Gamma = \mathcal{H}^\star f$. Consequently, 
\begin{align} \label{eq:schur_comp_harmonic}
    \mathcal{S}_\Gamma = \mathcal{H}^\star \mathcal{K} \mathcal{H},
\end{align}
and the solution of \eqref{eq:block_KKT} can equivalently be expressed as
\begin{align}
    x = \iota_I \mathcal{K}_{II}^{-1} \iota_I^\star f + \mathcal{H}\mathcal{S}_\Gamma^{-1} \mathcal{H}^\star f.
\end{align}
\end{corollary}
\begin{proof}
By \eqref{eq:harmonic_extension_a},
$\mathcal K_{I\Gamma}^{\mathcal H}=0$ and, by self-adjointness,
$\mathcal K_{\Gamma I}^{\mathcal H}=0$. Thus the block representation
of \Cref{lem:equivalent_block_product_form} is diagonal, with
$\mathcal K_{\Gamma\Gamma}^{\mathcal H}
=\mathcal H^\star\mathcal K\mathcal H=\mathcal S_\Gamma$ and reduced
load $\mathcal H^\star f$. The reconstruction formula follows by
inverting the two diagonal blocks.
\end{proof}

\subsection{Local subdomain function spaces and operators}\label{ss:spaces}

To present a unified
analysis in the cases (DC) and (NC), we
follow~\cite{heinkenschloss2006neumann} and partition the subdomain indices as
$\{1,\dots,N\} = \mathcal{N}_c \cup \mathcal{N}_\circ$, where
$\mathcal{N}_c := \{i : \mathcal{U}_i \neq \emptyset\}$ denotes the
set of \emph{control subdomains} and
$\mathcal{N}_\circ := \{i : \mathcal{U}_i = \emptyset\}$ the set of
\emph{non-control subdomains}, with
$\mathcal{U}_i := \Omega_i$ (DC) or
$\mathcal{U}_i := \partial\Omega_i \cap \partial\Omega$ (NC).  In case (DC), every subdomain
is a control subdomain; in case (NC), both $\mathcal{N}_\circ$ and $\mathcal{N}_c$ are nonempty. For later use, we introduce the following definition to differentiate between subdomains that intersect the boundary $\partial \Omega$ and those that do not:
\begin{definition}[Anchored and floating subdomains]
A subdomain $\Omega_i$ is called \emph{anchored} if
$\partial\Omega_i \cap \partial\Omega \neq \emptyset$. Otherwise,
$\Omega_i$ is called \emph{floating}.
\end{definition}

\subsubsection{Local functional setting}
On each subdomain $\Omega_i$, we define $V_i:=\{v\in H^1(\Omega_i):v|_{\partial\Omega_i\cap\partial\Omega_D}=0\}$ and $V_0^i:=\cbr[0]{ v \in V_i \, : \, v|_{\Gamma_i} = 0}$, where $\partial\Omega_D=\partial\Omega$ in case (DC) and $\partial\Omega_D=\emptyset$ in case (NC). 
The space $V_0^i$ is equipped with the $H^1(\Omega_i)$ semi-norm, which is a norm by the Poincare--Friedrichs inequality.
We write $\Gamma_i := \partial\Omega_i \cap \Gamma$ for the portion
of the interface belonging to~$\Omega_i$. Let
$\operatorname{tr}_{\Gamma_i} : V_i \to H^{\frac{1}{2}}(\Gamma_i)$
denote the local trace operator and define the local interface trace
space
\begin{align} \label{eq:def_local_interface_trace_space}
    V_{\Gamma_i} := \operatorname{ran}(\operatorname{tr}_{\Gamma_i}),
    \quad
    \|\xi\|_{V_{\Gamma_i}} := \inf\cbr{\|v_i\|_{V_i} \, : \, \operatorname{tr}_{\Gamma_i}(v_i) = \xi}.
\end{align}
The local active state set is defined as $\mathcal{A}_{\bar y}^i:= \mathcal{A}_{\bar y} \cap \Omega_i$, with restriction operator $\mathcal{P}_{\mathcal{A}_{\bar y}^i}:= V_i\to L^2(\mathcal{A}_{\bar y}^i)$. The local control space is $U_i=L^2(\mathcal{U}_i)$ for $i\in\mathcal{N}_c$ and $U_i=\{0\}$ for $i\in\mathcal{N}_\circ$. We define the local active and inactive control sets as $\mathcal{A}_{\bar u,\bar\mu}^i: \mathcal{A}_{\bar u,\bar\mu} \cap \mathcal{U}_i$ and $\mathcal{I}_{\bar u,\bar\mu}^i:= \mathcal{I}_{\bar u,\bar\mu} \cap \mathcal{U}_i$, respectively. Note that for $i\in\mathcal{N}_\circ$, $\mathcal{A}_{\bar u,\bar\mu}^i = \mathcal{I}_{\bar u,\bar\mu}^i = \emptyset$. For $i\in\mathcal{N}_c$, we denote $U_{\mathcal{A}_i}:=L^2(\mathcal{A}_{\bar u,\bar\mu}^i)$ and $U_{\mathcal{I}_i}:=L^2(\mathcal{I}_{\bar u,\bar\mu}^i)$, with restriction operators $\mathcal{P}_{\mathcal{A}_u^i}:U_i\to U_{\mathcal{A}_i}$ and $\mathcal{P}_{\mathcal{I}_u^i}:U_i\to U_{\mathcal{I}_i}$. 

\subsubsection{Local operators} \label{sss:local_operators}
We define local operators as follows: $A_i:V_i\to V_i^\star$ is the local weak Laplacian, $M_i:V_i\hookrightarrow L^2(\Omega_i)$ is the canonical embedding, and $B_i:U_i\to V_i^\star$ is the local control operator, with $B_i=0$ for $i\in\mathcal{N}_\circ$. We also define $\widetilde A_i:=A_i+\sigma M_i^\star M_i$ and
$L_i:=M_i^\star\del[0]{I+\gamma^{-1}\mathcal{P}_{\mathcal{A}_{\bar y}^i}^\star\mathcal{P}_{\mathcal{A}_{\bar y}^i}}M_i$. 
We denote by $r_{\Omega_i}: V \to V_i$, $r_{\Gamma_i}: V_\Gamma \to V_{\Gamma_i}$, $r_{\mathcal{U}_i}: U \to U_{i}$, and $r_{\mathcal{A}_i} : U_\mathcal{A} \to U_{\mathcal{A}_i}$ the restriction operators to $\Omega_i$, $\Gamma_i$, $\mathcal{U}_i$, and $\mathcal{A}_{\bar u, \bar \mu}^i$, respectively. Their adjoint operators
$r_{\Omega_i}^\star: V_i^\star \to V^\star$, $r_{\Gamma_i}^\star: V_{\Gamma_i}^\star \to V_{\Gamma}^\star$, $r_{\mathcal{U}_i}^\star:  U_{i} \to U$, and $r_{\mathcal{A}_i}^\star :  U_{\mathcal{A}_i} \to U_\mathcal{A} $ are the corresponding prolongation operators. Note that since we have identified $U$, $U_{\mathcal{A}}$, $U_i$, and $U_{\mathcal{A}_i}$ with their dual spaces via the Riesz-representation theorem, $r_{\mathcal{U}_i}^\star$ and $r_{\mathcal{\mathcal{A}}_i}^\star$ coincide with the extension-by-zero operators. We denote by $\operatorname{tr}_{\Gamma_i}: V_i\to
V_{\Gamma_i}$ the trace operator restricted to $\Gamma_i$.

Restricting the bounds and the desired state, we set
$y_{d,i} := y_d|_{\Omega_i}$, $y_{a,i} := y_a|_{\Omega_i}$, and
$y_{b,i} := y_b|_{\Omega_i}$, and denote by $\Phi_\gamma$ also the Nemytskii
operator on $L^2(\Omega_i)$ induced by $y_{a,i}$ and $y_{b,i}$.  Recalling
the linearization point $\bar{x} = (\bar y, \bar p, \bar u, \bar \mu)$ of
\Cref{ss:symmetrized}, we write
\begin{align} \label{eq:local_iterate}
  \bar{x}_i := (\bar{y}_i, \bar{p}_i, \bar{u}_i, \bar{\mu}_i)
  := \del[1]{r_{\Omega_i}\bar{y},\, r_{\Omega_i}\bar{p},\,
             r_{\mathcal{U}_i}\bar{u},\, r_{\mathcal{U}_i}\bar{\mu}}
\end{align}
for its restriction to $\Omega_i$ and $\mathcal{U}_i$, and define the local
residual components, in analogy with~\eqref{eq:residual_components}, by
\begin{subequations} \label{eq:local_residual_components}
\begin{align}
  \bar{\mathcal{F}}_{1,i}
    &:= \widetilde{A}_i^\star \bar{p}_i
       + M_i^\star\del[1]{M_i \bar{y}_i - y_{d,i}
         + \Phi_\gamma(M_i \bar{y}_i)},
  \label{eq:local_residual_components_1} \\
  \bar{\mathcal{F}}_{2,i}
    &:= \widetilde{A}_i \bar{y}_i - B_i \bar{u}_i,
  \label{eq:local_residual_components_2} \\
  \bar{\mathcal{F}}_{3,i}
    &:= \alpha \bar{u}_i - B_i^\star \bar{p}_i + \bar{\mu}_i,
  \label{eq:local_residual_components_3} \\
  \bar{\mathcal{F}}_{4,i}
    &:= \Psi(\bar{u}_i, \bar{\mu}_i).
  \label{eq:local_residual_components_4}
\end{align}
\end{subequations}
Note that the local operators and residuals are induced from the global operators by restriction in the following sense:

\begin{proposition}[Assembly identities]\label{prop:assembly}
The global operators defined in \Cref{ss:spaces_ops} decompose additively over subdomains:
\begin{align} \label{eq:operator_decomp}
\widetilde{A} = \sum_{i=1}^N  r_{\Omega_i}^\star \widetilde{A}_i r_{\Omega_i},
\; \; 
L = \sum_{i=1}^N  r_{\Omega_i}^\star  L_i r_{\Omega_i},
\; \; 
B = \sum_{i=1}^N r_{\Omega_i}^\star  B_i r_{\mathcal{U}_i}, \; \;  \mathcal{P}_{\mathcal{A}} = \sum_{i=1}^N r_{\mathcal{A}_i}^\star \mathcal{P}_{\mathcal{A}_i} r_{\mathcal{U}_i}.
\end{align} 
The residual components of \eqref{eq:residual_components} decompose as
\begin{align} \label{eq:residual_decomp}
  \bar{\mathcal{F}}_1 = \sum_{i=1}^N r_{\Omega_i}^\star \bar{\mathcal{F}}_{1,i}, \quad   \bar{\mathcal{F}}_2 = \sum_{i=1}^N r_{\Omega_i}^\star \bar{\mathcal{F}}_{2,i}, \quad   \bar{\mathcal{F}}_3 = \sum_{i=1}^N r_{\mathcal{U}_i}^\star \bar{\mathcal{F}}_{3,i}, \quad \bar{\mathcal{F}}_4 = \sum_{i=1}^N r_{\mathcal{U}_i}^\star \bar{\mathcal{F}}_{4,i}.
\end{align}
\end{proposition}
\begin{proof}
 The unions $\overline \Omega = \bigcup_{i=1}^N \overline \Omega_i$, $\overline{\mathcal{U}} = \bigcup_{i=1}^N \overline{\mathcal{U}}_i$, and $ \overline{\mathcal{A}}_{\bar y} = \bigcup_{i=1}^N \overline{\mathcal{A}}_{\bar y}^i$ localize the integrals in the forms associated with each global operator to subdomain pieces. 
\end{proof}

\subsubsection{Local product spaces}
\label{sss:local_product_spaces}
To define the local subdomain KKT problems, we introduce local analogues of the product spaces $X$ and $X_\Gamma$:
\begin{equation}\label{eq:local_products}
  X_i=V_i\times V_i\times U_i\times U_{\mathcal{A}_i}, \quad
  X_{\Gamma_i}=V_{\Gamma_i}\times V_{\Gamma_i},
\end{equation}
as well as the following trace operator on the local product space:
\begin{align}
   \tau_i : X_i \to X_{\Gamma_i}, \quad \tau: (y_i,p_i,u_i,\mu_{\mathcal{A}_i}) \mapsto (\operatorname{tr}_{\Gamma_i}(y_i),\operatorname{tr}_{\Gamma_i}(p_i)).
\end{align}
On these product spaces, we define the component-wise restriction operators 
\begin{align} \label{eq:componetwise_volume_restrict}
   \rho_i &:= X \to X_i, \quad \rho_i := (r_{\Omega_i},r_{\Omega_i},r_{\mathcal{U}_i}, r_{\mathcal{A}_i}), \\ \label{eq:componetwise_trace_restrict}
   \rho_{\Gamma_i} &:= X_\Gamma \to X_{\Gamma_i}, \quad \rho_{\Gamma_i} := (r_{\Gamma_i},r_{\Gamma_i}).
\end{align}
Observe that the following identity holds:
\begin{align}
 \rho_{\Gamma_i}  \tau & =  \tau_i\rho_i. \label{eq:rho_tau_identity.1}
\end{align}
Analogously to the global spaces, we define the subspace of tuples in $X_i$ whose state and adjoint components vanish on $\Gamma_i$:
\begin{align} \label{eq:def_local_interior_space}
    X_I^i := \operatorname{ker}(\tau_i) = V_0^i \times V_0^i \times U_i \times U_{\mathcal{A}_i},
\end{align}
and denote by $\iota_i : X_I^i \hookrightarrow X_i$ the canonical injection of $X_I^i$ into $X_i$.

We identify the global interior space with the product of the local interior spaces. Define the \emph{broken interior space}
\begin{align} \label{eq:broken_interior_space}
    \widetilde{X}_I := \prod_{i=1}^N X_I^i,
\end{align}
equipped with the product norm and with coordinate projections $\widetilde{\pi}_i : \widetilde{X}_I \to X_I^i$ and injections $\widetilde{\jmath}_i : X_I^i \to \widetilde{X}_I$. 
\begin{lemma}[Decomposition of the interior space] \label{lem:interior_space_decomp}
The map
\begin{align*}
    \mathcal{B}_I : X_I \to \widetilde{X}_I, \qquad \mathcal{B}_I x := (\rho_i \iota_I x)_{i=1}^N,
\end{align*}
is well defined and an isometric isomorphism. Moreover, the maps
\begin{align} \label{eq:interior_coordinate_maps}
    \pi_i := \widetilde{\pi}_i \mathcal{B}_I : X_I \to X_I^i, \qquad \jmath_i := \mathcal{B}_I^{-1} \widetilde{\jmath}_i : X_I^i \to X_I,
\end{align}
satisfy the following identities:
\begin{align} \label{eq:rho_identity_2}
    \iota_i \pi_i = \rho_i \iota_I, \qquad \rho_j \iota_I \jmath_i = \delta_{ij} \iota_i, \qquad i,j = 1,\dots,N.
\end{align}
\end{lemma}

\begin{proof}
After verifying its hypotheses, we aim to apply \Cref{prop:homogeneous_decomp} with the choices $\mathcal{X} := X$, $\widetilde{\mathcal{X}}:= \widetilde{X}$, $\mathcal{X}_0 := X_I$, $\mathcal{X}_0^i := X_I^i$, $C_i := \tau_i$, and 
\begin{align*}
    \mathsf{R}: x \mapsto (\rho_i x)_{i=1}^N.
\end{align*}
To this end, note that $\mathsf{R}$ is an isometry, since the squares of the norms on $V$, $U$, and $U_\mathcal{A}$ decompose additively over the subdomains. Moreover, by \eqref{eq:rho_tau_identity.1}, if $x \in X_I$, then  $ C_i \widetilde{\pi}_i \mathsf{R} x = \tau_i \rho_i x = \rho_{\Gamma_i} \tau x = 0$ for $i=1,\dots,N$. Conversely, if $x \in X$ is such that $C_i \widetilde{\pi}_i \mathsf{R} x  = 0$ for $i=1,\dots,N$, then $\rho_{\Gamma_i} \tau x = 0$ for $i=1,\dots,N$. Since the map $\xi \mapsto (\rho_{\Gamma_i} \xi)_{i=1}^N$ is injective, it holds that $\tau x = 0$ and hence $x \in X_I$. Finally, we verify the gluing condition \eqref{eq:gluing_condition}. Let $\widetilde{x} = (y_i,p_i,u_i,\mu_i)_{i=1}^N \in \widetilde{X}_I$ and note that  extending by zero for each $i=1,\dots, N$ yields a tuple $(\widetilde y_i,\widetilde p_i, \widetilde u_i, \widetilde \mu_i)_{i=1}^N \in X_I$. Then,   $ x = \sum_{i=1}^N \iota_I(\widetilde y_i,\widetilde p_i, \widetilde u_i, \widetilde \mu_i) \in X$ is such that $\mathsf{R}x = \widetilde{x}$. The result now follows from \Cref{prop:homogeneous_decomp}.
\end{proof}

\subsection{Local subdomain KKT problems}

We define the following local KKT operator $\mathcal{K}_i:X_i \to X_i^\star$, built from the local operators $A_i$, $L_i$, $B_i$, and $\mathcal{P}_{\mathcal{A}_i}$ defined in \Cref{sss:local_operators}:
\begin{align*}
    \mathcal{K}_i := \begin{bmatrix}
    L_i & \widetilde{A}_i^\star & 0 & 0 \\
    \widetilde{A}_i & 0 & -B_i & 0 \\
    0 & -B_i^\star & \alpha\,I_{U_i}
      & \mathcal{P}_{\mathcal{A}_i}^\star \\
    0 & 0 & \mathcal{P}_{\mathcal{A}_i} & 0
  \end{bmatrix}.
\end{align*}
As a consequence of \Cref{prop:assembly}, the local KKT operator $\mathcal{K}_i$ is induced from the global KKT operator $\mathcal{K}$ by restriction in the following sense:
\begin{corollary}[KKT operator assembly]\label{cor:KKT_operator_assembly}
    With $\rho_i: X \to X_i$ the component-wise restriction operator, it holds that
    \begin{align} \label{eq:K_global_to_local}
        \mathcal{K} = \sum_{i=1}^N \rho_i^\star \mathcal{K}_i \rho_i, \quad f = \sum_{i=1}^N \rho_i^\star f_i,
    \end{align}
 where the local load is
\begin{align*}
  f_i:=\begin{bmatrix}
    -\bar{\mathcal F}_{1,i}\\
    -\bar{\mathcal F}_{2,i}\\
    -\bar{\mathcal F}_{3,i}
    -c^{-1}\mathcal P_{\mathcal I_i}^\star
     \mathcal P_{\mathcal I_i}\bar{\mathcal F}_{4,i}\\
    -\mathcal P_{\mathcal A_i}\bar{\mathcal F}_{4,i}
  \end{bmatrix}.
\end{align*}
\end{corollary}

Consider the following local subdomain problems: for each  $i = 1, \dots, N$, prescribe Dirichlet interface
data $x_{\Gamma_i} = (y_{\Gamma_i}, p_{\Gamma_i}) \in X_{\Gamma_i}$ for the state
and adjoint traces. The \emph{local subdomain problem} is to find $x_i \in X_i$
such that
\begin{align}\label{eq:local_subdomain_problem}
  \iota_i^\star \del[0]{\mathcal{K}_i x_i - f_i} = 0
    \ \text{ in } (X_I^i)^\star,
  \qquad
  \tau_i x_i = x_{\Gamma_i}.
\end{align}

Written out for $x_i := (y_i, p_i, u_i, \mu_{\mathcal{A}_i}) \in X_i$,
problem~\eqref{eq:local_subdomain_problem} reads
\begin{subequations}\label{eq:local_subdomain_block}
\begin{align}
  L_i y_i + \widetilde{A}_i^\star p_i
    &= -\bar{\mathcal{F}}_{1,i} && \text{in } (V_0^i)^\star,
    \label{eq:local_subdomain_block_a} \\
  \widetilde{A}_i y_i - B_i u_i
    &= -\bar{\mathcal{F}}_{2,i} && \text{in } (V_0^i)^\star,
    \label{eq:local_subdomain_block_b} \\
  -B_i^\star p_i + \alpha u_i + \mathcal{P}_{\mathcal{A}_i}^\star \mu_{\mathcal{A}_i}
    &= -\bar{\mathcal{F}}_{3,i}
       - c^{-1} \mathcal{P}_{\mathcal{I}_i}^\star \mathcal{P}_{\mathcal{I}_i}
         \bar{\mathcal{F}}_{4,i}
    && \text{in } U_i,
    \label{eq:local_subdomain_block_c} \\
  \mathcal{P}_{\mathcal{A}_i} u_i
    &= -\mathcal{P}_{\mathcal{A}_i} \bar{\mathcal{F}}_{4,i}
    && \text{in } U_{\mathcal{A}_i},
    \label{eq:local_subdomain_block_d}
\end{align}
subject to the Dirichlet interface conditions
\begin{align}\label{eq:local_subdomain_dirichlet}
  \operatorname{tr}_{\Gamma_i}(y_i) = y_{\Gamma_i},
  \qquad
  \operatorname{tr}_{\Gamma_i}(p_i) = p_{\Gamma_i}.
\end{align}
\end{subequations}
A discussion of the well-posedness of the local subdomain problems \eqref{eq:local_subdomain_block} is deferred to \Cref{sss:local_KKT_solvable}. We now show that solving the global problem is equivalent to solving the local subdomain problems, subject to an additional interface flux condition coupling them together.

\begin{proposition}[Interface flux]\label{prop:interface_flux}
Suppose $\iota_i^\star\del[0]{\mathcal{K}_i x_i - f_i} = 0$ in $(X_I^i)^\star$.
Then there exists a unique $\Lambda_i \in X_{\Gamma_i}^\star$ such that
\begin{align}\label{eq:flux_functional}
    \mathcal{K}_i x_i - f_i = \tau_i^\star \Lambda_i \quad \text{in } X_i^\star.
\end{align}
\end{proposition}
\begin{proof}
By assumption, $\mathcal{K}_i x_i - f_i$ annihilates
$X_I^i = \operatorname{ker}(\tau_i)$. Since $\tau_i : X_i \to X_{\Gamma_i}$ is
surjective, it has closed range, so the closed range theorem gives
$\operatorname{ran}(\tau_i^\star) = \operatorname{ker}(\tau_i)^\circ$, the polar set of $\operatorname{ker}(\tau_i)$. Hence
$\mathcal{K}_i x_i - f_i \in \operatorname{ran}(\tau_i^\star)$, and the uniqueness of $\Lambda_i$
follows from the injectivity of $\tau_i^\star$.
\end{proof}

\begin{theorem}[Global--local equivalence]\label{thm:equivalence}
Suppose that $x \in X$ solves the global KKT
system~\eqref{eq:block_KKT}, and set
$x_i := \rho_i x \in X_i$,
  $x_\Gamma := \tau x \in X_\Gamma$.
Then the following conditions hold:
\begin{enumerate}[label=\textup{(\alph*)}]
\item \emph{(Local subdomain problems)}
For each $i=1,\dots,N$, $x_i\in X_i$ solves the local
subdomain problem with interface data $\rho_{\Gamma_i}x_\Gamma \in X_{\Gamma_i}$; that is,
\begin{align}\label{eq:local_interior}
  \iota_i^\star\del{\mathcal K_i x_i-f_i}
  &=0
  &&\text{in }(X_I^i)^\star,
  &
  \tau_i x_i
  &=\rho_{\Gamma_i}x_\Gamma.
\end{align}

\item \emph{(Interface transmission)}
The corresponding interface fluxes
$\Lambda_i\in X_{\Gamma_i}^\star$ of
\Cref{prop:interface_flux} satisfy
\begin{align}\label{eq:transmission}
  \sum_{i=1}^N
  \rho_{\Gamma_i}^\star\Lambda_i
  =0
  \qquad\text{in }X_\Gamma^\star.
\end{align}
\end{enumerate}
Conversely, let $x_\Gamma\in X_\Gamma$ be given, and suppose that,
for each $i=1,\dots,N$, $x_i\in X_i$ solves the local subdomain
problem in \textup{(a)} with interface data
$\rho_{\Gamma_i}x_\Gamma$. If the corresponding interface fluxes
satisfy \textup{(b)}, then there exists a unique $x\in X$ satisfying~\eqref{eq:block_KKT}, $\rho_i x = x_i$, and $\tau x = x_\Gamma$.
\end{theorem}

\begin{proof}

\emph{Global $\Rightarrow$ local.} Let $x$ solve \eqref{eq:block_KKT}.  The trace relation in
\eqref{eq:local_interior} follows from
$\tau_i\rho_i=\rho_{\Gamma_i}\tau$.  Fix $i$ and let
$w_I^i\in X_I^i$.  Testing \eqref{eq:block_KKT} with $\iota_I \jmath_i w_I^i \in X$ and using \eqref{eq:K_global_to_local} as well as the fact that
$\rho_j  \iota_I  \jmath_i=\delta_{ij}\iota_i$ by \eqref{eq:rho_identity_2},  we obtain
\begin{align*}
 0
  =\langle\mathcal Kx-f,\iota_I\jmath_iw_I^i\rangle_X =\sum_{j=1}^N
    \langle\mathcal K_j\rho_jx-f_j,
      \rho_j\iota_I\jmath_iw_I^i\rangle_{X_j}=\langle\mathcal K_i x_i-f_i,\iota_iw_I^i\rangle_{X_i}.
\end{align*}
Since $w_I^i$ was arbitrary, the interior equation in
\eqref{eq:local_interior} follows.  By
\Cref{prop:interface_flux}, there is a unique
$\Lambda_i\in X_{\Gamma_i}^\star$ satisfying
$\mathcal K_i x_i-f_i=\tau_i^\star\Lambda_i$.
Taking adjoints in
\eqref{eq:rho_tau_identity.1} gives
\begin{align}\label{eq:adjoint_trace_commutation}
  \rho_i^\star\tau_i^\star
  =\tau^\star\rho_{\Gamma_i}^\star.
\end{align}
Consequently,
\begin{align*}
  0
  &=\mathcal Kx-f
   =\sum_{i=1}^N\rho_i^\star\tau_i^\star\Lambda_i
   =\tau^\star
    \sum_{i=1}^N\rho_{\Gamma_i}^\star\Lambda_i.
\end{align*}
Since $\tau$ is surjective, $\tau^\star$ is injective, and
\eqref{eq:transmission} follows.

\medskip
\emph{Local $\Rightarrow$ global.}
Conversely, define $x$ such that $\rho_i x = x_i$ for $i=1,\dots, N$.
The trace conditions in \textup{(a)} imply that the piecewise
$H^1$-functions defined by $r_{\Omega_i}y=y_i$ and
$r_{\Omega_i}p=p_i$ have matching traces on every common interface, and thus $y,p\in V$.  Therefore, $x\in X$ and moreover $x$ satisfies $\rho_i x=x_i$ and $\tau x=x_\Gamma$.
Using \eqref{eq:K_global_to_local},
\eqref{eq:flux_functional}, \eqref{eq:adjoint_trace_commutation}, and
\eqref{eq:transmission}, we find
\begin{align*}
  \mathcal Kx-f
  &=\sum_{i=1}^N\rho_i^\star
    \del{\mathcal K_i x_i-f_i}
   =\sum_{i=1}^N\rho_i^\star\tau_i^\star\Lambda_i=\tau^\star
    \sum_{i=1}^N\rho_{\Gamma_i}^\star\Lambda_i
   =0.
\end{align*}
Thus $x$ solves \eqref{eq:block_KKT}.
\end{proof}

\begin{remark}[Interface fluxes]\label{rem:interface_interpretation}
Writing $\Lambda_i := (\Lambda_i^y, \Lambda_i^p) \in V_{\Gamma_i}^\star \times
V_{\Gamma_i}^\star$, the components are the weak normal fluxes of the state and
adjoint on $\Gamma_i$: for any  $(v_{\Gamma_i}, q_{\Gamma_i}) \in X_{\Gamma_i}$ and any $(v_i,q_i) \in X_i$ such that $(\operatorname{tr}_{\Gamma_i} v_i,\operatorname{tr}_{\Gamma_i} q_i) = (v_{\Gamma_i}, q_{\Gamma_i}) $, it holds that
\begin{align*}
  \langle \Lambda_i^y, q_{\Gamma_i} \rangle_{V_{\Gamma_i}}
    = \langle \widetilde{A}_i y_i - B_i u_i, q_i \rangle_{V_i},
  \qquad
  \langle \Lambda_i^p, v_{\Gamma_i} \rangle_{V_{\Gamma_i}}
    = \langle L_i y_i + \widetilde{A}_i^\star p_i, v_i
      \rangle_{V_i},
\end{align*}
which coincide with $\partial_{n_i} y_i$ and $\partial_{n_i} p_i$ for
sufficiently regular solutions. The transmission
condition~\eqref{eq:transmission} is then the flux balance law: for all
$(v_\Gamma, q_\Gamma) \in X_\Gamma$,
\begin{align}\label{eq:flux_balance}
  \sum_{i=1}^N \int_{\Gamma_i} \frac{\partial p_i}{\partial n_i}\, v_{\Gamma_i}
    \,\mathrm{d}s
  + \sum_{i=1}^N \int_{\Gamma_i} \frac{\partial y_i}{\partial n_i}\, q_{\Gamma_i}
    \,\mathrm{d}s
  = 0,
\end{align}
generalizing the classical transmission condition for elliptic problems
(cf.~\cite{quarteroni-valli-1999-domain-decomposition}) and unconstrained
elliptic optimal control problems (cf.~\cite{heinkenschloss2006neumann}).
\end{remark}

\subsubsection{Decomposition of the local KKT operator}

Analogously to the global decomposition \eqref{eq:X_decomp_E}, we can decompose functions in $X_i$ into a local interior component and a lifted trace component.
Owing to the surjectivity of the trace operator $\tau_i: X_i \to X_{\Gamma_i}$, there exists a (non-unique) bounded lifting $E_i : X_{\Gamma_i} \to X_i$ satisfying $\tau_i E_i = I_{X_{\Gamma_i}}$. Moreover, any such choice of $E_i$ induces a decomposition of $X_i$ into the topological direct sum:
\begin{align} \label{eq:X_decomp_E_loc}
    X_i = X_I^i \oplus E_i(X_{\Gamma_i}).
\end{align}
In other words, any $x_i \in X_i$ can be decomposed as $x_i = \iota_i x_I^i + E_i x_{\Gamma_i}$, for some $(x_I^i,x_{\Gamma_i}) \in X_I^i \times X_{\Gamma_i}$. This decomposition is unique only up to the choice of lifting.  
Analogously to the global setting, we define the operators $\mathcal{K}^i_{II}: X_I^i \to (X_I^i)^\star$, $\mathcal{K}_{I\Gamma}^{E_i}: X_{\Gamma_i} \to (X_I^i)^\star$, $\mathcal{K}_{\Gamma I}^{E_i} : X_I^i \to X_{\Gamma_i}^\star$, and $\mathcal{K}_{\Gamma \Gamma}^{E_i}: X_{\Gamma_i} \to X_{\Gamma_i}^\star$ via
\begin{align} \label{eq:local_KKT_blocks}
    \mathcal{K}^i_{II} := (\iota_i)^\star \mathcal{K}_i \iota_i, \quad \mathcal{K}_{I\Gamma}^{E_i} := \iota_i^\star \mathcal{K}_i E_i, \quad \mathcal{K}_{\Gamma I}^{E_i} := E_i^\star\mathcal{K}_i \iota_i, \quad \mathcal{K}_{\Gamma \Gamma}^{E_i} := E_i^\star\mathcal{K}_i E_i.
\end{align}
Note that the operator $\mathcal{K}^i_{II}$, which we will refer to as the \emph{local interior KKT operator}, is independent of the choice of lifting. Moreover, since $\mathcal{K}_i$ is self-adjoint, $(\mathcal{K}_{I\Gamma}^{E_i})^\star = (\mathcal{K}_{\Gamma I}^{E_i})$.
The \emph{local interior KKT operator} $\mathcal{K}^i_{II}$ is independent of the choice of lifting, and $(\mathcal{K}_{I\Gamma}^{E_i})^\star = \mathcal{K}_{\Gamma I}^{E_i}$ since $\mathcal{K}_i$ is self-adjoint.\\

The global interior KKT operator $\mathcal{K}_{II}$ is congruent, via the isometric isomorphism $\mathcal{B}_I$ of \Cref{sss:local_product_spaces}, to the direct sum of the local interior KKT operators:

\begin{corollary}[Decomposition of the global interior KKT operator] \label{cor:KII_decomp}
Let $\mathcal{B}_I : X_I \to \widetilde{X}_I$ be the isomorphism in \Cref{lem:interior_space_decomp}. It holds that
\begin{align} \label{eq:KII_decomp}
    (\mathcal{B}_I^{-1})^\star \mathcal{K}_{II} \mathcal{B}_I^{-1} = \bigoplus_{i=1}^N \mathcal{K}_{II}^i.
\end{align}
Consequently, $\mathcal{K}_{II}$ is an isomorphism if and only if $\mathcal{K}_{II}^i$ is an isomorphism for each $i=1,\dots,N$.
\end{corollary}
\begin{proof}
Apply \Cref{cor:homogeneous_block} with $\mathcal{X} = X$, $\mathcal{X}_i = X_i$, $\mathsf{R}x = (\rho_i x)_{i=1}^N$, and $\mathsf{C}_i = \tau_i$, so that $\mathcal{X}_0 = X_I$, $\mathcal{X}_0^i = X_I^i$, $\widetilde{\mathcal{X}}_0 = \widetilde{X}_I$, $\iota_0 = \iota_I$, $\iota_0^i = \iota_i$, $\jmath_0^i = \jmath_i$, $\mathcal{B}_0 = \mathcal{B}_I$, and  $\mathsf{T}_i := \mathcal{K}_i$. 
\end{proof}

\subsubsection{Solvability of the local KKT problems} \label{sss:local_KKT_solvable}

\begin{theorem}[Properties of the local block $\mathcal K_i$]
\label{thm:KII_fredholm}
For each $i=1,\dots,N$ and each $\sigma\ge 0$, the local operator
$\mathcal K_i:X_i\to X_i^\star$ is Fredholm with
$\operatorname{ind}(\mathcal K_i)=0$, and the local interior operator
$\mathcal K_{II}^i:X_I^i\to(X_I^i)^\star$ is an isomorphism. Moreover:
\begin{enumerate}[label=\textup{(\alph*)},leftmargin=2.4em]
\item In case \textup{(DC)}, the following alternatives hold:
  \begin{enumerate}[label=\textup{(\roman*)},leftmargin=2.6em]
  \item $\mathcal K_i$ is an isomorphism if $\sigma>0$, or if
  $\sigma=0$ and either $\Omega_i$ is anchored or
  $\Omega_i$ is floating with
  $|\mathcal I_{\bar u,\bar\mu}^i|> 0$.

  \item If $\sigma=0$, $\Omega_i$ is floating, and
  $|\mathcal I_{\bar u,\bar\mu}^i|=0$, then
  \begin{align*}
    \ker(\mathcal K_i)
    =\cbr[1]{
c_i\cdot\del{0,1|_{\Omega_i},0,1|_{\mathcal A_i}}
      :c_i\in\mathbb R
    }.
  \end{align*}
  \end{enumerate}

\item In case \textup{(NC)}, the following alternatives hold:
  \begin{enumerate}[label=\textup{(\roman*)},leftmargin=2.6em]
  \item $\mathcal K_i$ is an isomorphism if $\sigma>0$, or if
  $\sigma=0$, $\Omega_i$ is anchored (i.e., $i \in \mathcal{N}_c$), and
  $|\mathcal I_{\bar u,\bar\mu}^i|>0$.

  \item If $\sigma=0$, $\Omega_i$ is anchored (i.e., $i \in \mathcal{N}_c$), and
  $|\mathcal I_{\bar u,\bar\mu}^i|=0$, then
  \begin{align*}
    \ker(\mathcal K_i)
    =\cbr[1]{
c_i\cdot\del{0,1|_{\Omega_i},0,1|_{\mathcal A_i}}
      :c_i\in\mathbb R
    }.
  \end{align*}

  \item If $\sigma=0$ and $\Omega_i$ is floating (i.e., $i \in \mathcal{N}_\circ$), then
  \begin{align*}
    \ker(\mathcal K_i)
    =\cbr[1]{
      c_i\cdot\del{0,1|_{\Omega_i},0,0}
      :c_i\in\mathbb R
    }.
  \end{align*}
  \end{enumerate}
\end{enumerate}
\end{theorem}
\begin{proof}
The proofs that $\mathcal{K}_i$ and $\mathcal{K}_{II}^i$ are Fredholm with index $0$ are identical to the proof that the global KKT operator $\mathcal{K}$ is Fredholm (\emph{cf.} \Cref{thm:global_well_posed}), and are therefore omitted. That $\mathcal{K}_{II}^i$ is furthermore an isomorphism is a consequence  of the fact that $\widetilde{A}_i : V_0^i \to (V_0^i)^\star$ is an isomorphism owing to the Poincar\'{e} inequality. It remains to characterize $\operatorname{ker}(\mathcal{K}_i)$. Arguing analogously to the proof of \Cref{thm:global_well_posed}, it can be shown that 
\begin{align*}
    (y_i,p_i,u_i,\mu_i) \in \text{ker}(\mathcal{K}_i) \quad \Longrightarrow \quad (y_i, p_i, u_i, \mu_i) = (0, c_i,  0,\mathcal{P}_{\mathcal{A}_i} B_i^\star c_i).
\end{align*}

\textbf{Case \textup{(DC)}.}
If $\sigma>0$, or if $\sigma=0$ and $\Omega_i$ is anchored, then
$\widetilde{A}_i$ is an isomorphism by the Poincar\'{e} inequality and Lax--Milgram theorem. Hence $c_i=0$, and therefore
$\operatorname{ker}(\mathcal{K}_i)=\cbr{0}$.
If $\sigma=0$ and $\Omega_i$ is floating, 
the same argument as in the
(NC) case of \Cref{thm:global_well_posed} shows that
$|\mathcal{I}_{\bar{u},\bar{\mu}}^i|>0$ implies $c_i=0$.
If instead $|\mathcal{I}_{\bar{u},\bar{\mu}}^i|=0$, then
$
\operatorname{ker}(\mathcal{K}_i)
=
\cbr[1]{
c_i \cdot
\del{
0,
1|_{\Omega_i},
0,
1|_{\mathcal{A}_i}
}
:
c_i\in\mathbb{R}
}$.

\medskip
\textbf{Case \textup{(NC)}.}
If $\sigma>0$, then $\widetilde{A}_i$ is an isomorphism, and hence
$\operatorname{ker}(\mathcal{K}_i)=\cbr{0}$.
Suppose that $\sigma=0$. Then
$p_i=c_i\cdot 1|_{\Omega_i}$ for some $c_i\in\mathbb{R}$. If $\Omega_i$ is
anchored and $|\mathcal{I}_{\bar{u},\bar{\mu}}^i|>0$, an identical argument to that of the (NC) case
in the proof of \Cref{thm:global_well_posed} implies $c_i=0$.
If $|\mathcal{I}_{\bar{u},\bar{\mu}}^i| = 0$, then
$
\operatorname{ker}(\mathcal{K}_i)
=
\cbr[1]{
c_i \cdot
\del{
0,
1|_{\Omega_i},
0,
1|_{\mathcal{A}_i}
}
:
c_i\in\mathbb{R}
}$.
Finally, if $\Omega_i$ is floating, then
$U_i=U_{\mathcal{A}_i}=\cbr{0}$ and $B_i=0$. Therefore,
$
\operatorname{ker}(\mathcal{K}_i)
=
\cbr[1]{
c_i \cdot
\del{
0,
1|_{\Omega_i},
0,
0
}
:
c_i\in\mathbb{R}
}$.
\end{proof}

\subsubsection{Local Schur complement problem}
Having established that $\mathcal K_{II}^i$ is an isomorphism, let
$E_i:X_{\Gamma_i}\to X_i$ be a lifting and write
$x_i=\iota_i x_I^{E_i}+E_i x_{\Gamma_i}$. The local interior equation
gives
\begin{align*}
  x_I^{E_i}
  =(\mathcal K_{II}^i)^{-1}
   \bigl(f_I^i-\mathcal K_{I\Gamma}^{E_i}x_{\Gamma_i}\bigr).
\end{align*}
If $\Lambda_i\in X_{\Gamma_i}^\star$ is the interface flux from
\Cref{prop:interface_flux}, then applying $E_i^\star$ to
$\mathcal K_i x_i-f_i=\tau_i^\star\Lambda_i$ yields
\begin{align}\label{eq:local_flux_schur}
  \Lambda_i
  =\mathcal S_{\Gamma_i}^{E_i}x_{\Gamma_i}
   -g_{\Gamma_i}^{E_i},
\end{align}
where
\begin{align}
  \mathcal S_{\Gamma_i}^{E_i}
  &:=\mathcal K_{\Gamma\Gamma}^{E_i}
    -\mathcal K_{\Gamma I}^{E_i}
     (\mathcal K_{II}^i)^{-1}\mathcal K_{I\Gamma}^{E_i},
  \label{eq:local_schur_operator}\\
  g_{\Gamma_i}^{E_i}
  &:=f_{\Gamma_i}^{E_i}
    -\mathcal K_{\Gamma I}^{E_i}
     (\mathcal K_{II}^i)^{-1}f_I^i.
  \label{eq:local_schur_load}
\end{align}
Thus the local Schur complement maps prescribed interface data to the
corresponding interface flux, up to the reduced local load.

\begin{lemma}\label{lem:local_harmonic}
For each $i=1,\dots,N$, the following statements hold.
\begin{enumerate}[label=\textup{(\alph*)}]
\item The local Schur complement and reduced load do not depend on the
choice of lifting $E_i$. Thus, we simply write $\mathcal{S}_{\Gamma_i}$ and $g_{\Gamma_i}$ in what follows.
\item There is a unique bounded linear lifting
$\mathcal H_i:X_{\Gamma_i}\to X_i$ satisfying
\begin{align*}
  \iota_i^\star\mathcal K_i\mathcal H_i=0,
  \qquad
  \tau_i\mathcal H_i=I_{X_{\Gamma_i}}.
\end{align*}
For any lifting $E_i$, it is given by
\begin{align}\label{eq:local_harmonic_formula}
  \mathcal H_i
  =\bigl(I_{X_i}-\iota_i(\mathcal K_{II}^i)^{-1}
       \iota_i^\star\mathcal K_i\bigr)E_i.
\end{align}
\item It holds that
\begin{align} \label{eq:local_schur_harmonic}
  \mathcal S_{\Gamma_i}
  =\mathcal H_i^\star\mathcal K_i\mathcal H_i,
  \qquad
  g_{\Gamma_i}=\mathcal H_i^\star f_i.
\end{align}
\item It holds that
\begin{align}\label{eq:local_harmonic_flux}
  \mathcal K_i\mathcal H_i
  =\tau_i^\star\mathcal S_{\Gamma_i}.
\end{align}
\end{enumerate}
\end{lemma}

\begin{proof}
Parts (a) and (b) follow exactly as in
\Cref{lem:schur_independent,thm:global_harmonic_extension}. For (c), taking
$E_i=\mathcal H_i$ in the local block representation gives
$\mathcal K_{I\Gamma}^{\mathcal H_i}=0$ and, by self-adjointness,
$\mathcal K_{\Gamma I}^{\mathcal H_i}=0$. Hence the local Schur
complement and reduced load are
$\mathcal H_i^\star\mathcal K_i\mathcal H_i$ and
$\mathcal H_i^\star f_i$, respectively.
Finally, to show (d), note that $\iota_i^\star\mathcal K_i\mathcal H_i=0$, so
$\mathcal K_i\mathcal H_i$ annihilates
$X_I^i=\ker(\tau_i)$. Since $\tau_i$ is surjective,
$\operatorname{ran}(\tau_i^\star)=\ker(\tau_i)^\circ$; therefore
there is a unique $T_i:X_{\Gamma_i}\to X_{\Gamma_i}^\star$ with
$\mathcal K_i\mathcal H_i=\tau_i^\star T_i$. Applying
$\mathcal H_i^\star$ to both sides and using $\tau_i\mathcal H_i=I$ gives
$T_i=\mathcal H_i^\star\mathcal K_i\mathcal H_i
  =\mathcal S_{\Gamma_i}$.
which proves \eqref{eq:local_harmonic_flux}.
\end{proof}

\begin{theorem}[Local Schur complement properties]
\label{thm:local_fredholm}
For each $i = 1,\dots,N$, the local Schur complement
$\mathcal{S}_{\Gamma_i} : X_{\Gamma_i} \to X_{\Gamma_i}^\star$ is
Fredholm with $\operatorname{ind}(\mathcal{S}_{\Gamma_i}) = 0$.
Moreover:
\begin{enumerate}[label=\textup{(\alph*)}]
\item $\mathcal{S}_{\Gamma_i}$ is an isomorphism precisely in the cases in which $\mathcal{K}_i: X_i \to X_i^\star$ is an isomorphism in \Cref{thm:KII_fredholm}.
\item In every remaining case,
\begin{align} \label{eq:local_schur_kernel}
    \operatorname{ker}(\mathcal{S}_{\Gamma_i})
    = \cbr[1]{\del{0, c_i \cdot 1|_{\Gamma_i}} : c_i \in \mathbb{R}}.
\end{align}
\end{enumerate}
\end{theorem}

\begin{proof}
Fix a lifting $E_i : X_{\Gamma_i} \to X_i$, let $J_{E_i} : X_I^i \times X_{\Gamma_i} \to X_i$ denote the isomorphism $J_{E_i}(w_I^i, x_{\Gamma_i}) := \iota_i w_I^i + E_i x_{\Gamma_i}$ induced by the decomposition \eqref{eq:X_decomp_E_loc}, and let $\Phi_i : (X_I^i)^\star \times X_{\Gamma_i}^\star \to (X_I^i \times X_{\Gamma_i})^\star$ be the canonical identification (\emph{cf}. \Cref{rem:notation}). Analogously to the global setting, the local KKT operator $\mathcal{K}_i$ is congruent to the block operator
\begin{align} \label{eq:local_KKT_block_form_E}
    \widehat{\mathcal{K}}_i := \Phi_i^{-1} J_{E_i}^\star \mathcal{K}_i J_{E_i}
    = \begin{bmatrix} \mathcal{K}_{II}^i & \mathcal{K}_{I\Gamma}^{E_i} \\ \mathcal{K}_{\Gamma I}^{E_i} & \mathcal{K}_{\Gamma\Gamma}^{E_i} \end{bmatrix},
\end{align}
and thus applying \Cref{prop:schur_fredholm} with $\mathcal{Y} = X_I^i$, $\mathcal{Z} = X_{\Gamma_i}$, $\mathcal{X} = X_i$, $J = J_{E_i}$, $\mathsf{K} = \mathcal{K}_i$, $\mathsf{A} = \mathcal{K}_{II}^i$, $\mathsf{B} = \mathcal{K}_{I\Gamma}^{E_i}$, $\mathsf{C} = \mathcal{K}_{\Gamma I}^{E_i}$, and $\mathsf{D} = \mathcal{K}_{\Gamma\Gamma}^{E_i}$ alongside  \Cref{thm:KII_fredholm} yields that the local Schur complement
$\mathcal{S}_{\Gamma_i} : X_{\Gamma_i} \to X_{\Gamma_i}^\star$ is
Fredholm with $\operatorname{ind}(\mathcal{S}_{\Gamma_i}) = 0$, and that $\mathcal{S}_{\Gamma_i}$ is an isomorphism precisely when $\mathcal{K}_i$ is.

For part (b), \Cref{prop:schur_fredholm}(b) shows that the map $\Theta : \operatorname{ker}(\mathcal{S}_{\Gamma_i}) \to \operatorname{ker}(\mathcal{K}_i)$ defined by $\Theta(z) := J_{E_i}\del[1]{-(\mathcal{K}_{II}^i)^{-1} \mathcal{K}_{I\Gamma}^{E_i} z, \, z}$ is an isomorphism. Moreover, by \eqref{eq:local_KKT_blocks} and \eqref{eq:local_harmonic_formula},
\begin{align*}
  \Theta(z) = J_{E_i}\del[1]{-(\mathcal{K}_{II}^i)^{-1} \mathcal{K}_{I\Gamma}^{E_i} z, \, z} = \del[1]{I_{X_i}  - \iota_i (\mathcal{K}_{II}^i)^{-1} \iota_i^\star \mathcal{K}_i} E_i z = \mathcal{H}_i z.
\end{align*}
 Hence $\mathcal{H}_i : \operatorname{ker}(\mathcal{S}_{\Gamma_i}) \to \operatorname{ker}(\mathcal{K}_i)$ is an isomorphism, and applying $\tau_i$ together with $\tau_i \mathcal{H}_i = I_{X_{\Gamma_i}}$ gives $\operatorname{ker}(\mathcal{S}_{\Gamma_i}) = \tau_i\del[1]{\operatorname{ker}(\mathcal{K}_i)}$. The kernel characterizations in \Cref{thm:KII_fredholm} therefore give \eqref{eq:local_schur_kernel}.
\end{proof}

We end this subsection by characterizing the relationship between the global Schur complement operator $\mathcal{S}_\Gamma$ and the local subdomain Schur complement operators $\mathcal{S}_{\Gamma_i}$.

\begin{lemma}
For each $i=1,\dots,N$, the following relationship between the global and local harmonic extensions holds:
\begin{align} \label{eq:rho_H_identity}
    \smash{\rho_i  \mathcal{H} = \mathcal{H}_i  \rho_{\Gamma_i}.}
\end{align}
    Consequently, the following decompositions of the global Schur complement and global reduced load hold:
    \begin{align} \label{eq:global_schur_assembly}
 \mathcal S_\Gamma
  =
  \sum_{i=1}^N
  \rho_{\Gamma_i}^\star
  \mathcal S_{\Gamma_i}
  \rho_{\Gamma_i}, \quad   g_\Gamma
  =\sum_{i=1}^N
    \rho_{\Gamma_i}^\star g_{\Gamma_i}.
\end{align}
\end{lemma}

\begin{proof}
We first prove \eqref{eq:rho_H_identity}.
Let $x_\Gamma \in X_\Gamma$ be fixed and set $x_{\Gamma_i} := \rho_{\Gamma_i} x_\Gamma.$
Testing \eqref{eq:harmonic_extension_a} with an arbitrary function $w_I^i\in X_I^i$ and using the decomposition of the global KKT operator \eqref{eq:K_global_to_local} as well as  the fact that $\rho_j \iota_I \jmath_i=\delta_{ij}\iota_i$ by \eqref{eq:rho_identity_2}, we have 
\begin{align*}
\smash{  0
  =\langle\mathcal K\mathcal Hx_\Gamma,
    \iota_I\jmath_iw_I^i\rangle_X =\sum_{j=1}^N
    \langle\mathcal K_j\rho_j\mathcal Hx_\Gamma,
\rho_j\iota_I\jmath_iw_I^i\rangle_{X_j}=\langle\mathcal K_i\rho_i\mathcal Hx_\Gamma,
    \iota_iw_I^i\rangle_{X_i}}.
\end{align*}
The fact that $w_I^i \in X_I^i$ is arbitrary, combined with \eqref{eq:rho_tau_identity.1}, yields
\begin{align*}
    \iota_i^\star \mathcal K_i(\rho_i\mathcal Hx_\Gamma) &= 0, \\
  \tau_i(\rho_i\mathcal Hx_\Gamma)
  &=\rho_{\Gamma_i}x_\Gamma.
\end{align*}
The uniqueness of the local harmonic extension $\mathcal H_i$ established in \Cref{lem:local_harmonic} then yields \eqref{eq:rho_H_identity}. 
We next prove \eqref{eq:global_schur_assembly}. Since $\mathcal{S}_\Gamma = \mathcal{H}^\star \mathcal{K} \mathcal{H}$ by \Cref{cor:diagonalized_block_product_form} and $\mathcal{S}_{\Gamma_i} = \mathcal{H}_i^\star \mathcal{K}_i \mathcal{H}_i$ by \Cref{lem:local_harmonic}, \cref{eq:K_global_to_local} and \cref{eq:rho_H_identity} yield
\begin{align*}
  \mathcal S_\Gamma
  &=\mathcal H^\star\mathcal K\mathcal H
   =\sum_{i=1}^N
     \del{\rho_i\mathcal H}^\star
     \mathcal K_i\del{\rho_i\mathcal H}
  =\sum_{i=1}^N
    \rho_{\Gamma_i}^\star
    \mathcal H_i^\star\mathcal K_i\mathcal H_i
    \rho_{\Gamma_i}
   =\sum_{i=1}^N
    \rho_{\Gamma_i}^\star
    \mathcal S_{\Gamma_i}\rho_{\Gamma_i}.
    \end{align*}
Moreover, since $g_\Gamma = \mathcal{H}^\star f$ by \Cref{cor:diagonalized_block_product_form} and $g_{\Gamma_i} = \mathcal{H}_i^\star f_i$ by \Cref{lem:local_harmonic}, \cref{eq:rho_H_identity} yields
\begin{align*}
  g_\Gamma
  &=\mathcal H^\star f
   =\sum_{i=1}^N\del{\rho_i\mathcal H}^\star f_i=\sum_{i=1}^N
    \rho_{\Gamma_i}^\star\mathcal H_i^\star f_i
   =\sum_{i=1}^N
    \rho_{\Gamma_i}^\star g_{\Gamma_i}.
\end{align*}
\end{proof}

Thus, the action of the global Schur complement operator $\mathcal S_\Gamma: X_\Gamma \to X_\Gamma^\star$ amounts to solving $N$ independent local problems and assembling their interface contributions. The local solves are embarrassingly parallel; only the interface assembly is global. This property can be exploited by, e.g., iterative Krylov methods for the efficient solution of \eqref{eq:block_KKT}. However, the conditioning of the interface
Schur complement operator is delicate, and convergence of Krylov methods for the solution of \eqref{eq:global_schur_eq}
depends on effective preconditioning.  This motivates the BDDC
preconditioner developed in \Cref{sec:bddc_abstract}.

%

%----------------------------------------------------------------------

\section{Abstract BDDC preconditioner}\label{sec:bddc_abstract}
%======================================================================

The BDDC method builds an approximate inverse operator of the interface Schur complement from a partial assembly of $\mathcal{S}_\Gamma : X_\Gamma \to X_\Gamma^\star$. Rather than enforcing agreement of the local subdomain solution traces
on the interface, BDDC designates a finite set of trace functionals as
\emph{primal constraints} and enforces agreement only of these
quantities.
The broken and partially assembled interface operators are introduced
in \Cref{ss:broken_schur,ss:partial_schur}, where their Fredholm and
kernel properties are characterized. The BDDC operator is defined in
\Cref{ss:preconditioner}. The standard BDDC primal--dual splitting then yields a
two-level additive Schwarz representation consisting of independent
local dual-subspace solves and a finite-dimensional coarse correction.
%----------------------------------------------------------------------
\subsection{Broken interface operators}\label{ss:broken_schur}
%----------------------------------------------------------------------

The local Schur complements $\mathcal{S}_{\Gamma_i}$ each act on the
per-subdomain trace space $X_{\Gamma_i}$. The fully assembled
operator $\mathcal{S}_\Gamma$ couples them through the conformity
built into $X_\Gamma$. Removing the conformity constraints across subdomains yields the broken
interface space
\begin{align}\label{eq:broken_interface}
  \widetilde X_\Gamma := \prod_{i=1}^N X_{\Gamma_i}, \quad \| \widetilde{x} \|_{\widetilde{X}_\Gamma}^2 = \sum_{i=1}^N \| \widetilde{x}_{\Gamma_i} \|_{X_{\Gamma_i}}^2,
\end{align}
whose elements are tuples of independent subdomain trace pairs. In contrast
to the conforming space $X_\Gamma$, tuples in $\widetilde X_\Gamma$ need not agree across
shared boundaries $\overline{\Gamma}_i \cap \overline{\Gamma}_j$.
The conforming space embeds into the broken one via
$\iota_\Gamma : X_\Gamma \hookrightarrow \widetilde X_\Gamma$,
$\iota_\Gamma\, x_\Gamma := (\rho_{\Gamma_1} x_\Gamma,\dots,\rho_{\Gamma_N} x_\Gamma)$. For each $i=1,\dots,N$, we denote by $\widetilde{\pi}_{\Gamma_i}: \widetilde{X}_\Gamma \to X_{\Gamma_i}$ and $\widetilde{\jmath}_{\Gamma_i}: X_{\Gamma_i} \to \widetilde{X}_\Gamma$ the coordinate projection and coordinate injection operators, respectively.

A first approximation amenable to parallel inversion drops the
conformity constraint and works directly on the broken interface space $\widetilde X_\Gamma$. The local
Schur complements assemble into a block-diagonal operator on this
larger space:
\begin{align}\label{eq:broken_schur}
  \widetilde{\mathcal{S}}_{\Gamma} := \bigoplus_{i=1}^N \mathcal{S}_{\Gamma_i}
    : \widetilde X_\Gamma \to \widetilde X_\Gamma^\star.
\end{align}
Equivalently, 
\begin{align*}
    \widetilde{\mathcal{S}}_{\Gamma} = \sum_{i=1}^N \widetilde{\pi}_{\Gamma_i}^\star \mathcal{S}_{\Gamma_i} \widetilde{\pi}_{\Gamma_i}.
\end{align*}
The conforming Schur complement is recovered by
\begin{align*}
  \mathcal S_\Gamma
  =\iota_\Gamma^\star\widetilde{\mathcal S}_\Gamma\iota_\Gamma.
\end{align*}
When it is invertible, applying
$\widetilde{\mathcal S}_\Gamma^{-1}$ reduces to independent local
Schur-complement solves and requires no inter-subdomain
communication. This is the optimal control analogue of the local
substructuring solves used in Neumann--Neumann
methods~\cite{TosellWidlund2005}.

\begin{theorem}[Broken Schur complement properties]
\label{thm:broken_fredholm}
The operator
$\widetilde{\mathcal S}_\Gamma:\widetilde X_\Gamma\to
\widetilde X_\Gamma^\star$ is Fredholm with $\operatorname{ind}(\widetilde{\mathcal S}_\Gamma) = 0$ and
  $\ker(\widetilde{\mathcal S}_\Gamma)
  =\prod_{i=1}^N\ker(\mathcal S_{\Gamma_i})$.
Consequently:
\begin{enumerate}[label=\textup{(\alph*)}]
\item If $\ker(\mathcal S_{\Gamma_i})=\{0\}$ for every $i$, then
$\widetilde{\mathcal S}_\Gamma$ is an isomorphism.
\item Otherwise, it holds that
\begin{align*}
  \ker(\widetilde{\mathcal S}_\Gamma)
  =\prod_{i\in\mathcal N_*}
    \cbr[1]{(0,c_i \cdot 1|_{\Gamma_i}):c_i\in\mathbb R},
  \qquad
  \mathcal N_*:=\cbr[1]{i:\ker(\mathcal S_{\Gamma_i})\neq\{0\}}.
\end{align*}
In case \textup{(DC)}, $i\in\mathcal N_*$ precisely when
$\sigma=0$, $\Omega_i$ is floating, and
$|\mathcal I_{\bar u,\bar\mu}^i|=0$. In case \textup{(NC)}, it occurs
when $\sigma=0$ and either $i\in\mathcal N_\circ$ or
$i\in\mathcal N_c$ with $|\mathcal I_{\bar u,\bar\mu}^i|=0$.
\end{enumerate}
\end{theorem}
\begin{proof}
The result follows from 
\eqref{eq:broken_schur} and \Cref{thm:local_fredholm}.
\end{proof}
Inversion of the broken operator $\widetilde{\mathcal{S}}_\Gamma$ can be parallelized across subdomains. However, since it does not couple any information
across subdomains, the approximation degrades as the number of
subdomains grows.  A scalable preconditioner must restore some
interface coupling while preserving the parallel structure of the
local solves. BDDC does this by \emph{partial assembly}, i.e., by imposing the continuity of selected trace
functionals across interface entities.

%----------------------------------------------------------------------
\subsection{Partially assembled interface operators}\label{ss:partial_schur}
%----------------------------------------------------------------------

A more effective approximation to $\mathcal S_\Gamma^{-1}$ is obtained
by passing from the fully broken interface problem to a partially
assembled one.  Rather than requiring full trace continuity, BDDC enforces agreement of finitely many bounded linear functionals of the subdomain state and adjoint interface traces. In practice, these are typically vertex evaluations and edge or face averages.  The
resulting space lies between the conforming interface space and the
fully broken interface space, with only the primal trace quantities
shared across subdomains and all remaining trace components left
broken.

\subsubsection{Primal constraints}
The passage from the fully broken trace space to the partially
assembled one is determined by the choice of primal constraints.  These are
bounded linear functionals of the state and adjoint traces associated
with interface entities.  The partially assembled space is obtained by
identifying their values across all incident subdomains, while leaving
the remaining trace components independent.

\begin{definition}[Primal constraint family]\label{def:primal_constraints}
Let $N_\Pi \in \mathbb{N}$. For each $k = 1,\dots,N_\Pi$, let
$\mathcal{N}_k \subset \cbr{1,\dots,N}$ with $|\mathcal{N}_k| \ge 2$ denote the
set of subdomains sharing the $k^{\text{th}}$ primal constraint, and for each
$i \in \mathcal{N}_k$, let
\begin{align*}
  \ell_k^i := \del[1]{\ell_k^{i,y}, \ell_k^{i,p}}
    \in V_{\Gamma_i}^\star \times V_{\Gamma_i}^\star
\end{align*}
be a pair of bounded linear functionals acting on the state and adjoint traces on $\Gamma_i$. A constraint acting on one variable alone is admitted by
taking the other component to be zero for every $i \in \mathcal{N}_k$. The collection
$\cbr[0]{\del[0]{\mathcal{N}_k, (\ell_k^i)_{i \in \mathcal{N}_k}}}_{k=1}^{N_\Pi}$
is a \emph{primal constraint family} if, for every $k = 1,\dots,N_\Pi$ and all
$i,j \in \mathcal{N}_k$, the following \emph{consistency conditions} hold in $V_\Gamma^\star$:
\begin{align}\label{eq:primal_compatibility}
  \ell_k^{i,\bullet} r_{\Gamma_i}
  = \ell_k^{j,\bullet}  r_{\Gamma_j},
  \qquad \bullet \in \cbr{y,p}.
\end{align}
\end{definition}
Condition \eqref{eq:primal_compatibility} ensures the $k^{\text{th}}$ constraint is single-valued on conforming traces, and thus the conforming interface space
is contained in the partially assembled one:

\begin{proposition}[Conforming inclusion]\label{prop:conforming_inclusion}
It holds that $\iota_\Gamma(X_\Gamma) \subset \widehat{X}_\Gamma$.
\end{proposition}

\begin{proof}
Let $x_\Gamma := (y_\Gamma, p_\Gamma) \in X_\Gamma$, so that
$\widetilde{\pi}_{\Gamma_i}\iota_\Gamma x_\Gamma
 = \rho_{\Gamma_i} x_\Gamma
 = \del[1]{r_{\Gamma_i}y_\Gamma, r_{\Gamma_i}p_\Gamma}$.
For all $i,j \in \mathcal{N}_k$, \eqref{eq:primal_compatibility} gives
\begin{align*}
  \ell_k^{i,y}\del[1]{r_{\Gamma_i}y_\Gamma}
  = \langle r_{\Gamma_i}^\star \ell_k^{i,y}, y_\Gamma \rangle_{V_\Gamma}
  = \langle r_{\Gamma_j}^\star \ell_k^{j,y}, y_\Gamma \rangle_{V_\Gamma}
  = \ell_k^{j,y}\del[1]{r_{\Gamma_j}y_\Gamma},
\end{align*}
and the argument for the adjoint components is identical.
\end{proof}

The constraints provide global coupling absent from the broken operator
$\widetilde{\mathcal{S}}_\Gamma$.
In particular, the partially assembled space is the subspace of $\widetilde{X}_\Gamma$ on which
every constraint is single-valued. Let
$\widetilde{V}_\Gamma := \prod_{i=1}^N V_{\Gamma_i}$ and define the scalar partially assembled spaces for $\bullet \in \cbr{y,p}$:
\begin{align}\label{eq:partially_assembled_factors}
  \widehat{V}_\Gamma^\bullet
  := \cbr[1]{
    v \in \widetilde{V}_\Gamma :
    \ell_k^{i,\bullet}(v_i) = \ell_k^{j,\bullet}(v_j)
    \text{ for all } i,j \in \mathcal{N}_k,\;
    k = 1,\dots,N_\Pi}.
\end{align}
We define the partially assembled space as the product
\begin{align}\label{eq:partially_assembled_space}
  \widehat{X}_\Gamma := \widehat{V}_\Gamma^y \times \widehat{V}_\Gamma^p,
\end{align}
which can be identified with a subspace of $\widetilde{X}_\Gamma$ via coordinate permutation. We denote by $\widehat{\iota}_\Gamma : \widehat{X}_\Gamma \hookrightarrow \widetilde{X}_\Gamma$ the canonical embedding. Every primal constraint takes a single value on $\widehat{X}_\Gamma$ by construction. For $k = 1,\dots,N_\Pi$, we may therefore define
\begin{align}\label{eq:primal_global_functional}
  \ell_k := \del[1]{\ell_k^y, \ell_k^p} \in \del[1]{\widehat{V}_\Gamma^y}^\star \times \del[1]{\widehat{V}_\Gamma^p}^\star,
  \qquad
  \ell_k^\bullet(v) := \ell_k^{i,\bullet}(v_i)
  \quad \text{for any } i \in \mathcal{N}_k,
\end{align}
the right-hand side being independent of the choice of $i \in \mathcal{N}_k$ by \eqref{eq:partially_assembled_factors}. We call $\ell_k$, acting on $\widehat{X}_\Gamma = \widehat{V}_\Gamma^y \times \widehat{V}_\Gamma^p$ componentwise, the $k^{\text{th}}$ \emph{primal value}, shared by the subdomains $\Omega_i$, $i \in \mathcal{N}_k$. In principle, the state and adjoint primal constraints can be chosen independently. However, many of the results below require the following compatibility condition, which essentially requires that the adjoint primal constraints be at least as rich as the state primal constraints:
\begin{definition}[Compatible family]\label{def:compatible_family}
A primal constraint family is \emph{compatible} if
$\widehat{V}_\Gamma^p \subseteq \widehat{V}_\Gamma^y$; that is, if
every $v \in \widetilde{V}_\Gamma$ satisfying
$\ell_k^{i,p}(v_i) = \ell_k^{j,p}(v_j)$ for all $i,j \in \mathcal{N}_k$
and all $k = 1,\dots,N_\Pi$ also satisfies
$\ell_k^{i,y}(v_i) = \ell_k^{j,y}(v_j)$ for all $i,j \in \mathcal{N}_k$
and all $k = 1,\dots,N_\Pi$.
\end{definition}

\subsubsection{Partially assembled Schur complement}
Given arbitrary $\widehat{x},\widehat{v} \in \widehat{X}_\Gamma$,  define $(\widehat{x}_i, \widehat{v}_i) := (\widetilde{\pi}_{\Gamma_i}\widehat{\iota}_\Gamma\, \widehat{x},\widetilde{\pi}_{\Gamma_i}\widehat{\iota}_\Gamma\, \widehat{v}) $. Restricting $\widetilde{\mathcal{S}}_{\Gamma}$ to
$\widehat{X}_\Gamma$ yields the partially assembled Schur complement
$\widehat{\mathcal{S}}_\Gamma := \widehat{\iota}_\Gamma^{\; \star}\, \widetilde{\mathcal{S}}_{\Gamma}\, \widehat{\iota}_\Gamma :
\widehat{X}_\Gamma \to \widehat{X}_\Gamma^\star$, i.e.,
\begin{align}\label{eq:partial_schur}
  \langle \widehat{\mathcal{S}}_\Gamma\,\widehat{x},
          \widehat{v} \rangle_{\widehat{X}_\Gamma}
  =  \langle \widetilde{\mathcal{S}}_{\Gamma} \widehat{\iota}_\Gamma \widehat{x},
         \widehat{\iota}_\Gamma  \widehat{v} \rangle_{\widetilde{X}_\Gamma} = \sum_{i=1}^N
    \langle \mathcal{S}_{\Gamma_i}\widehat{x}_i,
            \widehat{v}_i
    \rangle_{X_{\Gamma_i}},
  \qquad \forall\, \widehat{x}, \widehat{v} \in \widehat{X}_\Gamma.
\end{align}

\begin{theorem}
\label{thm:partial_fredholm}
Let $\Lambda_\Pi$ be a family of primal constraints, $\widehat{X}_{\Gamma}$ the corresponding partially assembled trace space, and $\widehat{S}_{\Gamma}: \widehat{X}_{\Gamma} \to \widehat{X}_{\Gamma}^\star$ the partially assembled Schur complement operator. The following statements hold:
\begin{enumerate}[label=\textup{(\alph*)}]
\item $\widehat{X}_\Gamma$ is a closed subspace of
  $\widetilde{X}_\Gamma$ of finite codimension;
\item $\widehat{\mathcal{S}}_\Gamma : \widehat{X}_\Gamma \to
  \widehat{X}_\Gamma^\star$ is Fredholm with
  $\operatorname{ind}(\widehat{\mathcal{S}}_\Gamma) = 0$, and it holds that 
  \begin{align}\label{eq:kernel_inclusion}
    \widehat{X}_\Gamma \cap \ker(\widetilde{\mathcal{S}}_\Gamma)
      \subseteq \ker(\widehat{\mathcal{S}}_\Gamma).
  \end{align}
  \item If the primal constraint family is compatible in the sense of \Cref{def:compatible_family}, then in fact
\begin{align} \label{eq:ker_equiv}
   \widehat{X}_\Gamma  \cap \operatorname{ker}(\widetilde{\mathcal{S}}_\Gamma) = \operatorname{ker}(\widehat{\mathcal{S}}_\Gamma).
\end{align}
\end{enumerate}
    Consequently, if the primal constraint family is compatible in the sense of \Cref{def:compatible_family}, then
$\widehat{\mathcal{S}}_\Gamma$ is an isomorphism if and only if
    $\widehat{X}_\Gamma \cap \operatorname{ker}(\widetilde{\mathcal{S}}_\Gamma)
    = \cbr{0}$.
\end{theorem}

\begin{proof}
(a) For $\bullet\in\{y,p\}$, collect the finitely many constraint jumps
in a bounded map
$J_\bullet:\widetilde V_\Gamma\to\mathbb R^{m_\bullet}$, with $m_\bullet$ the number of jump conditions. Then
$\widehat V_\Gamma^\bullet=\ker(J_\bullet)$, so each factor is closed
and of finite codimension. Hence
$\widehat X_\Gamma=\widehat V_\Gamma^y\times
\widehat V_\Gamma^p$ is a closed finite-codimensional subspace of
$\widetilde X_\Gamma$.

\medskip (b)
By part (a) and \Cref{thm:broken_fredholm,prop:finite_rank_perturb},
$\widehat{\mathcal S}_\Gamma$ is Fredholm of index zero. If
$\widehat x\in\widehat X_\Gamma$ and
$\widehat\iota_\Gamma\widehat x\in
\ker(\widetilde{\mathcal S}_\Gamma)$, then for every
$\widehat v\in\widehat X_\Gamma$,
\begin{align*}
  \langle\widehat{\mathcal S}_\Gamma\widehat x,\widehat v\rangle
  =\langle\widetilde{\mathcal S}_\Gamma
    \widehat\iota_\Gamma\widehat x,
    \widehat\iota_\Gamma\widehat v\rangle=0,
\end{align*}
which proves \eqref{eq:kernel_inclusion}.

\medskip (c)
Assume now that the primal family is compatible and let
$\widehat\xi=(\widehat\xi^y,\widehat\xi^p)\in\ker(\widehat{\mathcal S}_\Gamma)$.
For each $i$, set $\widehat\xi_i:=\widetilde{\pi}_{\Gamma_i}\widehat\iota_\Gamma\widehat\xi$
and $(y_i,p_i,u_i,\mu_i):=\mathcal H_i\widehat\xi_i$. Since
$\widehat X_\Gamma=\widehat V_\Gamma^y\times\widehat V_\Gamma^p$ is a product of
linear spaces, the vector $(\widehat\xi^y,-\widehat\xi^p)$ again lies in
$\widehat X_\Gamma$ and is therefore an admissible test in
$\langle\widehat{\mathcal S}_\Gamma\widehat\xi,\cdot\rangle_{\widehat X_\Gamma}$.
Testing $\widehat{\mathcal S}_\Gamma\widehat\xi$ against it and applying the local harmonic
identity~\Cref{lem:local_harmonic}(d), together with the self-adjointness of
$\widetilde A_i$, yields
\begin{align*}
  0=\langle\widehat{\mathcal S}_\Gamma\widehat\xi,(\widehat\xi^y,-\widehat\xi^p)\rangle_{\widehat X_\Gamma}
   =\sum_{i=1}^N
    \bigl(\langle L_i y_i,y_i\rangle_{V_i}+\alpha\|u_i\|_{U_i}^2\bigr).
\end{align*}
Since $\langle L_i y_i,y_i\rangle_{V_i}\ge\|M_i y_i\|_{L^2(\Omega_i)}^2$ and
$M_i$ is injective, $y_i=0$ and $u_i=0$ for every $i=1,\dots,N$.
Compatibility (\Cref{def:compatible_family}) gives
$\widehat\xi^p\in\widehat V_\Gamma^y$, so $(\widehat\xi^p,0)\in\widehat X_\Gamma$
is also an admissible test. Testing $\widehat{\mathcal S}_\Gamma\widehat\xi$ against it and using
$y_i=0$ in~\Cref{lem:local_harmonic}(d),
\begin{align*}
  0=\langle\widehat{\mathcal S}_\Gamma\widehat\xi,(\widehat\xi^p,0)\rangle_{\widehat X_\Gamma}
   =\sum_{i=1}^N\langle\widetilde A_i^\star p_i,p_i\rangle_{V_i},
\end{align*}
and since $\widetilde A_i^\star$ is positive semidefinite,
$\widetilde A_i^\star p_i=0$ for every $i$. With $y_i=u_i=0$ and
$\widetilde A_i^\star p_i=0$, all four rows of
$\mathcal K_i\mathcal H_i\widehat\xi_i$ vanish, so
$\mathcal S_{\Gamma_i}\widehat\xi_i=0$ by
$\mathcal K_i\mathcal H_i=\tau_i^\star\mathcal S_{\Gamma_i}$
(\Cref{lem:local_harmonic}(d)) and the injectivity of $\tau_i^\star$.
Hence $\widetilde{\mathcal S}_\Gamma\widehat\iota_\Gamma\widehat\xi=0$, which
proves the reverse inclusion in~\eqref{eq:ker_equiv}. The result follows.
\end{proof}
We close this subsection with a sufficient condition on the family of primal
constraints, phrased in terms of the connectivity of a graph on the subdomains,
that ensures that $\widehat{\mathcal{S}}_\Gamma$ is an isomorphism.

\begin{definition}[Constraint graph]\label{def:constraint_graph}
For $k = 1,\dots,N_\Pi$, set $m_k := \ell_k^{p}\del[0]{1|_{\Gamma}}$. The
\emph{constraint graph} $\mathcal{G}_\Pi$ is the undirected graph with vertex
set $\cbr{1,\dots,N}$ in which two distinct vertices $i$ and $j$ are adjacent
whenever there is a $k \in \cbr{1,\dots,N_\Pi}$ with $i,j \in \mathcal{N}_k$ and
$m_k \ne 0$.
\end{definition}

\begin{theorem}[Graph condition]\label{thm:graph_property}
Let $\mathcal{N}_*$ be as in \Cref{thm:broken_fredholm}\textup{(b)} and set
$\mathcal{N}_0 := \cbr{1,\dots,N} \setminus \mathcal{N}_*$. If every connected
component of $\mathcal{G}_\Pi$ contains a vertex in $\mathcal{N}_0$, then
\begin{align*}
  \widehat{X}_\Gamma \cap \operatorname{ker}(\widetilde{\mathcal{S}}_\Gamma)
  = \cbr{0}.
\end{align*}
Consequently, if the primal constraint family is compatible in the sense of \Cref{def:compatible_family}, then
$\widehat{\mathcal{S}}_\Gamma : \widehat{X}_\Gamma \to \widehat{X}_\Gamma^\star$
is an isomorphism.
\end{theorem}

\begin{proof}
Let $\widehat{\xi} \in \widehat{X}_\Gamma \cap
\operatorname{ker}(\widetilde{\mathcal{S}}_\Gamma)$ and set
$\widehat{\xi}_i := (\widehat{\xi}_i^{\, y}, \widehat{\xi}_i^{\, p})
:= \widetilde{\pi}_{\Gamma_i} \widehat{\iota}_\Gamma \widehat{\xi}$.
By \Cref{thm:broken_fredholm} and \Cref{thm:local_fredholm}, there exist
$c_i \in \mathbb{R}$ with
\begin{align*}
  \widehat{\xi}_i = \del{0,\, c_i \cdot 1|_{\Gamma_i}},
  \qquad i = 1,\dots,N,
\end{align*}
and $c_i = 0$ for every $i \in \mathcal{N}_0$.
Suppose $i$ and $j$ are adjacent in $\mathcal{G}_\Pi$. By
\Cref{def:constraint_graph}, there is a $k$ with $i,j \in \mathcal{N}_k$ and
$m_k \ne 0$. Since $1|_{\Gamma_i} = r_{\Gamma_i}\del[0]{1|_{\Gamma}}$ and
$1|_{\Gamma_j} = r_{\Gamma_j}\del[0]{1|_{\Gamma}}$,
\eqref{eq:primal_global_functional} gives
$\ell_k^{i,p}\del[0]{1|_{\Gamma_i}} = \ell_k^{j,p}\del[0]{1|_{\Gamma_j}} = m_k$.
 As
$\widehat{\xi} \in \widehat{X}_\Gamma$, \eqref{eq:partially_assembled_factors}
therefore yields
\begin{align*}
  m_k\, c_i
  = \ell_k^{i,p}\del[0]{\widehat{\xi}_i^{\, p}}
  = \ell_k^{j,p}\del[0]{\widehat{\xi}_j^{\, p}}
  = m_k\, c_j,
\end{align*}
and $m_k \ne 0$ yields $c_i = c_j$. The coefficients are therefore constant on
every connected component of $\mathcal{G}_\Pi$. Each component contains a vertex
in $\mathcal{N}_0$, at which the coefficient vanishes, so $c_i = 0$ for every
$i = 1,\dots,N$ and hence $\widehat{\xi} = 0$.
If the primal constraint family is compatible, \Cref{thm:partial_fredholm} yields that
$\widehat{\mathcal{S}}_\Gamma$ is an isomorphism.
\end{proof}

\subsection{The BDDC preconditioner}
\label{ss:preconditioner}
With the partially assembled Schur complement operator in hand, we are now ready to give the definition of the BDDC preconditoner.
Suppose $\widehat{\mathcal{S}}_\Gamma: \widehat{X}_\Gamma \to \widehat{X}_\Gamma^\star$ is an isomorphism. By \Cref{prop:conforming_inclusion}, there is a bounded $R_\Gamma : X_\Gamma \to \widehat{X}_\Gamma$ with $\widehat{\iota}_\Gamma  R_\Gamma = \iota_\Gamma$, taking a conforming interface trace to the corresponding partially assembled trace.

\begin{definition}[The BDDC preconditioner] \label{def:bddc_preconditioner}
Let $\mathcal{W} : \widehat{X}_\Gamma \to X_\Gamma$ be a bounded linear operator satisfying
\begin{align} \label{eq:W_projection}
    \mathcal{W} R_\Gamma = I_{X_\Gamma}.
\end{align}
The BDDC preconditioner $\mathcal{P}_{\mathrm{BDDC}} : X_\Gamma^\star \to X_\Gamma$ is defined as the bounded linear operator
\begin{align} \label{eq:BDDC_preconditioner}
\mathcal{P}_{\mathrm{BDDC}} := \mathcal{W}\, \widehat{\mathcal{S}}_\Gamma^{-1}\, \mathcal{W}^\star.
\end{align}
\end{definition}

Since the partially assembled Schur complement operator $\widehat{\mathcal{S}}_{\Gamma}$ is indefinite, it is not automatic that the BDDC operator $\mathcal{P}_{\mathrm{BDDC}}$ is injective.  Thus, left preconditioning with $\mathcal{P}_{\mathrm{BDDC}}$ is potentially problematic. This is in stark contrast with BDDC methods for typical positive-definite problems. However, under a mild assumption on $\mathcal{W}$, injectivity is guaranteed as the following result will show.

\begin{proposition}\label{prop:bddc_injective}
Assume that
$\widehat{\mathcal S}_\Gamma:\widehat X_\Gamma\to
\widehat X_\Gamma^\star$ is an isomorphism and that
\begin{align*}
  \mathcal W(\widehat\xi^y,\widehat\xi^p)
  =(\mathcal W_V^y\widehat\xi^y,
    \mathcal W_V^p\widehat\xi^p),
\end{align*}
where $\mathcal{W}_V^y: \widehat{V}_\Gamma^y \to V_\Gamma$, $\mathcal{W}_V^p : \widehat{V}_\Gamma^p \to V_\Gamma$ are bounded linear operators with $\operatorname{ker}(\mathcal{W}_V^p) \subseteq \operatorname{ker}(\mathcal{W}_V^y)$.
Then
$\mathcal P_{\mathrm{BDDC}}
=\mathcal W\widehat{\mathcal S}_\Gamma^{-1}\mathcal W^\star$ is
injective. Consequently,
\begin{align*}
  \mathcal S_\Gamma x_\Gamma=g_\Gamma
  \quad\Longleftrightarrow\quad
  \mathcal P_{\mathrm{BDDC}}
  (\mathcal S_\Gamma x_\Gamma-g_\Gamma)=0.
\end{align*}
\end{proposition}
\begin{proof}
  Consider the following saddle-point operator defined by 
    \begin{align*}
      \mathcal{M} := \begin{bmatrix}
    -\widehat{\mathcal{S}}_\Gamma & \mathcal{W}^\star \\
        \mathcal{W} & 0 
        \end{bmatrix} : \widehat{X}_\Gamma \times X_\Gamma^\star \to \widehat{X}_\Gamma^\star \times X_\Gamma,
    \end{align*}
and observe that $\mathcal{P}_{\mathrm{BDDC}}$ is the Schur complement of $\mathcal{M}$ with respect to the $\widehat{\mathcal{S}}_\Gamma$-block. Hence by \Cref{prop:schur_fredholm}, 
  $ \ker(\mathcal{P}_{\mathrm{BDDC}})$ and $\ker(\mathcal{M})$ are isomorphic. Therefore, we must show $\ker(\mathcal{M}) = \cbr{0}$. 

  To this end, it suffices to show injectivity of $\iota_K^\star \widehat{\mathcal{S}}_\Gamma \iota_K $, where $\iota_K : \operatorname{ker}(\mathcal{W}) \hookrightarrow \widehat{X}_\Gamma$ is the canonical embedding. Indeed, suppose $\iota_K^\star \widehat{\mathcal{S}}_\Gamma \iota_K $ is injective and  $(\widehat{\xi},g) \in \operatorname{ker}(\mathcal{M})$, 
  \begin{subequations}
\begin{align}\label{eq:BDDC_injective.1}
-\widehat{\mathcal{S}}_\Gamma \widehat \xi + \mathcal{W}^\star g &= 0, \quad \text{in } \widehat{X}_\Gamma^\star, \\
      \mathcal{W} \widehat \xi &= 0, \quad \text{in } X_\Gamma. \label{eq:BDDC_injective.2}
  \end{align}
  \end{subequations}
By \eqref{eq:BDDC_injective.2}, $\widehat{\xi} \in \operatorname{ker}(\mathcal{W})$. Therefore, for all $\widehat{\eta} \in \operatorname{ker}(\mathcal{W})$, \eqref{eq:BDDC_injective.1} yields
\begin{align}
    \langle \iota_K^\star \widehat{\mathcal{S}}_\Gamma \iota_K \widehat \xi, \widehat \eta \rangle_{\operatorname{ker}(\mathcal{W})} = \langle \widehat{\mathcal{S}}_\Gamma \iota_K \widehat \xi, \iota_K \widehat \eta \rangle_{\widehat{X}_\Gamma} = \langle \mathcal{W}^\star g, \widehat \eta \rangle_{\widehat{X}_\Gamma} = 0,
\end{align}
from which we deduce that $\iota_K^\star \widehat{\mathcal{S}}_\Gamma \iota_K \widehat \xi = 0$, and thus $\widehat \xi = 0$. Since $\mathcal{W}$ is surjective, $\mathcal{W}^\star$ is injective, and \eqref{eq:BDDC_injective.1} then yields $g = 0$, so $\operatorname{ker}(\mathcal{M}) = \cbr{0}$.

It remains to prove the injectivity of $\iota_K^\star \widehat{\mathcal{S}}_\Gamma \iota_K$. By assumption, $\operatorname{ker}(\mathcal{W}) = \operatorname{ker}(\mathcal{W}_V^y)\times \operatorname{ker}(\mathcal{W}_V^p)$. Suppose $\widehat{\xi} = (\widehat{\xi}^y,\widehat{\xi}^p)\in \operatorname{ker}(\mathcal{W}_V^y)\times \operatorname{ker}(\mathcal{W}_V^p)$ and $\iota_K^\star \widehat{\mathcal{S}}_\Gamma \iota_K \widehat{\xi} = 0$. Since both $(\widehat{\xi}^y,-\widehat{\xi}^p)\in \operatorname{ker}(\mathcal{W}_V^y)\times \operatorname{ker}(\mathcal{W}_V^p)$ and $(\widehat{\xi}^p,0)\in \operatorname{ker}(\mathcal{W}_V^p)\times \operatorname{ker}(\mathcal{W}_V^p) \subseteq \operatorname{ker}(\mathcal{W}_V^y)\times \operatorname{ker}(\mathcal{W}_V^p)$ are admissible tests against $\iota_K^\star \widehat{\mathcal{S}}_\Gamma \iota_K \widehat{\xi}$, we can repeat the same argument as in the proof of \Cref{thm:partial_fredholm} to deduce that $\widehat{\mathcal{S}}_\Gamma \widehat{\xi} = 0$. Since $\widehat{\mathcal{S}}_\Gamma$ is an isomorphism, $\widehat{\xi} = 0$. The result follows.
\end{proof}

\begin{remark}
    While we have intentionally avoided placing further constraints on the operator $\mathcal{W}$, the choice of $\mathcal{W}$ can significantly impact the quality of the BDDC preconditioner. In the discrete setting, \eqref{eq:W_projection} reduces to the usual partition-of-unity property of the interface scaling; some common choices appearing in the literature are discussed below in \Cref{sss:discrete_BDDC}.
\end{remark}

The preceding definition characterizes the BDDC preconditioner entirely in
terms of the partially assembled interface problem. At this level, however,
its application still requires the solution of the globally coupled problem
associated with the partially assembled Schur complement operator $\widehat{\mathcal{S}}_\Gamma$. In the following subsection, we show
how this solve can be reduced to independent local subproblems supplemented by
a correction obtained from a finite-dimensional coarse problem, thereby
realizing the BDDC preconditioner as a two-level additive Schwarz method.

\subsection{The primal-dual splitting}
\label{sss:primal-dual_constraint}
Each partially assembled interface trace will be separated into a component (primal) representing the common state and adjoint interface quantities imposed by the primal constraints and a component (dual) on which all primal constraints vanish. The primal component provides the coupling between subdomains, whereas the dual component decomposes into independent local contributions. The dual variables can therefore be eliminated from the partially assembled Schur complement problem by local solves, analogously to the static condensation of the interior variables in \Cref{sec:dd}, leaving a finite-dimensional Schur complement problem for the primal variables. Here, \emph{dual} is used in the standard BDDC sense and is unrelated to the topological dual.

To this end, for each $i=1,\dots,N$ we denote $\mathcal{N}^i := \{ k : i \in \mathcal{N}_k \}$ for the set of primal
constraint indices corresponding to $\Omega_i$ and set $N_\Pi^i := \text{card} (\mathcal{N}^i)$. The values of the primal
constraints are collected
by the bounded operators
\begin{align}\label{eq:CPi_def}
C_\Pi &: \widehat{X}_\Gamma \to \mathbb{R}^{2N_\Pi},
\qquad
C_\Pi(\widehat{x}) := \big(\ell_k(\widehat{x})\big)_{k=1}^{N_\Pi}, \\
\label{eq:CPi_local_def}
C^i_\Pi &: X_{\Gamma_i} \to \mathbb{R}^{2N_{\Pi}^i},
\qquad
C^i_\Pi(\xi) := \big(\ell^i_k(\xi)\big)_{k \in \mathcal{N}^i}.
\end{align}
The dual subspace, its local factors, and the primal subspace, respectively, are defined as
\begin{equation}\label{eq:pd_split}
\widehat{X}_\Delta := \ker(C_\Pi),
\qquad
\widehat{X}^i_\Delta := \ker(C^i_\Pi) \subset X_{\Gamma_i},
\qquad
\widehat{X}_\Pi := \widehat{X}_\Delta^{\perp},
\end{equation}
the orthogonal complement taken with respect to the inner product of
$\widehat{X}_\Gamma$. We write
$\iota_\Pi : \widehat{X}_\Pi \hookrightarrow \widehat{X}_\Gamma$,
$\iota_\Delta : \widehat{X}_\Delta \hookrightarrow \widehat{X}_\Gamma$ for
the canonical inclusions and $\pi_\Pi : \widehat{X}_\Gamma \to  \widehat{X}_\Pi$, $\pi_\Delta:  \widehat{X}_\Gamma \to  \widehat{X}_\Delta$ for the orthogonal projections,
so that 
\begin{align} \label{eq:primal_dual_decomposition}
\widehat{X}_\Gamma = \widehat{X}_\Pi \oplus \widehat{X}_\Delta.
\end{align}

Analogously to the developments in \Cref{subsec:global_decomposition}, we now derive a block representation of the partially assembled Schur complement operator $\widehat{\mathcal{S}}_\Gamma$ relative to the splitting \eqref{eq:primal_dual_decomposition}. To this end, we define the mapping
\begin{align} \label{eq:JPiDelta_def}
    J_{\Pi\Delta} : \widehat{X}_\Pi \times \widehat{X}_\Delta \to \widehat{X}_\Gamma, \qquad J_{\Pi\Delta}(x_\Pi, x_\Delta) := \iota_\Pi x_\Pi + \iota_\Delta x_\Delta,
\end{align}
which is an isometric isomorphism with $J_{\Pi\Delta}^{-1} \widehat{x} = (\pi_\Pi \widehat{x}, \pi_\Delta \widehat{x})$, and let $\Phi_{\Pi\Delta} : \widehat{X}_\Pi^\star \times \widehat{X}_\Delta^\star \to \del[0]{\widehat{X}_\Pi \times \widehat{X}_\Delta}^\star$ be the canonical identification (\emph{cf}. \Cref{rem:notation}, so that $\Phi_{\Pi\Delta}^{-1} J_{\Pi\Delta}^\star f = \del[0]{\iota_\Pi^\star f, \iota_\Delta^\star f}$. Relative to this splitting,
\begin{align}\label{eq:Shat_blocks}
  \Phi_{\Pi\Delta}^{-1}\, J_{\Pi\Delta}^\star\, \widehat{\mathcal{S}}_\Gamma\, J_{\Pi\Delta}
  = \begin{bmatrix}
      \widehat{\mathcal{S}}_{\Pi\Pi} & \widehat{\mathcal{S}}_{\Pi\Delta} \\
      \widehat{\mathcal{S}}_{\Delta\Pi} & \widehat{\mathcal{S}}_{\Delta\Delta}
    \end{bmatrix},
  \qquad
  \widehat{\mathcal{S}}_{\bullet\circ} := \iota_\bullet^\star\, \widehat{\mathcal{S}}_\Gamma\, \iota_\circ,
  \quad \bullet, \circ \in \cbr{\Pi,\Delta}.
\end{align}
Note that $\widehat{\mathcal{S}}_{\Pi\Delta} = \widehat{\mathcal{S}}_{\Delta\Pi}^\star$ and since $\widehat{\mathcal{S}}_\Gamma$ is self-adjoint, so are $\widehat{\mathcal{S}}_{\Pi \Pi}$ and $\widehat{\mathcal{S}}_{\Delta\Delta}$.

Arguing similarly as for the interior space $X_I$ in \Cref{sss:local_product_spaces}, we identify the dual subspace with the product of its local factors. Define the \emph{broken dual subspace}
\begin{align} \label{eq:broken_dual_space}
    \widetilde{X}_\Delta := \prod_{i=1}^N \widehat{X}_\Delta^i,
\end{align}
equipped with the product norm and with coordinate projections $\widetilde{\pi}_{\Delta}^i : \widetilde{X}_\Delta \to \widehat{X}_\Delta^i$ and injections $\widetilde{\jmath}_{\Delta}^{\, i} : \widehat{X}_\Delta^i \to \widetilde{X}_\Delta$, and let $\iota_\Delta^i : \widehat{X}_\Delta^i \hookrightarrow X_{\Gamma_i}$ denote the canonical inclusions.
\begin{lemma}[Decomposition of the dual subspace] \label{lem:dual_subspace_decomp}
The map
\begin{align*}
    \mathcal{B}_\Delta : \widehat{X}_\Delta \to \widetilde{X}_\Delta, \qquad \mathcal{B}_\Delta \widehat{x} := \del[1]{\widetilde{\pi}_{\Gamma_i} \widehat{\iota}_\Gamma \iota_\Delta \widehat{x}}_{i=1}^N,
\end{align*}
is well defined and an isometric isomorphism. Moreover, the maps
\begin{align} \label{eq:dual_coordinate_maps}
    \pi_{\Delta}^i := \widetilde{\pi}_{\Delta}^i \mathcal{B}_\Delta : \widehat{X}_\Delta \to \widehat{X}_\Delta^i, \qquad
    \jmath_\Delta^i := \mathcal{B}_\Delta^{-1} \widetilde{\jmath}_{\Delta}^{\, i} : \widehat{X}_\Delta^i \to \widehat{X}_\Delta,
\end{align}
satisfy the following identities:
\begin{align} \label{eq:dual_ambient_identities}
    \iota_\Delta^i \pi_{\Delta}^i = \widetilde{\pi}_{\Gamma_i} \widehat{\iota}_\Gamma \iota_\Delta, \qquad
    \widehat{\iota}_\Gamma \iota_\Delta \jmath_\Delta^i = \jmath_{\Gamma_i} \iota_\Delta^i, \qquad i = 1,\dots,N.
\end{align}
\end{lemma}
\begin{proof}
After verifying its hypotheses, we aim to apply \Cref{prop:homogeneous_decomp} with the choices $\mathcal{X} := \widehat{X}_\Gamma$, $\mathcal{X}_i := X_{\Gamma_i}$, $\widetilde{\mathcal{X}} := \widetilde{X}_\Gamma$, $\mathsf{C}_i := C_\Pi^i$, and $\mathsf{R} := \widehat{\iota}_\Gamma$.
To this end, note that $\mathsf{R}$ is an isometry, since $\widehat{X}_\Gamma$ carries the norm inherited from $\widetilde{X}_\Gamma$.
Moreover, every primal constraint is single-valued on $\widehat{X}_\Gamma$ by \eqref{eq:partially_assembled_factors}, so that $\ell_k(\widehat{x}) = \ell_k^i\del[1]{\pi_{\Gamma_i} \widehat{\iota}_\Gamma \widehat{x}}$ for each $\widehat{x} \in \widehat{X}_\Gamma$, each $k = 1,\dots,N_\Pi$, and each $i \in \mathcal{N}_k$. Hence, by \eqref{eq:CPi_local_def}, if $\widehat{x} \in \widehat{X}_\Delta$, then $\mathsf{C}_i \widetilde{\pi}_i \mathsf{R} \widehat{x} = C_\Pi^i \del[1]{\pi_{\Gamma_i} \widehat{\iota}_\Gamma \widehat{x}} = 0$ for $i = 1,\dots,N$. Conversely, suppose $\widehat{x} \in \widehat{X}_\Gamma$ is such that
$\mathsf{C}_i \widetilde{\pi}_i \mathsf{R} \widehat{x} = 0$ for $i = 1,\dots,N$,
and fix $k \in \cbr{1,\dots,N_\Pi}$.
Since $C_\Pi^i$ collects only the constraints indexed by $\mathcal{N}^i$, we
choose $i \in \mathcal{N}_k$ so
that $k \in \mathcal{N}^i$ (which is possible as $|\mathcal{N}_k| \ge 2$).
By \eqref{eq:CPi_local_def}, the value
$\ell_k^i\del[1]{\pi_{\Gamma_i} \widehat{\iota}_\Gamma \widehat{x}}$ is one of the
components of $C_\Pi^i\del[1]{\pi_{\Gamma_i} \widehat{\iota}_\Gamma \widehat{x}}$,
which vanishes by assumption, and
\eqref{eq:primal_global_functional} then gives $\ell_k\del[0]{\widehat{x}} = 0$.
As $k$ was arbitrary, every component of $C_\Pi \widehat{x}$ vanishes by
\eqref{eq:CPi_def}, so $\widehat{x} \in \ker(C_\Pi) = \widehat{X}_\Delta$
by \eqref{eq:pd_split}.

Finally, we verify the gluing condition \eqref{eq:gluing_condition}.
Let $\widetilde{x} = (x_i)_{i=1}^N \in \widetilde{X}_\Delta$ and write
$x_i = \del[1]{x_i^y, x_i^p} \in X_{\Gamma_i}$ for $i = 1,\dots,N$.
By \eqref{eq:pd_split}, $x_i \in \ker(C_\Pi^i)$ for $i = 1,\dots,N$, so
\eqref{eq:CPi_local_def} yields
\begin{align*}
  \ell_k^{i,\bullet}\del[1]{x_i^\bullet} = 0,
  \qquad i = 1,\dots,N, \quad k \in \mathcal{N}^i, \quad \bullet \in \cbr{y,p}.
\end{align*}
Now let $k \in \cbr{1,\dots,N_\Pi}$ and $i,j \in \mathcal{N}_k$.
Then $k \in \mathcal{N}^i \cap \mathcal{N}^j$, so
$\ell_k^{i,\bullet}\del[1]{x_i^\bullet} = 0 = \ell_k^{j,\bullet}\del[1]{x_j^\bullet}$
for $\bullet \in \cbr{y,p}$.
Since $k$ and $i,j \in\mathcal{N}_k$ were arbitrary, $\widetilde{x} \in \widehat{X}_\Gamma$ by
\eqref{eq:partially_assembled_factors}.
Since $\mathsf{R}$ is the canonical embedding of $\widehat{X}_\Gamma$ into
$\widetilde{X}_\Gamma$, it holds that
$\operatorname{ran}(\mathsf{R}) = \widehat{X}_\Gamma$ and hence
$\widetilde{X}_\Delta \subseteq \operatorname{ran}(\mathsf{R})$.
The assertions now follow from \Cref{prop:homogeneous_decomp}.
\end{proof}
Via $\mathcal{B}_\Delta$, the operator $\widehat{\mathcal{S}}_{\Delta\Delta}$ is congruent to the direct sum of the local blocks
\begin{align} \label{eq:local_dual_subspace_blocks}
     \widehat{\mathcal{S}}_{\Delta\Delta}^i := (\iota_\Delta^i)^\star \mathcal{S}_{\Gamma_i} \iota_\Delta^i : \widehat{X}_\Delta^i \to (\widehat{X}_\Delta^i)^\star, \qquad i = 1,\dots,N.
\end{align}
\color{black}
More precisely, the following result holds:
\begin{proposition}[Block-diagonality of $\widehat{\mathcal{S}}_{\Delta\Delta}$]
\label{prop:SDD_diagonal}
It holds that
\begin{align} \label{eq:SDD_decomp}
  (\mathcal{B}_\Delta^{-1})^\star \widehat{\mathcal{S}}_{\Delta\Delta} \mathcal{B}_\Delta^{-1} = \bigoplus_{i=1}^N \widehat{\mathcal{S}}_{\Delta\Delta}^i.
\end{align}
Consequently, $\widehat{\mathcal{S}}_{\Delta\Delta}$ is an isomorphism if and only if every local block $\widehat{\mathcal{S}}_{\Delta\Delta}^i$, $i = 1,\dots,N$, is an isomorphism, in which case
\begin{align} \label{eq:SDD_inv_decomp}
  \widehat{\mathcal{S}}_{\Delta\Delta}^{-1} = \sum_{i=1}^N \jmath_\Delta^i \, (\widehat{\mathcal{S}}_{\Delta\Delta}^i)^{-1} (\jmath_\Delta^i)^\star.
\end{align}
\end{proposition}
\begin{proof}
Apply \Cref{cor:homogeneous_block} with $\mathcal{X} = \widehat{X}_\Gamma$, $\mathcal{X}_i = X_{\Gamma_i}$, $\mathsf{R} = \widehat{\iota}_\Gamma$, and $\mathsf{C}_i = C_\Pi^i$, so that $\mathcal{X}_0 = \widehat{X}_\Delta$, $\mathcal{X}_0^i = \widehat{X}_\Delta^i$, $\widetilde{\mathcal{X}}_0 = \widetilde{X}_\Delta$, $\iota_0 = \iota_\Delta$, $\iota_0^i = \iota_\Delta^i$, $\jmath_0^i = \jmath_\Delta^i$, $\mathcal{B}_0 = \mathcal{B}_\Delta$, and $\mathsf{T}_i := \mathcal{S}_{\Gamma_i}$. 
\end{proof}

\subsubsection{The coarse basis and block factorization} \label{sss:LDU_factor}

We now eliminate the dual component via static condensation. Throughout
\Cref{sss:LDU_factor} we assume that
$\widehat{\mathcal{S}}_{\Delta\Delta}$ is an isomorphism; sufficient
conditions are established in
\Cref{sss:SDD_wellposed} below. Eliminating the dual component in
\eqref{eq:Shat_blocks} as in \eqref{eq:xI_solve}--\eqref{eq:Schur_complement_E}
yields the \emph{coarse Schur complement}
\begin{align}\label{eq:S0_def}
  \widehat{\mathcal{S}}_0
  := \widehat{\mathcal{S}}_{\Pi\Pi}
   - \widehat{\mathcal{S}}_{\Pi\Delta}\,
     \widehat{\mathcal{S}}_{\Delta\Delta}^{-1}\,
     \widehat{\mathcal{S}}_{\Delta\Pi}
  : \widehat{X}_\Pi \to \widehat{X}_\Pi^\star.
\end{align}
 The elimination is realized by the following operator, the coarse space counterpart of the harmonic extension (\emph{cf}. \Cref{sec:dd}).

\begin{definition}[Coarse basis operator]\label{def:coarse_basis}
The \emph{coarse basis operator}
$\mathcal{H}_\Pi : \widehat{X}_\Pi \to \widehat{X}_\Gamma$ is
\begin{align}\label{eq:Phi_def}
  \mathcal{H}_\Pi := \iota_\Pi - \iota_\Delta\, \widehat{\mathcal{S}}_{\Delta\Delta}^{-1}\, \widehat{\mathcal{S}}_{\Delta\Pi}.
\end{align}
\end{definition}

\begin{proposition}\label{prop:Phi_props}
It holds that $\pi_\Pi \mathcal{H}_\Pi = I_{\widehat{X}_\Pi}$,
$\iota_\Delta^\star \widehat{\mathcal{S}}_\Gamma \mathcal{H}_\Pi = 0$, and
$\operatorname{ran}(\mathcal{H}_\Pi) = \operatorname{ker}\del[0]{\iota_\Delta^\star \widehat{\mathcal{S}}_\Gamma}$.
Moreover,
$\widehat{\mathcal{S}}_0 = \mathcal{H}_\Pi^\star\, \widehat{\mathcal{S}}_\Gamma\, \mathcal{H}_\Pi$,
and $\widehat{\mathcal{S}}_0$ is self-adjoint.
\end{proposition}
\begin{proof}
The first identity follows from $\pi_\Pi \iota_\Pi = I_{\widehat{X}_\Pi}$ and
$\pi_\Pi \iota_\Delta = 0$, the second from
$\iota_\Delta^\star \widehat{\mathcal{S}}_\Gamma \mathcal{H}_\Pi = \widehat{\mathcal{S}}_{\Delta\Pi} - \widehat{\mathcal{S}}_{\Delta\Delta} \widehat{\mathcal{S}}_{\Delta\Delta}^{-1} \widehat{\mathcal{S}}_{\Delta\Pi}$.
Conversely, if
$\iota_\Delta^\star \widehat{\mathcal{S}}_\Gamma \widehat{x} = \widehat{\mathcal{S}}_{\Delta\Pi} \pi_\Pi \widehat{x} + \widehat{\mathcal{S}}_{\Delta\Delta} \pi_\Delta \widehat{x} = 0$,
then
$\pi_\Delta \widehat{x} = -\widehat{\mathcal{S}}_{\Delta\Delta}^{-1} \widehat{\mathcal{S}}_{\Delta\Pi} \pi_\Pi \widehat{x}$,
i.e.\ $\widehat{x} = \mathcal{H}_\Pi \pi_\Pi \widehat{x}$. The last claims
follow by expanding
$\mathcal{H}_\Pi^\star \widehat{\mathcal{S}}_\Gamma \mathcal{H}_\Pi$ with
\eqref{eq:Phi_def}, using
$\widehat{\mathcal{S}}_{\Pi\Delta} = \widehat{\mathcal{S}}_{\Delta\Pi}^\star$
and the self-adjointness of $\widehat{\mathcal{S}}_{\Delta\Delta}^{-1}$.
\end{proof}

\begin{theorem}[Block factorization]\label{thm:LDU}
Suppose $\widehat{\mathcal{S}}_{\Delta\Delta}$ is an isomorphism. Then the
operator
\begin{align}\label{eq:JH_def}
  J_{\mathcal{H}_\Pi} : \widehat{X}_\Pi \times \widehat{X}_\Delta \to \widehat{X}_\Gamma,
  \qquad
  J_{\mathcal{H}_\Pi}(x_\Pi, x_\Delta) := \mathcal{H}_\Pi x_\Pi + \iota_\Delta x_\Delta,
\end{align}
is an isomorphism, with inverse
$J_{\mathcal{H}_\Pi}^{-1}\widehat{x}
 = \del[1]{\pi_\Pi \widehat{x},\; \pi_\Delta \widehat{x} + \widehat{\mathcal{S}}_{\Delta\Delta}^{-1}\widehat{\mathcal{S}}_{\Delta\Pi}\pi_\Pi \widehat{x}}$,
and the partially assembled Schur complement is congruent to a block-diagonal
operator:
\begin{align}\label{eq:LDU}
  \Phi_{\Pi\Delta}^{-1}\, J_{\mathcal{H}_\Pi}^\star\, \widehat{\mathcal{S}}_\Gamma\, J_{\mathcal{H}_\Pi}
  = \begin{bmatrix}
      \widehat{\mathcal{S}}_0 & 0 \\
      0 & \widehat{\mathcal{S}}_{\Delta\Delta}
    \end{bmatrix}.
\end{align}
Consequently, $\widehat{\mathcal{S}}_\Gamma$ is an isomorphism if and only if
$\widehat{\mathcal{S}}_0$ is, in which case
\begin{align}\label{eq:Shat_inv}
  \widehat{\mathcal{S}}_\Gamma^{-1}
  = \mathcal{H}_\Pi\, \widehat{\mathcal{S}}_0^{-1}\, \mathcal{H}_\Pi^\star
  + \iota_\Delta\, \widehat{\mathcal{S}}_{\Delta\Delta}^{-1}\,
    \iota_\Delta^\star.
\end{align}
\end{theorem}
\begin{proof}
From \eqref{eq:Phi_def}  and \eqref{eq:JPiDelta_def},
$J_{\mathcal{H}_\Pi}(x_\Pi, x_\Delta)
 = J_{\Pi\Delta}\del[1]{x_\Pi,\, x_\Delta - \widehat{\mathcal{S}}_{\Delta\Delta}^{-1}\widehat{\mathcal{S}}_{\Delta\Pi} x_\Pi}$, and thus it is the composition of $J_{\Pi\Delta}$ and the lower triangular block operator
\begin{align*}
    \begin{bmatrix}
     I_{\widehat{X}_\Pi} & 0  \\
     - \widehat{\mathcal{S}}_{\Delta\Delta}^{-1}\widehat{\mathcal{S}}_{\Delta\Pi} & I_{\widehat{X}_{\Delta}} 
 \end{bmatrix},
\end{align*}
both of which are isomorphisms. Thus, we deduce $J_{\mathcal{H}_\Pi}$ is an isomorphism. The stated formula for $J_{\mathcal{H}_\Pi}^{-1}$ can be verified by direct computation using that $\pi_\Pi \mathcal{H}_\Pi = I_{\widehat{X}_\Pi}$ and
$\pi_\Delta \mathcal{H}_\Pi = -\widehat{\mathcal{S}}_{\Delta\Delta}^{-1}\widehat{\mathcal{S}}_{\Delta\Pi}$. 
Next, since the following decomposition holds,
\begin{align}
  \Phi_{\Pi\Delta}^{-1}\, J_{\mathcal{H}_\Pi}^\star\, \widehat{\mathcal{S}}_\Gamma\, J_{\mathcal{H}_\Pi}
  = \begin{bmatrix}
      \mathcal{H}_\Pi^\star \widehat{\mathcal{S}}_\Gamma \mathcal{H}_\Pi & \mathcal{H}_\Pi^\star \widehat{\mathcal{S}}_\Gamma \iota_\Delta  \\
      \iota_\Delta^\star \widehat{\mathcal{S}}_\Gamma \mathcal{H}_\Pi &  \iota_\Delta^\star \widehat{\mathcal{S}}_\Gamma \iota_\Delta
    \end{bmatrix},
\end{align}
 \eqref{eq:LDU} follows from \Cref{prop:Phi_props} and the self-adjointness of
$\widehat{\mathcal{S}}_\Gamma$.
Moreover, since $J_{\mathcal{H}_\Pi}$ and $\Phi_{\Pi\Delta}$ are isomorphisms,
$\widehat{\mathcal{S}}_\Gamma$ is an isomorphism if and only if the right-hand
side of \eqref{eq:LDU} is, i.e.\ if and only if $\widehat{\mathcal{S}}_0$ is.
In that case, inverting \eqref{eq:LDU} gives
$\widehat{\mathcal{S}}_\Gamma^{-1}
 = J_{\mathcal{H}_\Pi} \operatorname{diag}\del[1]{\widehat{\mathcal{S}}_0^{-1}, \widehat{\mathcal{S}}_{\Delta\Delta}^{-1}} \Phi_{\Pi\Delta}^{-1} J_{\mathcal{H}_\Pi}^\star$,
and evaluating at $f \in \widehat{X}_\Gamma^\star$ via
$\Phi_{\Pi\Delta}^{-1} J_{\mathcal{H}_\Pi}^\star f = \del[1]{\mathcal{H}_\Pi^\star f, \iota_\Delta^\star f}$
yields \eqref{eq:Shat_inv}.
\color{black}
\end{proof}

\begin{remark}[Action of the coarse operator]\label{rem:local_KKT_Pi}
For $x_\Pi \in \widehat{X}_\Pi$, the dual component of
$\mathcal{H}_\Pi x_\Pi$ is
$-\widehat{\mathcal{S}}_{\Delta\Delta}^{-1} \widehat{\mathcal{S}}_{\Delta\Pi} x_\Pi$,
which is computed by the $N$ independent local dual problems involving the local dual subspace operator defined in
\eqref{eq:local_dual_subspace_blocks}. Applying
$\widehat{\mathcal{S}}_0 = \mathcal{H}_\Pi^\star \widehat{\mathcal{S}}_\Gamma \mathcal{H}_\Pi$
therefore consists of independent local solves followed by assembly of the
primal residuals, whereas each application of $\widehat{\mathcal{S}}_0^{-1}$
is a global solve on the finite-dimensional space $\widehat{X}_\Pi$.
\end{remark}

\subsubsection{Well-posedness of the local dual subspace problems}
\label{sss:SDD_wellposed}
We now determine conditions under which the local dual subspace operators $\widehat{\mathcal{S}}_{\Delta \Delta}^i: \widehat{X}_\Delta^i \to (\widehat{X}_\Delta^i)^\star$ are isomorphisms, and thus in light of \Cref{prop:SDD_diagonal}, conditions under which $\widehat{\mathcal{S}}_{\Delta \Delta}: \widehat{X}_\Delta \to \widehat{X}_\Delta^\star$ is an isomorphism.
To this end, we
collect the constraint values of a scalar trace by the \emph{scalar constraint map} for $\bullet \in \cbr{y,p}$,
\begin{align}\label{eq:scalar_constraint_map}
    c_{\Pi_i}^\bullet : V_{\Gamma_i} \to \mathbb{R}^{N_\Pi^i},
    \qquad
    c_{\Pi_i}^\bullet(v) := \del[1]{\ell_k^{i,\bullet}(v)}_{k \in \mathcal{N}^i},
\end{align}
so that, up to a reordering of components, $C_\Pi^i(\xi_i) = \del[1]{c_{\Pi_i}^y(\xi_i^y), c_{\Pi_i}^p(\xi_i^p)}$ for every $\xi_i = (\xi_i^y, \xi_i^p) \in X_{\Gamma_i}$, with $C_\Pi^i$ the product constraint map \eqref{eq:CPi_local_def}. For $\bullet \in \cbr{y,p}$, we define $ \widehat{V}_{\Delta_i}^\bullet := \ker(c_{\Pi_i}^\bullet) \subset V_{\Gamma_i}$.
The local dual subspace then factors as, 
\begin{align}\label{eq:dual_factor}
    \widehat{X}_\Delta^i = \widehat{V}_{\Delta_i}^y \times \widehat{V}_{\Delta_i}^p.
\end{align}
\begin{theorem}[Local Schur complement properties]
\label{thm:local_dual_fredholm}
For each $i = 1,\dots,N$, the local dual subspace operator
$\widehat{\mathcal{S}}_{\Delta \Delta}^i: \widehat{X}_\Delta^i \to (\widehat{X}_\Delta^i)^\star$ is
Fredholm with $\operatorname{ind}(\widehat{\mathcal{S}}_{\Delta \Delta}^i) = 0$ and
\begin{align}\label{eq:SDD_kernel}
    \ker\del[1]{\widehat{\mathcal{S}}_{\Delta\Delta}^i}
    = \cbr[1]{x_\Delta \in \widehat{X}_\Delta^i : \iota_\Delta^i\, x_\Delta \in \ker(\mathcal{S}_{\Gamma_i})}.
\end{align}
If the primal constraint family is compatible in the sense of
\Cref{def:compatible_family}, then $\widehat{\mathcal{S}}_{\Delta\Delta}^i$ is an
isomorphism if and only if
\begin{align}\label{eq:SDD_iso_condition}
    i \notin \mathcal{N}_*
    \qquad \text{or} \qquad
    c_{\Pi_i}^p\del[1]{1|_{\Gamma_i}} \neq 0,
\end{align}
with $\mathcal{N}_*$ defined as in \textup{\Cref{thm:broken_fredholm}(b)}.
\end{theorem}
\begin{proof}
 \emph{Step 1: Fredholm property.}  Since $\widehat{\mathcal{S}}_{\Delta\Delta}^i
  =
  (\iota_\Delta^i)^\star
  \mathcal{S}_{\Gamma_i}
  \iota_\Delta^i$  and $\widehat{X}_\Delta^i \subset \widehat{X}_\Gamma$ has finite codimension, the fact that $\widehat{\mathcal{S}}_{\Delta\Delta}^i$ is Fredholm with $\textup{ind}(\widehat{\mathcal{S}}_{\Delta\Delta}^i) = 0$ follows from \Cref{thm:local_fredholm} and \Cref{prop:finite_rank_perturb}.

  \medskip
  \emph{Step 2: Kernel.} The inclusion $\supseteq$ in \eqref{eq:SDD_kernel} is
immediate. Conversely, let
$\widehat \xi_i = (\widehat\xi_i^y, \widehat\xi_i^p) \in \ker(\widehat{\mathcal{S}}_{\Delta\Delta}^i)$ and set
$(\widehat{y}_i, \widehat{p}_i, \widehat{u}_i, \widehat{\mu}_i)
 := \mathcal{H}_i\, \iota_\Delta^i\, \widehat\xi_i$.
Testing $\widehat{\mathcal{S}}_{\Delta\Delta}^i\, \widehat\xi_i$ against
$\widehat \eta_i := (\widehat \xi_i^y, -\widehat\xi_i^p) \in \widehat{X}_\Delta^i$ and applying \Cref{lem:local_harmonic}(d), we find
\begin{align*}
    0 = \langle \widehat{\mathcal{S}}_{\Delta\Delta}^i \widehat\xi_i, \widehat\eta_i \rangle_{\widehat{X}_\Delta^i}
      = \langle \mathcal{S}_{\Gamma_i} \iota_\Delta^i \widehat\xi_i, \iota_\Delta^i \widehat\eta_i \rangle_{X_{\Gamma_i}}
      = \langle L_i \widehat{y}_i, \widehat{y}_i \rangle_{V_i}
        + \alpha \|\widehat{u}_i\|_{U_i}^2.
\end{align*}
Since
$\langle L_i \widehat{y}_i, \widehat{y}_i \rangle_{V_i} \ge \|M_i \widehat{y}_i\|_{L^2(\Omega_i)}^2$
and $M_i$ is injective, $\widehat{y}_i = 0$ and $\widehat{u}_i = 0$. Next, since the primal constraint family is assumed compatible,
$(\widehat\xi_i^p, 0) \in \widehat{X}_\Delta^i$. Testing $\mathcal{S}_{\Gamma_i}\, \iota_\Delta^i \widehat\xi_i$ with it and using
\Cref{lem:local_harmonic}(d) $(\widehat{p}_i, 0)$,
\begin{align*}
    0 = \langle \mathcal{S}_{\Gamma_i}\, \iota_\Delta^i \widehat\xi_i, (\xi^p, 0) \rangle_{X_{\Gamma_i}}
      = \langle L_i\, \widehat{y}_i + \widetilde{A}_i^\star\, \widehat{p}_i, \widehat{p}_i \rangle_{V_i}
      = \langle \widetilde{A}_i^\star\, \widehat{p}_i, \widehat{p}_i \rangle_{V_i}.
\end{align*}
As $\widetilde{A}_i^\star$ is self-adjoint and positive semidefinite, 
$\widetilde{A}_i^\star\, \widehat{p}_i = 0$. Consequently,  $\iota_\Delta^i \xi \in \ker(\mathcal{S}_{\Gamma_i})$, proving
\eqref{eq:SDD_kernel}.

\medskip
\emph{Step 3: Invertibility.} By Step 1,
$\widehat{\mathcal{S}}_{\Delta\Delta}^i$ is an isomorphism if and only if its
kernel is trivial. If $i \notin \mathcal{N}_*$, then
$\ker(\mathcal{S}_{\Gamma_i}) = \cbr{0}$ and \eqref{eq:SDD_kernel} is trivial.
If $i \in \mathcal{N}_*$, then
$\ker(\mathcal{S}_{\Gamma_i}) = \cbr[1]{(0, c_i \cdot 1|_{\Gamma_i}) : c_i \in \mathbb{R}}$
by \eqref{eq:local_schur_kernel}, and
$(0, c_i \cdot 1|_{\Gamma_i}) \in \widehat{X}_\Delta^i$ if and only if
$c_i c_\Pi^i(1|_{\Gamma_i}) = 0$. Hence \eqref{eq:SDD_kernel} is trivial if
and only if $c_\Pi^i(1|_{\Gamma_i}) \neq 0$.
 Hence \eqref{eq:SDD_kernel} is trivial if
and only if \eqref{eq:SDD_iso_condition} holds.
\end{proof}

\begin{corollary}\label{cor:SDD_iso}
$\widehat{\mathcal{S}}_{\Delta\Delta}$ is an isomorphism if and only if
\eqref{eq:SDD_iso_condition} holds for every $i = 1,\dots,N$. 
\end{corollary}
\begin{proof}
Combine \Cref{thm:local_dual_fredholm} with \Cref{prop:SDD_diagonal}.
\end{proof}

\subsubsection{Two-level additive Schwarz form}
Finally, we note that the BDDC preconditioner can be interpreted as a two-level additive Schwarz method:
\begin{theorem}[Two-level additive Schwarz form]\label{thm:twolevel}
Suppose $\widehat{\mathcal{S}}_{\Delta \Delta}$ and $\widehat{\mathcal{S}}_0$ are isomorphisms. Then, defining  $\mathcal{W}_i := \mathcal{W}\, \iota_\Delta\, \jmath_{\Delta}^i$, we have
\begin{align}\label{eq:twolevel}
  \mathcal{P}_{\mathrm{BDDC}}
  = \sum_{i=1}^N \mathcal{W}_i\, (\widehat{\mathcal{S}}_{\Delta\Delta}^i)^{-1}\,
    \mathcal{W}_i^\star
  + \del[1]{\mathcal{W} \mathcal{H}_\Pi}\, \widehat{\mathcal{S}}_0^{-1}\, \del[1]{\mathcal{W} \mathcal{H}_\Pi}^\star.
\end{align}
\end{theorem}
\begin{proof}
Insert \eqref{eq:Shat_inv} into \eqref{eq:BDDC_preconditioner} and expand $\iota_\Delta \widehat{\mathcal{S}}_{\Delta\Delta}^{-1} \iota_\Delta^\star$ with \eqref{eq:SDD_inv_decomp}.
\end{proof}
An application of $\mathcal{P}_{\mathrm{BDDC}}$ thus consists of $N$ independent local dual subspace problems and a finite dimensional global coarse solve of dimension at most $2N_\Pi$.

\section{Numerical results}\label{sec:numerics_results}
We next describe the finite element realization of the interface Schur
complement and BDDC preconditioner, followed by three numerical studies.
The nonlinear optimality system is solved by a semi-smooth Newton method,
and each Newton equation is reduced to the state--adjoint interface and
solved by preconditioned GMRES.  We report \emph{LI}, the average number
of GMRES iterations per Newton step for each nonlinear solve.  The
experiments assess mesh dependence, growth with $H/h$, the effect of the
primal constraints, and sensitivity to $\alpha$, $\beta$, and the
Moreau--Yosida parameter $\gamma$.

\subsection{Discretization and primal constraints}\label{sec:discretization_primal}

All experiments are carried out on the unit square $\Omega=(0,1)^2$.
For a refinement level $\ell\in\mathbb N$, let $\mathcal T_h$ denote
the triangulation of $\Omega$ obtained by subdividing a uniform
$n\times n$ Cartesian grid, $n:=2^{\ell+1}$, into $2n^2$ congruent
triangles, with mesh size $h:=1/n$, and let $\mathcal S_h$ denote the
induced triangulation of $\partial\Omega$. The subdomain partition
$\overline\Omega=\bigcup_{i=1}^N\overline\Omega_i$ of \Cref{sec:dd}
consists of $N=1/H^2$ congruent squares of side $H$, where $1/H$ and
$H/h$ are integers, so that every subdomain boundary is resolved by
$\mathcal T_h$. We vary $H/h\in\cbr{4,8,16}$. In analogy with the
control domain $\mathcal U$ of \Cref{ss:spaces_ops}, we introduce the
control mesh
\begin{align*}
  \mathcal T_h^{\mathcal U} :=
  \begin{cases}
    \mathcal T_h, & \textup{(DC)}, \\
    \mathcal S_h, & \textup{(NC)},
  \end{cases}
\end{align*}
whose cells partition $\mathcal U$. 
The discretization is defined by
the conforming pair
\begin{align*}
  V_h &:= \cbr[1]{v\in C(\overline\Omega)\cap V :
    v|_T \text{ is affine for all } T\in\mathcal T_h} \subset V, \\
  U_h &:= \cbr[1]{w\in L^\infty(\mathcal U) :
    w|_T \text{ is constant for all }
    T\in\mathcal T_h^{\mathcal U}} \subset U.
\end{align*}
The state and adjoint are approximated in $V_h$, and the control and
multiplier in $U_h$. 

\subsubsection{Finite element matrices}
Let $\cbr{\phi_i}_{i=1}^{N_V}$ denote the nodal
basis of $V_h$, associated with the interior vertices of
$\mathcal T_h$ in case (DC) and with all vertices in case (NC), and
let $\cbr{\chi_{T_j}}_{j=1}^{N_U}$ denote the indicator basis of
$U_h$, where $T_1,\dots,T_{N_U}$ enumerate the cells of
$\mathcal T_h^{\mathcal U}$. Discrete functions and their coefficient
vectors, the latter written in boldface, are related by
\begin{align*}
  y_h=\sum_{i=1}^{N_V}(\bm y)_i\,\phi_i,
  \quad
  p_h=\sum_{i=1}^{N_V}(\bm p)_i\,\phi_i,
  \quad
  u_h=\sum_{j=1}^{N_U}(\bm u)_j\,\chi_{T_j},
  \quad
  \mu_h=\sum_{j=1}^{N_U}(\bm \mu)_j\,\chi_{T_j}.
\end{align*}
The stiffness matrix
$A_h$, the control Gram matrix $Q_h$, and the control-to-state
coupling matrix $B_h$ are
\begin{alignat*}{2}
  (A_h)_{ij}
    &:= \langle A\phi_j,\phi_i\rangle_{V}
     = \int_\Omega \nabla\phi_j\cdot\nabla\phi_i\dif x,
    &\quad& i,j\in\cbr{1,\dots,N_V}, \\
  (Q_h)_{ij}
    &:= (\chi_{\kappa_j},\chi_{\kappa_i})_{U}
     = \delta_{ij}\,|\kappa_i|,
    &\quad& i,j\in\cbr{1,\dots,N_U}, \\
  (B_h)_{ij}
    &:= \langle B\chi_{\kappa_j},\phi_i\rangle_{V}
     = (\chi_{\kappa_j},\phi_i)_{L^2(\mathcal U)},
    &\quad& i\in\cbr{1,\dots,N_V},\; j\in\cbr{1,\dots,N_U},
\end{alignat*}
where $|\kappa|$ denotes the measure of the control cell $\kappa$. The lumped mass
matrix is
\begin{align*}
  (M_h)_{ij} := \delta_{ij}\int_\Omega \phi_i\dif x,
\end{align*}
and $\widetilde A_h := A_h+\sigma M_h$. Since $M_h$ is diagonal, so is
the discrete Moreau--Yosida penalty below.

\subsubsection{Symmetrized Newton system}
At a discrete Newton iterate
$\bar x_h=(\bar y_h,\bar p_h,\bar u_h,\bar\mu_h)$, the active sets of
\Cref{ss:active_sets} become index sets. The state active set
$\mathcal A_{\bar y,h}\subset\cbr{1,\dots,N_V}$ consists of the state
nodes lying in $\mathcal A_{\bar y}$, and, since $\bar u_h$ and
$\bar\mu_h$ are constant on each control cell, the active and inactive
control sets are unions of control cells, indexed by
$\mathcal A_{\bar u,\bar\mu,h}$ and $\mathcal I_{\bar u,\bar\mu,h}$ in
$\cbr{1,\dots,N_U}$. For an index set $\mathfrak A$, let
$\Pi_{\mathfrak A}$ be the diagonal projection onto the coordinates it
indexes. The discrete counterpart of
$L=M^\star\del[1]{I+\gamma^{-1}
\mathcal P_{\mathcal A_{\bar y}}^\star
\mathcal P_{\mathcal A_{\bar y}}}M$ is then
\begin{align*}
  L_h := \del[1]{I+\gamma^{-1}\Pi_{\mathcal A_{\bar y,h}}}M_h,
\end{align*}
which reduces to $M_h$ when no state constraint is active.

We set
$U_{\mathcal A,h}:=\operatorname{span}\cbr[1]{\chi_{\kappa_j} :
j\in\mathcal A_{\bar u,\bar\mu,h}}\subset U_{\mathcal A}$ and
$X_h := V_h\times V_h\times U_h\times U_{\mathcal A,h}\subset X$.
With $n_{\mathcal A}:=|\mathcal A_{\bar u,\bar\mu,h}|$, let
$P_{\mathcal A}\in\cbr{0,1}^{n_{\mathcal A}\times N_U}$ denote the
matrix selecting the active cells, so that
$P_{\mathcal A}^\top P_{\mathcal A}
=\Pi_{\mathcal A_{\bar u,\bar\mu,h}}$
~\cite{Porcelli2017Preconditioning}. It is the coefficient
representation of the restriction
$\mathcal P_{\mathcal A}|_{U_h}:U_h\to U_{\mathcal A,h}$. The
discrete symmetrized system on $X_h$ reads
\begin{align}\label{eq:discrete_KKT}
  \underbrace{\begin{bmatrix}
    L_h & \widetilde A_h & 0 & 0 \\
    \widetilde A_h & 0 & -B_h & 0 \\
    0 & -B_h^\top & \alpha\,Q_h & Q_h P_{\mathcal A}^\top \\
    0 & 0 & P_{\mathcal A} Q_h & 0
  \end{bmatrix}}_{=:\,\mathcal K_h}
  \begin{bmatrix} \bm y \\ \bm p \\ \bm u \\ \bm\mu_{\mathcal A}\end{bmatrix}
  =
  \begin{bmatrix}\bm f_1 \\ \bm f_2 \\ \bm f_3 \\ \bm f_4\end{bmatrix},
\end{align}
where $\bm y,\bm p,\bm u,\bm\mu_{\mathcal A}$ are the coefficient
vectors of the Newton increment and $\bm f_1,\dots,\bm f_4$ the
coefficient vectors of the right-hand side $f$
of~\eqref{eq:block_KKT} in the nodal and cellwise bases. Note that the matrix
$\mathcal K_h$ is symmetric.

\subsubsection{Subdomain restrictions and interface unknowns}
Since the subdomain boundaries are resolved by $\mathcal T_h$, each
$\Omega_i$ inherits the triangulation
$\mathcal T_{h,i}:=\cbr[1]{T\in\mathcal T_h :
T\subset\overline\Omega_i}$ and the control mesh
$\mathcal T_{h,i}^{\mathcal U}:=\cbr[1]{T\in
\mathcal T_h^{\mathcal U} : T\subset\overline{\mathcal U}_i}$, where
$\mathcal U_i=\Omega_i$ in case (DC) and
$\mathcal U_i=\partial\Omega_i\cap\partial\Omega$ in case (NC), as in
\Cref{ss:spaces}. The local finite element spaces
\begin{align*}
  V_{i,h} := \cbr[1]{v|_{\Omega_i} : v\in V_h} \subset V_i,
  \qquad
  U_{i,h} := \cbr[1]{w|_{\mathcal U_i} : w\in U_h} \subset U_i
\end{align*}
are conforming in the local spaces of \Cref{ss:spaces}, with
$U_{i,h}=\cbr{0}$ for non-control subdomains. 
Each subdomain $\Omega_i$ carries local matrices $A_{i,h}$,
$M_{i,h}$, $Q_{i,h}$, $B_{i,h}$, $\widetilde A_{i,h}$, and $L_{i,h}$,
given by the formulas above with all integrals restricted to
$\Omega_i$ and $\mathcal U_i$, where $\mathcal U_i=\Omega_i$ in case
(DC) and $\mathcal U_i=\partial\Omega_i\cap\partial\Omega$ in case
(NC), and with the control blocks void when $\Omega_i$ is a
non-control subdomain. 
They form the discrete
local KKT matrix $\mathcal K_{i,h}$. With $R_i$ the Boolean matrix
that maps global unknowns to the local ones on $\Omega_i$, the discrete analogue of the assembly identity in
\Cref{cor:KKT_operator_assembly} becomes
\begin{align*}
  \mathcal K_h=\sum_{i=1}^N R_i^\top \mathcal K_{i,h} R_i.
\end{align*}

Controls and multipliers are cellwise constant and carry no
interface degrees of freedom.  The interface
unknowns are therefore the state and adjoint nodal values at the
vertices on $\Gamma$. Writing
$\bm z=(\bm z_I,\bm z_\Gamma)$ for a global coefficient vector
partitioned accordingly, where $\bm z_\Gamma$ collects the interface
values and $\bm z_I$ the remaining interior unknowns,
$\mathcal K_h$ takes the $2\times 2$ block form
\begin{align*}
  \mathcal K_h=\begin{bmatrix}
    \mathcal K_{II,h} & \mathcal K_{I\Gamma,h}\\
    \mathcal K_{\Gamma I,h} & \mathcal K_{\Gamma\Gamma,h}
  \end{bmatrix},
\end{align*}
in which each block is the submatrix of $\mathcal K_h$ with rows and
columns drawn from the corresponding index sets. Interior basis functions
supported in distinct subdomains do not interact, so
$\mathcal K_{II,h}$ is block diagonal with local
interior blocks $\mathcal K_{II,h}^i$.
Each Newton step is reduced to
the interface by the Schur complement
\begin{align}\label{eq:discrete_schur}
  \mathcal S_{\Gamma,h}
  := \mathcal K_{\Gamma\Gamma,h}
   - \mathcal K_{\Gamma I,h}\,\mathcal K_{II,h}^{-1}\,
     \mathcal K_{I\Gamma,h}.
\end{align}
 The reduced interface system
$\mathcal S_{\Gamma,h}\bm z_\Gamma=\bm g_\Gamma$ is solved by
preconditioned GMRES. 

\subsubsection{The BDDC preconditioner}
\label{sss:discrete_BDDC}
The preconditioner for~\eqref{eq:discrete_schur} is built from the
subdomain interface blocks. On each $\Omega_i$, eliminating the
interior unknowns from $\mathcal K_{i,h}$ leaves the local interface
Schur complement $\mathcal S_{\Gamma_i,h}$, and, with $R_{\Gamma_i}$
the Boolean matrix that maps global interface unknowns to the local
ones on $\Gamma_i$,
\begin{align*}
  \mathcal S_{\Gamma,h}
  = \sum_{i=1}^N R_{\Gamma_i}^\top \mathcal S_{\Gamma_i,h} R_{\Gamma_i}.
\end{align*}
The $\mathcal S_{\Gamma_i,h}$ are formed once per Newton step from
factorizations of the local interior blocks, and every product
$\mathcal S_{\Gamma,h}\bm z_\Gamma$ is evaluated by this sum rather
than by reforming the global matrix.

The preconditioner enforces continuity only of primal quantities,
the corner values and edge averages of the interface traces. With
$m:=1/H$ subdomains per side there are $(m-1)^2$ interior corners,
each shared by four subdomains, and $2m(m-1)$ edges, each shared by
two. Both define
primal constraints in the sense of \Cref{def:primal_constraints} over
the discrete trace spaces, with equal state and adjoint functionals
$\ell_k^{i,y}=\ell_k^{i,p}$. The first assigns to each vertex $\nu$
the sharing set $\mathcal N_\nu:=\cbr[1]{j : \nu\in\partial\Omega_j}$
of subdomains meeting $\nu$, $\mathcal N_k:=\mathcal N_\nu$, and the
point evaluations
\begin{align}\label{eq:corner_constraint}
  \ell_k^{i,\bullet}\in V_{\Gamma_i,h}^\star,
  \qquad
  \ell_k^{i,\bullet}(\xi):=\xi(\nu),
  \qquad
  i\in\mathcal N_k,\;\bullet\in\cbr{y,p}.
\end{align}
The second adds, for each edge
$\lambda$ shared by $\mathcal N_k=\cbr{i,j}$, the nodal averages
\begin{align}\label{eq:edge_constraint}
  \ell_k^{i,\bullet}\in V_{\Gamma_i,h}^\star,
  \qquad
  \ell_k^{i,\bullet}(\xi)
    :=\frac{1}{|\mathcal V_\lambda|}
      \sum_{\nu\in\mathcal V_\lambda}\xi(\nu),
  \qquad
  i\in\mathcal N_k,\;\bullet\in\cbr{y,p},
\end{align}
over the mesh vertices $\mathcal V_\lambda$ interior to $\lambda$. Both quantities are single-valued on
conforming traces and identical for state and adjoint and thus compatible in the sense of \Cref{def:compatible_family}.

The functionals form the rows of $C_{\Pi,h}$ on the broken interface
vector $\widetilde{\bm z}=(\bm z_1,\dots,\bm z_N)$, and
$C_{\Pi,h}\widetilde{\bm z}=0$ enforces agreement of the primal
quantities. The partially assembled space $\widehat X_{\Gamma,h}$ enforces the
primal constraints and leaves the dual part $\ker(C_{\Pi,h})$ broken.
On it, $\widehat{\mathcal S}_{\Gamma,h}$ couples subdomains only
through the primal variables. The conforming injection
$R_{\Gamma,h}:X_{\Gamma,h}\to\widehat X_{\Gamma,h}$ is
$\del[0]{R_{\Gamma,h}\xi}_i:=R_{\Gamma_i}\xi$. The averaging
$\mathcal W_h:\widehat X_{\Gamma,h}\to X_{\Gamma,h}$ applies the same
scalar weights to the state and adjoint traces, for
$\widehat\xi=(\widehat\xi_1,\dots,\widehat\xi_N)
\in\widehat X_{\Gamma,h}$ and an interface node $\nu$
\begin{align}\label{eq:discrete_averaging}
  \del[1]{\mathcal W_h\widehat\xi\,}^{\!\bullet}(\nu)
  := \sum_{j\in\mathcal N_\nu}\delta_j(\nu)\,
     \widehat\xi_j^{\;\bullet}(\nu),
  \qquad \bullet\in\cbr{y,p},
\end{align}
with the stiffness-scaling weights
\cite{TosellWidlund2005,KlawonnRheinbach2006}
\begin{align}\label{eq:stiffness_scaling}
  \delta_i(\nu)
  := \frac{\del[1]{A_{i,h}}_{\nu\nu}}
          {\sum_{j\in\mathcal N_\nu}\del[1]{A_{j,h}}_{\nu\nu}},
  \qquad i\in\mathcal N_\nu.
\end{align}
Since $\sum_{i\in\mathcal N_\nu}\delta_i(\nu)=1$,
$\mathcal W_h R_{\Gamma,h}=I_{X_{\Gamma,h}}$. Assembled over subdomains,
$\mathcal W_h=\sum_{i=1}^N R_{\Gamma_i}^\top D_i$ with
$D_i=\operatorname{diag}\del[0]{\delta_i(\nu)}$.

The discrete preconditioner
$\mathcal P_{\mathrm{BDDC},h}
:=\mathcal W_h\,\widehat{\mathcal S}_{\Gamma,h}^{-1}\,\mathcal W_h^\star$
inherits the two-level structure of \Cref{thm:twolevel}. Eliminating
the dual variables by static condensation, as with the interior
variables earlier, splits the solve into $N$ independent local dual
blocks and one coarse problem of dimension at most $2N_\Pi$. Writing
$\widehat{\mathcal S}_{\Delta\Delta,h}^{\,i}$ for the local dual
block, $\mathcal H_{\Pi,h}$ for the matrix of coarse basis vectors,
$\widehat{\mathcal S}_{0,h}
=\mathcal H_{\Pi,h}^\top\widehat{\mathcal S}_{\Gamma,h}
\mathcal H_{\Pi,h}$ for the coarse matrix, and $\mathcal W_{i,h}$ for
$\mathcal W_h$ restricted to the dual unknowns of $\Omega_i$,
\begin{align*}
  \mathcal P_{\mathrm{BDDC},h}
  = \sum_{i=1}^N \mathcal W_{i,h}\,
    \del[1]{\widehat{\mathcal S}_{\Delta\Delta,h}^{\,i}}^{-1}
    \mathcal W_{i,h}^\top
  + \del[1]{\mathcal W_h\mathcal H_{\Pi,h}}\,
    \widehat{\mathcal S}_{0,h}^{-1}\,
    \del[1]{\mathcal W_h\mathcal H_{\Pi,h}}^\top.
\end{align*}
The dual block $\widehat{\mathcal S}_{\Delta\Delta,h}^{\,i}$ is the
restriction of $\mathcal S_{\Gamma_i,h}$ to $\ker(C_{\Pi,h}^i)$, with
$C_{\Pi,h}^i$ the rows of $C_{\Pi,h}$ local to $\Omega_i$. Rather
than form this restriction, the implementation applies
$\del[1]{\widehat{\mathcal S}_{\Delta\Delta,h}^{\,i}}^{-1}$ by
solving the augmented saddle-point system
\begin{align*}
  \begin{bmatrix}
    \mathcal S_{\Gamma_i,h} & (C_{\Pi,h}^i)^\top \\
    C_{\Pi,h}^i & 0
  \end{bmatrix}
  \begin{bmatrix} \bm z \\ \bm\lambda \end{bmatrix}
  =
  \begin{bmatrix} \bm r \\ \bm c \end{bmatrix}
\end{align*}
with $\bm c=0$, enforcing $C_{\Pi,h}^i\bm z=0$ through the Lagrange
multiplier $\bm\lambda$ and keeping $\mathcal S_{\Gamma_i,h}$ in
sparse form. The same factorization yields the local block
$\mathcal H_{\Pi,h}^i$ of columns of $\mathcal H_{\Pi,h}$, the
discrete harmonic extensions of the primal data, from
\begin{align*}
  \begin{bmatrix}
    \mathcal S_{\Gamma_i,h} & (C_{\Pi,h}^i)^\top \\
    C_{\Pi,h}^i & 0
  \end{bmatrix}
  \begin{bmatrix} \mathcal H_{\Pi,h}^i \\ \bm\Lambda_i \end{bmatrix}
  =
  \begin{bmatrix} 0 \\ I \end{bmatrix}.
\end{align*}

\subsection{Numerical experiments}
We consider three representative problems. Experiments~1 and~2 treat
control-constrained distributed and Neumann boundary control,
respectively, and examine the effects of mesh refinement, $H/h$,
$\alpha$, $\beta$, and the choice of coarse constraints. Experiment~3
treats state-constrained distributed control along the continuation
$\gamma=h^2$. The principal findings of the numerical experiments are summarized as follows.
In Experiments~1 and~2, the iteration counts remain nearly constant
under mesh refinement at fixed $H/h$ and increase moderately as $H/h$ grows. Adding edge-average constraints reduces the counts by factors of
$2.8$--$4.1$ in Experiment~1 and $2.4$--$3.3$ in Experiment~2, while
also substantially weakening their dependence on $\alpha$ and $\beta$.
In Experiment~3, the counts for all three methods stabilize after the first few continuation steps; BDDC with corner and edge constraints
remains below $10$ iterations throughout. 

%----------------------------------------------------------------------
\subsubsection{Experiment 1: Control-constrained distributed control}%
\label{sec:experiment1}
%--------------------------------------------

We first consider the benchmark
from~\cite{Stadler2009Elliptic,Porcelli2017Preconditioning} on
$\Omega=(0,1)^2$:
\begin{gather*}
  \min_{(y,u)} J(y,u) :=
    \tfrac{1}{2}\|y-y_d\|_{L^2(\Omega)}^2
  + \tfrac{\alpha}{2}\|u\|_{L^2(\Omega)}^2
  + \beta\|u\|_{L^1(\Omega)} \\[-0.3em]
  \text{subject to } -30\le u\le30 \text{ a.e.\ in }\Omega,\qquad
  -\Delta y=u \text{ in }\Omega,\quad y=0 \text{ on }\partial\Omega.
\end{gather*}
Here
$y_d(x_1,x_2)=\tfrac16\sin(2\pi x_1)\sin(2\pi x_2)\exp(2x_1)$.
There are no state constraints and $\sigma=0$.  We test
$\alpha\in\{10^{-4},10^{-6},10^{-8}\}$ and
$\beta\in\{0,0.0018\}$; the positive value of $\beta$ produces a
nontrivial sparse-control region~\cite{Stadler2009Elliptic}.  We also
report iteration counts for the closely related Balancing Neumann Neumann (BNN) method in~\cite{heinkenschloss2006neumann}, extended to the present setting, which differs from BDDC only in the coarse space.

We collect the results in
\Cref{tab:distributed_control}, and make a number of observations.
First, for fixed $H/h$, the iteration counts remain essentially constant as the number of subdomains increases. Quadrupling the number of subdomains changes the average counts by at most 3.6 iterations for BNN, 1.8 for BDDC with corner constraints, and 2.0 for BDDC with corner and edge constraints. This behavior indicates good scalability with respect to the number of subdomains.

Second, the iteration counts increase with $H/h$ for all three methods.
For BNN, they range from $13.8$--$20.5$ at $H/h=4$, from
$19.6$--$31.1$ at $H/h=8$, and from $26.4$--$41.3$ at $H/h=16$.
The corresponding ranges are $18.5$--$21.3$, $23.2$--$28.8$, and
$28.6$--$36.0$ for BDDC with corner constraints, and $4.8$--$6.3$,
$6.2$--$9.0$, and $9.0$--$12.0$ for BDDC with corner and edge
constraints. Thus, increasing $H/h$ from $4$ to $16$ approximately
doubles the iteration counts for all three methods. This moderate growth
appears consistent with the polylogarithmic dependence on
$H/h$ established for the unconstrained problem
in~\cite{LiuZhang2025b}.

Third, the enriched coarse space for BDDC with corner and edge constraints also exhibits substantially weaker sensitivity
to the regularization parameters. As $\alpha$ decreases from $10^{-4}$ to $10^{-8}$, the average iteration counts change by as much as $12.8$ for BNN and $5.8$ for BDDC with corner constraints, compared with at most $2.8$ for BDDC with corner and edge constraints. A similar contrast is observed for $\beta$: increasing $\beta$ from 0 to 0.0018 reduces the average iteration counts by as much as $8.3$ for BNN and $2.4$ for BDDC with corner constraints, whereas the counts for BDDC with corner and edge constraints change by at most $0.8$.

Finally, BDDC with corner constraints yields iteration counts comparable
to those of BNN, with neither method uniformly outperforming the other.
Adding edge-average constraints substantially improves the performance,
reducing the iteration counts by factors of $2.8$--$4.1$ relative to
both methods and yielding markedly weaker sensitivity to the
regularization parameters.
\color{black}

\begin{table}[h]
\centering
\caption{Average GMRES iteration counts (LI) across
Newton iterations for control-constrained distributed control. Three methods are compared: the BNN method of \cite{heinkenschloss2006neumann}, the proposed BDDC method with only corner constraints, and the proposed BDDC method with both corner and edge constraints.}
\label{tab:distributed_control}
\small
\setlength{\tabcolsep}{4pt}
\renewcommand{\arraystretch}{1.05}
\begin{tabular}{ccc c c c c c c}
\toprule
\multirow{2}{*}{$H/h$} & \multirow{2}{*}{$\ell$} & \multirow{2}{*}{$\log_{10}\alpha$} & \multicolumn{2}{c}{BNN} & \multicolumn{2}{c}{Corners} & \multicolumn{2}{c}{Corners \& Edges} \\
\cmidrule(lr){4-5}\cmidrule(lr){6-7}\cmidrule(lr){8-9}
 & & & $\beta=0$ & $\beta=0.0018$ & $\beta=0$ & $\beta=0.0018$ & $\beta=0$ & $\beta=0.0018$ \\
\midrule
\multirow{6}{*}{4}
 & \multirow{3}{*}{8} & $-4$ & 15.2 & 14.8 & 18.8 & 18.5 & 5.2 & 5.2 \\
 &                    & $-6$ & 17.8 & 16.8 & 21.0 & 19.7 & 6.0 & 6.0 \\
 &                    & $-8$ & 20.3 & 17.4 & 21.3 & 20.3 & 6.1 & 6.3 \\
\cmidrule(lr){2-9}
 & \multirow{3}{*}{9} & $-4$ & 14.3 & 13.8 & 19.5 & 19.0 & 4.8 & 4.8 \\
 &                    & $-6$ & 17.8 & 15.8 & 21.0 & 20.0 & 5.5 & 5.7 \\
 &                    & $-8$ & 20.5 & 17.2 & 21.0 & 20.8 & 6.0 & 6.0 \\
\midrule
\multirow{9}{*}{8}
 & \multirow{3}{*}{8} & $-4$ & 23.0 & 22.0 & 24.0 & 23.2 & 7.2 & 7.2 \\
 &                    & $-6$ & 28.9 & 25.5 & 27.1 & 26.3 & 8.0 & 8.2 \\
 &                    & $-8$ & 27.5 & 26.0 & 28.8 & 26.4 & 9.0 & 8.7 \\
\cmidrule(lr){2-9}
 & \multirow{3}{*}{9} & $-4$ & 21.5 & 20.4 & 24.5 & 24.0 & 6.8 & 6.8 \\
 &                    & $-6$ & 28.6 & 24.5 & 27.3 & 25.5 & 7.6 & 7.7 \\
 &                    & $-8$ & 30.1 & 25.3 & 27.3 & 26.4 & 8.2 & 8.3 \\
\cmidrule(lr){2-9}
 & \multirow{3}{*}{10}& $-4$ & 22.0 & 19.6 & 25.0 & 24.0 & 6.2 & 6.2 \\
 &                    & $-6$ & 27.8 & 23.2 & 27.0 & 25.5 & 7.8 & 7.4 \\
 &                    & $-8$ & 31.1 & 24.7 & 27.0 & 26.3 & 8.0 & 8.0 \\
\midrule
\multirow{9}{*}{16}
 & \multirow{3}{*}{8} & $-4$ & 31.5 & 29.6 & 30.0 & 29.2 & 9.2  & 9.4  \\
 &                    & $-6$ & 37.1 & 34.5 & 33.9 & 32.5 & 10.4 & 10.3 \\
 &                    & $-8$ & 38.5 & 35.3 & 34.7 & 33.3 & 12.0 & 11.2 \\
\cmidrule(lr){2-9}
 & \multirow{3}{*}{9} & $-4$ & 30.0 & 27.8 & 30.2 & 28.6 & 9.2  & 9.0  \\
 &                    & $-6$ & 40.0 & 33.0 & 33.8 & 31.9 & 10.0 & 10.1 \\
 &                    & $-8$ & 38.6 & 33.7 & 36.0 & 33.6 & 11.0 & 10.3 \\
\cmidrule(lr){2-9}
 & \multirow{3}{*}{10}& $-4$ & 28.5 & 26.4 & 30.3 & 29.0 & 9.0  & 9.0  \\
 &                    & $-6$ & 37.0 & 31.4 & 34.5 & 32.5 & 10.0 & 10.0 \\
 &                    & $-8$ & 41.3 & 33.0 & 34.4 & 32.7 & 10.0 & 10.0 \\
\bottomrule
\end{tabular}
\end{table}

%----------------------------------------------------------------------
\subsubsection{Experiment 2: Control-constrained Neumann control}%
\label{sec:experiment2}
%----------------------------------------------------------------------

We next consider the boundary-control analogue of Experiment~1. We
retain the same desired state, impose the control bounds $u_a=-2$,
$u_b=3$ on $\partial\Omega$, and take the reaction coefficient
$\sigma=10^{-2}$ to ensure solvability of the local Neumann problems on
floating subdomains without special treatment of the constant mode:
\begin{gather*}
  \min_{(y,u)} J(y,u) :=
    \tfrac{1}{2}\|y-y_d\|_{L^2(\Omega)}^2
  + \tfrac{\alpha}{2}\|u\|_{L^2(\partial\Omega)}^2
  + \beta\|u\|_{L^1(\partial\Omega)} \\[-0.3em]
  \text{subject to } u_a \le u \le u_b \text{ a.e.\ on } \partial\Omega, \quad
  -\Delta y + \sigma y = 0 \text{ in }\Omega,\ \partial_n y = u \text{ on }\partial\Omega.
\end{gather*}
The control variables occur only on subdomains adjacent to
$\partial\Omega$.  We test
$\alpha\in\{10^{-4},10^{-5},10^{-6}\}$ and
$\beta\in\{0,0.0018\}$ and compare the two BDDC coarse spaces.
\Cref{tab:neumann_control} reports the results. The iteration counts are
higher than in Experiment~1, but the same general trends are observed.

First, for fixed $H/h$, the counts remain nearly constant as the number
of subdomains increases. Under refinement, they vary by at most $5.6$
iterations for BDDC with corner constraints and by at most $1.6$ for
BDDC with corner and edge constraints.

Second, the iteration counts increase moderately with $H/h$. For BDDC
with corner constraints, they range from $26.6$--$33.6$ at $H/h=4$,
from $30.5$--$39.8$ at $H/h=8$, and from $36.7$--$45.2$ at $H/h=16$.
The corresponding ranges for BDDC with corner and edge constraints are
$10.0$--$11.4$, $11.4$--$13.6$, and $13.4$--$15.6$, respectively.
Thus, both coarse spaces exhibit only moderate dependence on $H/h$ over
the tested range.

 Third, BDDC with corner and edge constraints is considerably less
sensitive to the regularization parameters. As $\alpha$ decreases from
$10^{-4}$ to $10^{-6}$, the average iteration counts change by as much
as $5.6$ for BDDC with corner constraints, compared with at most $1.6$
for BDDC with corner and edge constraints. Increasing $\beta$ from $0$
to $0.0018$ changes the counts by as much as $9.3$ for BDDC with corner
constraints, typically reducing them, whereas the counts for BDDC with
corner and edge constraints change by at most $0.8$.

Taken together, these results show that adding edge-average constraints
reduces the iteration counts pointwise by factors of $2.4$--$3.3$ and
substantially weakens their sensitivity to both $\alpha$ and $\beta$.

\begin{table}[h]
\centering
\caption{Average GMRES iteration counts (LI) across
Newton iterations for control-constrained distributed control. Two methods are compared: the proposed BDDC method with only corner constraints, and the proposed BDDC method with both corner and edge constraints.}
\label{tab:neumann_control}
\small
\setlength{\tabcolsep}{4pt}
\renewcommand{\arraystretch}{1.05}
\begin{tabular}{ccc c c c c}
\toprule
\multirow{2}{*}{$H/h$} & \multirow{2}{*}{$\ell$} & \multirow{2}{*}{$\log_{10}\alpha$} & \multicolumn{2}{c}{Corners} & \multicolumn{2}{c}{Corners \& Edges} \\
\cmidrule(lr){4-5}\cmidrule(lr){6-7}
 & & & $\beta=0$ & $\beta=0.0018$ & $\beta=0$ & $\beta=0.0018$ \\
\midrule
\multirow{6}{*}{4}
 & \multirow{3}{*}{8} & $-4$ & 29.0 & 28.2  & 10.0 & 10.0 \\
 &                    & $-5$ & 31.0 & 28.6  & 10.2 & 10.4 \\
 &                    & $-6$ & 33.6 & 26.6 & 11.4 & 10.9 \\
\cmidrule(lr){2-7}
 & \multirow{3}{*}{9} & $-4$ & 28.0 & 27.2 & 10.0 & 10.0 \\
 &                    & $-5$ & 29.0 & 27.6 & 10.0 & 10.0 \\
 &                    & $-6$ & 30.8 & 28.4 & 10.8 & 11.2 \\
\midrule
\multirow{9}{*}{8}
 & \multirow{3}{*}{8}  & $-4$ & 37.8 & 36.2 & 12.0 & 12.0 \\
 &                     & $-5$ & 38.8 & 36.4 & 12.8 & 13.0 \\
 &                     & $-6$ & 37.0 & 33.1 & 13.6 & 13.1 \\
\cmidrule(lr){2-7}
 & \multirow{3}{*}{9}  & $-4$ & 36.0 & 34.6 & 12.0 & 12.0 \\
 &                     & $-5$ & 37.8 & 35.2 & 12.0 & 12.0 \\
 &                     & $-6$ & 38.2 & 35.0 & 13.0 & 13.4 \\
\cmidrule(lr){2-7}
 & \multirow{3}{*}{10} & $-4$ & 34.4 & 33.2 & 11.4 & 11.7 \\
 &                     & $-5$ & 36.2 & 33.7 & 12.0 & 12.0 \\
 &                     & $-6$ & 39.8 & 30.5 & 12.0 & 12.1 \\
\midrule
\multirow{9}{*}{16}
 & \multirow{3}{*}{8}  & $-4$ & 44.0 & 43.6 & 14.0 & 14.0 \\
 &                     & $-5$ & 40.6 & 42.2 & 14.8 & 15.0 \\
 &                     & $-6$ & 38.6 & 38.1 & 15.6 & 14.8 \\
\cmidrule(lr){2-7}
 & \multirow{3}{*}{9}  & $-4$ & 45.0 & 43.2 & 14.0 & 14.0 \\
 &                     & $-5$ & 44.6 & 42.2 & 14.0 & 14.0 \\
 &                     & $-6$ & 39.4 & 39.5 & 14.8 & 15.0 \\
\cmidrule(lr){2-7}
 & \multirow{3}{*}{10} & $-4$ & 43.0 & 41.5 & 13.4 & 13.5 \\
 &                     & $-5$ & 45.2 & 42.0 & 14.0 & 14.0 \\
 &                     & $-6$ & 44.2 & 36.7 & 14.0 & 14.1 \\
\bottomrule
\end{tabular}
\end{table}

%----------------------------------------------------------------------
\subsubsection{Experiment 3: State-constrained distributed control}%
\label{sec:experiment3}
%----------------------------------------------------------------------

Finally, we test a distributed-control, state-constrained benchmark
from~\cite[Section 4, Example 1]{HintermuellerHinze2009} on
$\Omega=(0,1)^2$ with homogeneous Dirichlet boundary conditions:
\begin{gather*}
  \min_{(y,u)} J(y,u) :=
    \tfrac{1}{2}\|y-y_d\|_{L^2(\Omega)}^2
  + \tfrac{\alpha}{2}\|u\|_{L^2(\Omega)}^2
  \\[-0.3em]
  \text{subject to } y \le y_b \text{ a.e.\ in }\Omega,\quad
  -\Delta y = u \text{ in }\Omega,\ y = 0 \text{ on }\partial\Omega.
\end{gather*}
We take $\alpha=0.1$, $y_b=0.01$, and desired state
$y_d(x_1,x_2)=10\bigl(\sin(2\pi x_1)+x_2\bigr)$; no control constraints
are imposed and $\beta = 0$. 
Following~\cite{HintermuellerHinze2009}, we balancing the regularization and discretization contributions
by selecting the parameter continuation strategy $\gamma=h^2$, with $h$ the mesh size of the finite element grid.
The continuation loop proceeds as follows. Starting from $\ell=4$
($2048$ elements, $\gamma\approx 9.77\times10^{-4}$), we uniformly
refine the mesh at each step through $\ell=11$ ($3.36\times 10^7$
elements, $\gamma\approx 5.96\times 10^{-8}$). At each refinement, the
coarse-grid iterate is prolonged to the fine grid, and $\gamma$ is updated to $h^2$. The
subdomain partition is chosen to maintain the fixed ratio $H/h=16$
throughout.

\Cref{fig:state_constraints} shows that, after an initial increase on
the coarser meshes, the iteration counts for all three methods become
essentially independent of further refinement along the continuation
$\gamma=h^2$. From $\ell=7$ onward, the counts remain close to $29$--$30$
for BNN and BDDC with corner constraints, and between $8$ and $9$ for
BDDC with corner and edge constraints. The latter method also exhibits
only mild variation over the full continuation and requires roughly one
third as many iterations as the other two methods. Note that, since $h$ and
$\gamma$ are varied simultaneously, these results demonstrate stability
along the coupled continuation path rather than independence of
$\gamma$ alone.

\begin{figure}[ht]
    \centering
\includegraphics[width=0.72\linewidth]{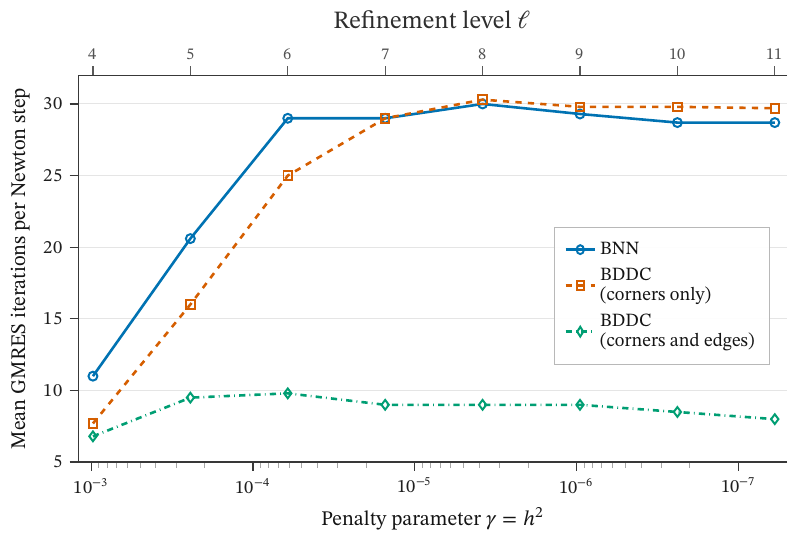}
    \caption{Experiment~3: state-constrained benchmark following the
    continuation strategy $\gamma=h^2$ with $H/h=16$ fixed. Starting
    from mesh refinement $\ell=4$ ($2048$ elements,
    $\gamma\approx 9.77\times 10^{-4}$) through $\ell=11$
    ($3.36\times 10^7$ elements,
    $\gamma\approx 5.96\times 10^{-8}$), the average number of GMRES
    iterations per Newton step remains nearly constant for each
    preconditioner.}
    \label{fig:state_constraints}
\end{figure}

%======================================================================
\section{Conclusion}\label{sec:conclusion}
%======================================================================

We developed a function-space BDDC preconditioner for the
state--adjoint interface systems produced by active-set semi-smooth
Newton linearizations of elliptic optimal control problems with box
control constraints, $L^1$ sparsity, and Moreau--Yosida regularized
state constraints.  The linearized KKT system is reduced by eliminating the control and multiplier, together with the interior state and adjoint variables, leaving a global interface Schur complement system posed on the state and adjoint traces. The BDDC preconditioner is constructed by partially assembling the operator formed from the local subdomain Schur complements, enforcing continuity only of selected primal trace quantities.
The analysis establishes the global--local equivalence and assembly of
the interface operator, characterizes the kernels of the local and
broken Schur complements, and gives conditions under which the
partially assembled problem and local dual-subspace problems are
well posed.

The interface Schur complement is self-adjoint but generally
indefinite, whereas left BDDC preconditioning produces a generally
nonsymmetric operator; we therefore solve the preconditioned system
with GMRES. Across the distributed- and Neumann-control experiments,
the iteration counts are nearly independent of mesh refinement and of
the number of subdomains at fixed $H/h$, and increase only moderately
as $H/h$ grows. Enriching the coarse space with edge-average
constraints substantially reduces the iteration counts and weakens
their sensitivity to $\alpha$ and $\beta$. For the state-constrained
problem, continuation with $\gamma=h^2$ yields nearly constant
iteration counts along the tested path, indicating little sensitivity
to the Moreau--Yosida parameter in this regime.

A rigorous convergence theory for GMRES applied to these indefinite
preconditioned interface systems remains open. Further directions
include parallel scalability studies measuring wall-clock time as a
function of $H/h$ and the number of processor cores, a sharper
analysis of parameter dependence, adaptive selection of primal
constraints, and multilevel BDDC, as well as extensions to more complicated models including optimal flow control
problems.

\bibliographystyle{siamplain}
\bibliography{references_new}

\appendix

\section{Auxiliary block-operator results}
\label{s:block_theory}

This appendix collects several results on block operator theory that are used repeatedly throughout the paper and may also be of independent interest. It first gives an abstract identification of an assembled homogeneous space with the product of its local homogeneous components, together with the resulting block-diagonal representation and inverse formula. It then records a number of results on the Fredholm theory of saddle-point operators, Schur complements, and restrictions to finite-codimensional subspaces.

\subsection{Homogeneous decompositions of assembled spaces}
\label{ss:homogeneous_assembly}

A common theme throughout this work involves global spaces assembled from subdomain components defined by compatibility conditions requiring selected local trace quantities to agree across shared interfaces. If these quantities vanish on every subdomain, the compatibility conditions are satisfied automatically, and any collection of local homogeneous components can be assembled into a global homogeneous element. Restriction therefore identifies the global homogeneous space with the product of the corresponding local kernels. 

The following proposition formulates this argument in an abstract setting.
Throughout this subsection, suppose that $\mathcal{X}$ and
$\mathcal{X}_1,\dots,\mathcal{X}_N$, $\mathcal{Z}_1,\dots,\mathcal{Z}_N$ are given
Hilbert spaces, and set
\begin{align*}
\widetilde{\mathcal{X}}:=\prod_{i=1}^N\mathcal{X}_i, \quad \|\widetilde{x}\|_{\widetilde{\mathcal{X}}}^2 = \sum_{i=1}^N \| \widetilde{x}_i \|_{\mathcal{X}_i}^2,
\end{align*}
 with coordinate
projections
$\widetilde{\pi}_i:\widetilde{\mathcal{X}}\to\mathcal{X}_i$
and injections
$\widetilde{\jmath}_i:\mathcal{X}_i\to\widetilde{\mathcal{X}}$.

\begin{proposition}
\label{prop:homogeneous_decomp}
 Let
$\mathsf{R} \in \mathcal{L}(\mathcal{X}, \widetilde{\mathcal{X}})$ be an isometry and let
$\mathsf{C}_i \in \mathcal{L}(\mathcal{X}_i, \mathcal{Z}_i)$
for $i = 1,\dots,N$. Define
\begin{align*}
  \mathcal{X}_0 := \bigl\{ x \in \mathcal{X} : \mathsf{C}_i \widetilde{\pi}_i \mathsf{R} x = 0, \; i = 1,\dots,N \bigr\},
  \quad
  \mathcal{X}_0^i := \ker(\mathsf{C}_i),
  \quad
  \widetilde{\mathcal{X}}_0 := \prod_{i=1}^N \mathcal{X}_0^i,
\end{align*}
where $\widetilde{\mathcal{X}}_0$ is equipped with the norm
inherited from $\widetilde{\mathcal{X}}$
and with coordinate projection $\widetilde{\pi}_{0}^i: \widetilde{\mathcal{X}}_0 \to \mathcal{X}_0^i$ and injection $\widetilde{\jmath}_{0}^{\, i}: \mathcal{X}_0^i \to \widetilde{\mathcal{X}}_0$,
and let $\iota_0 : \mathcal{X}_0 \hookrightarrow \mathcal{X}$,
$\iota_0^i : \mathcal{X}_0^i \hookrightarrow \mathcal{X}_i$, and
$\widetilde{\iota}_0 : \widetilde{\mathcal{X}}_0 \hookrightarrow \widetilde{\mathcal{X}}$
denote the canonical inclusions. Assume the \emph{gluing condition}
\begin{align} \label{eq:gluing_condition}
  \widetilde{\mathcal{X}}_0 \subseteq \operatorname{ran}(\mathsf{R}).
\end{align}
Then $\mathsf{R}\iota_0$ maps $\mathcal{X}_0$ bijectively onto
$\widetilde{\mathcal{X}}_0$, and its corestriction
$\mathcal{B}_0 : \mathcal{X}_0 \to \widetilde{\mathcal{X}}_0$, defined by
$\widetilde{\iota}_0 \mathcal{B}_0 = \mathsf{R}\iota_0$, is an isometric
isomorphism. The induced maps
\begin{align*}
  \pi_0^i := \widetilde{\pi}_{0}^i \mathcal{B}_0 : \mathcal{X}_0 \to \mathcal{X}_0^i,
  \qquad
  \jmath_0^i := \mathcal{B}_0^{-1} \widetilde{\jmath}_{0}^{\, i} : \mathcal{X}_0^i \to \mathcal{X}_0
\end{align*}
satisfy, for all $i,j = 1,\dots,N$,
\begin{align}
  \sum_{i=1}^N \jmath_0^i\, \pi_0^i &= I_{\mathcal{X}_0},
  \quad
  \pi_0^j\, \jmath_0^i = \delta_{ij} I_{\mathcal{X}_0^i},
  \label{eq:homogeneous_coordinate_identities} \quad 
  \iota_0^i \pi_0^i = \widetilde{\pi}_i \mathsf{R} \iota_0,
  \quad
  \mathsf{R} \iota_0 \jmath_0^i = \widetilde{\jmath}_i \iota_0^i.
 % \label{eq:homogeneous_ambient_identities}
\end{align}
\end{proposition}

\begin{proof}
Note that $\mathsf{R} \iota_0$ is injective as the composition of two injections. To show it is also surjective, we must show that $\mathsf{R} \iota_0(\mathcal{X}_0) = \widetilde{\mathcal{X}}_0$. To this end, suppose $x \in \mathcal{X}_0$. Then, the $i^{\text{th}}$ component of
$\mathsf{R} \iota_0 x \in \widetilde{\mathcal{X}}$ is $\widetilde{\pi}_i \mathsf{R} \iota_0 x$,
and $\mathsf{C}_i \widetilde{\pi}_i \mathsf{R} \iota_0 x = 0$ by the
definition of $\mathcal{X}_0$; hence every component $(\mathsf{R} \iota_0)_i$ lies in
$\mathcal{X}_0^i$ and we conclude that $\mathsf{R} \iota_0(\mathcal{X}_0) \subseteq \widetilde{\mathcal{X}}_0$.
Conversely, let $\widetilde{x} \in \widetilde{\mathcal{X}}_0$. By
\eqref{eq:gluing_condition} and the injectivity of $\mathsf{R}$, there exists a unique $x \in \mathcal{X}$ such that $\widetilde{x} = \mathsf{R} x$. Then
$\mathsf{C}_i \widetilde{\pi}_i \mathsf{R} x
 = \mathsf{C}_i \widetilde{\pi}_i \widetilde{x} = 0$
for every $i$, since $\widetilde{\pi}_i \widetilde{x} \in \ker(\mathsf{C}_i)$.
Consequently, $x \in \mathcal{X}_0$ and $\widetilde{x} = \mathsf{R} \iota_0 x$ from which we conclude $\widetilde{\mathcal{X}}_0 \subseteq \mathsf{R} \iota_0(\mathcal{X}_0)$. 
Thus, $\mathsf{R} \iota_0$ maps $\mathcal{X}_0$ bijectively onto
$\widetilde{\mathcal{X}}_0$.
Consequently, the corestriction $\mathcal{B}_0 : \mathcal{X}_0 \to \widetilde{\mathcal{X}}_0$ is well defined,
bijective, and therefore an isometric isomorphism (owing to the fact that $\mathsf{R}$, $\iota_0$, and
$\widetilde{\iota}_0$ are isometries).

Finally, we prove the identities \eqref{eq:homogeneous_coordinate_identities}.  
The first two identities follow from \eqref{eq:resolution} in \Cref{rem:notation}: 
\begin{align*}
   \sum_{i=1}^N \jmath_0^i\, \pi_0^i &= \mathcal{B}_0^{-1} \del[2]{\sum_{i=1}^N  \widetilde{\jmath}_{0}^{\, i}\widetilde{\pi}_{0}^i} \mathcal{B}_0 = \mathcal{B}_0^{-1} I_{\widetilde{X}_0} \mathcal{B}_0 = I_{\mathcal{X}_0}, \\
       \pi_0^j \jmath_0^{\, j} &= \widetilde{\pi}_0^i \mathcal{B}_0 \mathcal{B}_0^{-1} \widetilde{\jmath}_0^{\, i} =  \widetilde{\pi}_0^i \widetilde{\jmath}_0^{\, i} = \delta_{ij} I_{\mathcal{X}_0^i}.
\end{align*}
The remaining two identities are consequences of fact that $\widetilde{\iota}_0 \mathcal{B}_0 = \mathsf{R}\iota_0$ and the following elementary  identities: for
$i = 1,\dots,N$,
\begin{align} \label{eq:product_inclusion_identities}
  \widetilde{\pi}_i \widetilde{\iota}_0 = \iota_0^i\, \widetilde{\pi}_{0}^i,
  \qquad
  \widetilde{\iota}_0 \widetilde{\jmath}_{0}^{\, i} = \widetilde{\jmath}_i \iota_0^i.
\end{align}
Indeed, it holds that
\begin{align*}
  \iota_0^i \pi_0^i
  &= \iota_0^i \widetilde{\pi}_{0}^i \mathcal{B}_0
  = \widetilde{\pi}_i \widetilde{\iota}_0 \mathcal{B}_0
  = \widetilde{\pi}_i \mathsf{R} \iota_0,
 \\
  \mathsf{R} \iota_0 \jmath_0^i
  &= \widetilde{\iota}_0 \mathcal{B}_0 \mathcal{B}_0^{-1} \widetilde{\jmath}_{0}^i
  = \widetilde{\iota}_0 \widetilde{\jmath}_{0}^i
  = \widetilde{\jmath}_i \iota_0^i. 
\end{align*}
\end{proof}

\begin{corollary}
\label{cor:homogeneous_block}
Retaining the setting of \Cref{prop:homogeneous_decomp}, let
$\mathsf{T}_i \in \mathcal{L}(\mathcal{X}_i, \mathcal{X}_i^\star)$ for
$i = 1,\dots,N$ and define
$\mathsf{T}_0 := \iota_0^\star \mathsf{R}^\star \bigl( \bigoplus_{i=1}^N \mathsf{T}_i \bigr) \mathsf{R} \iota_0
 : \mathcal{X}_0 \to \mathcal{X}_0^\star$. Then
\begin{align} \label{eq:homogeneous_block}
  (\mathcal{B}_0^{-1})^\star \mathsf{T}_0 \mathcal{B}_0^{-1}
  = \bigoplus_{i=1}^N (\iota_0^i)^\star \mathsf{T}_i \iota_0^i.
\end{align}
In particular, $\mathsf{T}_0$ is an isomorphism if and only if $(\iota_0^i)^\star\, \mathsf{T}_i\, \iota_0^i$ is for each $i=1,\dots,N$, in which case
\begin{align*}
  \mathsf{T}_0^{-1}
  = \sum_{i=1}^N \jmath_0^i\, \bigl( (\iota_0^i)^\star\, \mathsf{T}_i\, \iota_0^i \bigr)^{-1} (\jmath_0^i)^\star.
\end{align*}
\end{corollary}

\begin{proof}
Since $\mathsf{R} \iota_0 \mathcal{B}_0^{-1} = \widetilde{\iota}_0$ and
$\widetilde{\pi}_i \widetilde{\iota}_0 = \iota_0^i \widetilde{\pi}_{0}^i$,
\begin{align*}
  (\mathcal{B}_0^{-1})^\star \mathsf{T}_0 \mathcal{B}_0^{-1}
  = \widetilde{\iota}_0^{\,\star} \del[2]{ \sum_{i=1}^N \widetilde{\pi}_i^\star \mathsf{T}_i\, \widetilde{\pi}_i } \widetilde{\iota}_0
  = \sum_{i=1}^N (\widetilde{\pi}_{0}^i )^\star  (\iota_0^i)^\star \mathsf{T}_i \iota_0^i \widetilde{\pi}_{0}^i,
\end{align*}
which is \eqref{eq:homogeneous_block}. The equivalence of invertibility and
the inverse formula follow from \eqref{eq:block_diag_inv} of
\Cref{rem:notation}, using $\mathcal{B}_0^{-1} \widetilde{\jmath}_{0}^{\, i} = \jmath_0^i$.
\end{proof}

\subsection{Fredholm theory of block operators}
\label{ss:fredholm_block}

\begin{proposition}[Fredholm double saddle-point operators]
\label{prop:KKT_fredholm}
Let $\mathcal{X}_1,\mathcal{X}_2,\mathcal{X}_3,\mathcal{X}_4$ be
Hilbert spaces, set
$\mathcal{X}:=\mathcal{X}_1\times\mathcal{X}_2\times
\mathcal{X}_3\times\mathcal{X}_4$, and define
\[
  \mathsf{K}
  :=
  \begin{bmatrix}
    \mathsf{L} & \mathsf{A}^\star & 0 & 0 \\
    \mathsf{A} & 0 & -\mathsf{B} & 0 \\
    0 & -\mathsf{B}^\star & \mathsf{D} & \mathsf{P}^\star \\
    0 & 0 & \mathsf{P} & 0
  \end{bmatrix}
  :
  \mathcal{X}\to\mathcal{X}^\star.
\]
Assume that $\mathsf{A} \in \mathcal{L} (\mathcal{X}_1,\mathcal{X}_2^\star)$ is
Fredholm, that $\mathsf{D} \in \mathcal{L}(\mathcal{X}_3,\mathcal{X}_3^\star)$ and
$\mathsf{P}\mathsf{D}^{-1}\mathsf{P}^\star
\in \mathcal{L}(\mathcal{X}_4,\mathcal{X}_4^\star)$ are isomorphisms, and that
$\mathsf{L} \in \mathcal{L}(\mathcal{X}_1,\mathcal{X}_1^\star)$ and
$\mathsf{B} \in \mathcal{L}(\mathcal{X}_3,\mathcal{X}_2^\star)$ are compact.
Then $\mathsf{K}$ is Fredholm with $\operatorname{ind}(\mathsf{K})=0$.
When $\mathcal{X}_3=\mathcal{X}_4=\{0\}$, the same conclusion holds for the reduced operator
\begin{align*}
  \begin{bmatrix}
    \mathsf{L} & \mathsf{A}^\star \\
    \mathsf{A} & 0
  \end{bmatrix}
  :
  \mathcal{X}_1\times\mathcal{X}_2
  \to
  \mathcal{X}_1^\star\times\mathcal{X}_2^\star.
\end{align*}
\end{proposition}

\begin{proof}
Write $\mathsf{K}=\mathsf{T}+\mathsf{C}$, where
\begin{align*}
  \mathsf{T}
  :=
  \begin{bmatrix}
    0 & \mathsf{A}^\star & 0 & 0 \\
    \mathsf{A} & 0 & 0 & 0 \\
    0 & 0 & \mathsf{D} & \mathsf{P}^\star \\
    0 & 0 & \mathsf{P} & 0
  \end{bmatrix},
  \qquad
  \mathsf{C}
  :=
  \begin{bmatrix}
    \mathsf{L} & 0 & 0 & 0 \\
    0 & 0 & -\mathsf{B} & 0 \\
    0 & -\mathsf{B}^\star & 0 & 0 \\
    0 & 0 & 0 & 0
  \end{bmatrix}.
\end{align*}
By assumption, $\mathsf{C}$ is compact.
Moreover, $\mathsf{T}=\mathsf{T}_1\oplus\mathsf{T}_2$, where
\begin{align*}
  \mathsf{T}_1
  :=
  \begin{bmatrix}
    0 & \mathsf{A}^\star \\
    \mathsf{A} & 0
  \end{bmatrix},
  \qquad
  \mathsf{T}_2
  :=
  \begin{bmatrix}
    \mathsf{D} & \mathsf{P}^\star \\
    \mathsf{P} & 0
  \end{bmatrix}.
\end{align*}
Since $\mathsf{A}$ is Fredholm,~\cite[Proposition~XI.3.4]{conway2007course} implies that $\mathsf{A}^\star$ is Fredholm
and $\operatorname{ind}(\mathsf{A}^\star)
=-\operatorname{ind}(\mathsf{A})$. 
Also,
$\ker(\mathsf{T}_1)=\ker(\mathsf{A})\times\ker(\mathsf{A}^\star)$ and
$\operatorname{ran}(\mathsf{T}_1)
=\operatorname{ran}(\mathsf{A}^\star)\times\operatorname{ran}(\mathsf{A})$.
Hence $\mathsf{T}_1$ is Fredholm and $\operatorname{ind}(\mathsf{T}_1)
  = \operatorname{ind}(\mathsf{A})
    + \operatorname{ind}(\mathsf{A}^\star)
  = 0.$
Next, $\mathsf{T}_2$ is an isomorphism, since it admits the block
factorization
\[
  \mathsf{T}_2
  =
  \begin{bmatrix}
    I & 0 \\
    \mathsf{P}\mathsf{D}^{-1} & I
  \end{bmatrix}
  \begin{bmatrix}
    \mathsf{D} & 0 \\
    0 & -\mathsf{P}\mathsf{D}^{-1}\mathsf{P}^\star
  \end{bmatrix}
  \begin{bmatrix}
    I & \mathsf{D}^{-1}\mathsf{P}^\star \\
    0 & I
  \end{bmatrix},
\]
and each factor is an isomorphism by hypothesis. Therefore $\mathsf{T}$ is Fredholm. Since $\mathsf{T}_2$ is an
isomorphism, $\operatorname{ind}(\mathsf{T}_2)=0$, and~\cite[Proposition~XI.3.4]{conway2007course} gives
$ \operatorname{ind}(\mathsf{T})
  = \operatorname{ind}(\mathsf{T}_1)
    + \operatorname{ind}(\mathsf{T}_2)
  = 0.$ Then,~\cite[Theorem~XI.3.11]{conway2007course} yields that the
compact perturbation $\mathsf{K}=\mathsf{T}+\mathsf{C}$ is Fredholm and $\operatorname{ind}(\mathsf{K})
  = \operatorname{ind}(\mathsf{T})
  = 0$. 
Finally, if $\mathcal{X}_3=\mathcal{X}_4=\{0\}$, the same argument applies to
\[
  \begin{bmatrix}
    \mathsf{L} & \mathsf{A}^\star \\
    \mathsf{A} & 0
  \end{bmatrix}
  =
  \begin{bmatrix}
    0 & \mathsf{A}^\star \\
    \mathsf{A} & 0
  \end{bmatrix}
  +
  \begin{bmatrix}
    \mathsf{L} & 0 \\
    0 & 0
  \end{bmatrix},
\]
since the first summand is Fredholm of index zero and the second is
compact.
\end{proof}

\begin{proposition}[Fredholm stability under Schur complementation] \label{prop:schur_fredholm}
Let $\mathcal{Y}$ and $\mathcal{Z}$ be Banach spaces, let
$\mathsf{A}  \in \mathcal{L}(\mathcal{Y},\mathcal{Y}^\star)$,
$\mathsf{B} \in \mathcal{L}(\mathcal{Z}, \mathcal{Y}^\star)$,
$\mathsf{C} \in \mathcal{L}( \mathcal{Y} ,\mathcal{Z}^\star)$, and
$\mathsf{D} \in \mathcal{L}( \mathcal{Z}, \mathcal{Z}^\star)$, and let
$\widehat{\mathsf{K}} : \mathcal{Y} \times \mathcal{Z}
 \to \mathcal{Y}^\star \times \mathcal{Z}^\star$
denote the block operator,
\begin{align*}
  \widehat{\mathsf{K}}
  := \begin{bmatrix}
      \mathsf{A} & \mathsf{B} \\
      \mathsf{C} & \mathsf{D}
    \end{bmatrix},
  \qquad
  \widehat{\mathsf{K}}(y,z)
  := \del[1]{\mathsf{A}\,y + \mathsf{B}\,z,\;
             \mathsf{C}\,y + \mathsf{D}\,z}.
\end{align*}
Let $\mathcal{X}$ be a Banach space and suppose that
$\mathsf{K} : \mathcal{X} \to \mathcal{X}^\star$ is \emph{congruent}
to $\widehat{\mathsf{K}}$, in the sense that there exists an
isomorphism $J : \mathcal{Y} \times \mathcal{Z} \to \mathcal{X}$ with
\begin{align}\label{eq:schur_congruence}
  \widehat{\mathsf{K}} = \Phi^{-1} J^\star \mathsf{K} J,
\end{align}
where $\Phi : \mathcal{Y}^\star \times \mathcal{Z}^\star
\to (\mathcal{Y} \times \mathcal{Z})^\star$ is the canonical
identification \textup{(}cf. \textup{\Cref{rem:notation})}.
If $\mathsf{A}$ is an isomorphism, then the Schur complement
$\mathsf{S} := \mathsf{D}
 - \mathsf{C}\,\mathsf{A}^{-1}\mathsf{B}
 : \mathcal{Z} \to \mathcal{Z}^\star$
satisfies the following properties:
\begin{enumerate}[label=\textup{(\alph*)}]
\item $\mathsf{S}$ is Fredholm if and only if $\mathsf{K}$ is
  Fredholm, and
  $\operatorname{ind}(\mathsf{S}) = \operatorname{ind}(\mathsf{K})$.

\item The map
  $\Theta : \ker(\mathsf{S}) \to \ker(\mathsf{K})$ defined by
  \begin{align}\label{eq:kernel_lift}
    \Theta(z) := J\del[1]{-\mathsf{A}^{-1}\mathsf{B}\,z,\; z}
  \end{align}
  is an isomorphism.  In particular,
  $\dim\ker(\mathsf{K}) = \dim\ker(\mathsf{S})$.

\item $\mathsf{K}$ is an isomorphism if and only if $\mathsf{S}$ is
  an isomorphism.
\end{enumerate}
\end{proposition}
\begin{proof}
By \eqref{eq:schur_congruence},
$\mathsf{K} = (J^{-1})^\star\,\Phi\,\widehat{\mathsf{K}}\,J^{-1}$.
Since $(J^{-1})^\star$, $\Phi$, and $J^{-1}$ are isomorphisms,
\cite[Proposition~XI.3.4 and Theorem~XI.3.7]{conway2007course} give:
$\mathsf{K}$ is Fredholm if and only if $\widehat{\mathsf{K}}$ is,
with $\operatorname{ind}(\mathsf{K}) = \operatorname{ind}(\widehat{\mathsf{K}})$;
$\mathsf{K}$ is an isomorphism if and only if
$\widehat{\mathsf{K}}$ is; and
$\ker(\mathsf{K}) = J\del[1]{\ker(\widehat{\mathsf{K}})}$.

 \medskip
\noindent\textup{(a)}\;
It suffices to show that $\widehat{\mathsf{K}}$  is Fredholm if and only if $\mathsf{S}$ is
  Fredholm, and
  $\operatorname{ind}(\widehat{\mathsf{K}}) = \operatorname{ind}(\mathsf{S}) $.
To this end, since $\mathsf{A}$ is an isomorphism, the following block LDU factorization holds:
\begin{align}\label{eq:block_fredholm_LDU}
  \widehat{\mathsf{K}}
  = \underbrace{
    \begin{bmatrix}
      I_{\mathcal{Y}^\star} & 0 \\
      \mathsf{C}\,\mathsf{A}^{-1} & I_{\mathcal{Z}^\star}
    \end{bmatrix}
    }_{\mathsf{L}}
    \underbrace{
    \begin{bmatrix}
      \mathsf{A} & 0 \\
      0 & \mathsf{S}
    \end{bmatrix}
    }_{\widetilde{\mathsf{D}}}
    \underbrace{
    \begin{bmatrix}
      I_{\mathcal{Y}} & \mathsf{A}^{-1}\mathsf{B} \\
      0 & I_{\mathcal{Z}}
    \end{bmatrix}
    }_{\mathsf{U}}.
\end{align}  
Since $\mathsf{A}^{-1}$ is bounded, $\mathsf{L} : \mathcal{Y}^\star \times \mathcal{Z}^\star
 \to \mathcal{Y}^\star \times \mathcal{Z}^\star$ and
$\mathsf{U} : \mathcal{Y} \times \mathcal{Z}
 \to \mathcal{Y} \times \mathcal{Z}$ are isomorphisms and thus Fredholm with index~$0$. If $\mathsf{S}$ is Fredholm with
$\operatorname{ind}(\mathsf{S}) = m$, then
$\widetilde{\mathsf{D}} = \mathsf{A} \oplus \mathsf{S}$ is Fredholm
with $\operatorname{ind}(\widetilde{\mathsf{D}}) = m$
by~\cite[Proposition~XI.3.4]{conway2007course}, since $\mathsf{A}$
is an isomorphism.  Consequently,
$\widehat{\mathsf{K}} = \mathsf{L}\,\widetilde{\mathsf{D}}\,\mathsf{U}$ is Fredholm
with $\operatorname{ind}(\widehat{\mathsf{K}}) = m$
by~\cite[Theorem~XI.3.7]{conway2007course}.
Conversely, if
$\mathsf{K}$ is Fredholm, then
$\widetilde{\mathsf{D}} = \mathsf{L}^{-1}\,\mathsf{K}\,\mathsf{U}^{-1}$
is a product of Fredholm operators and thus itself Fredholm.  Since
$\mathsf{A}$ is invertible,
\cite[Proposition~XI.3.4]{conway2007course} gives that
$\mathsf{S}$ is Fredholm with
$\operatorname{ind}(\mathsf{S})
 = \operatorname{ind}(\widetilde{\mathsf{D}})
 = \operatorname{ind}(\mathsf{K})$. 

 \medskip
\noindent\textup{(b)}\;
Since $\mathsf{L}$ and $\mathsf{U}$ are isomorphisms,
$(y, z) \in \ker(\widehat{\mathsf{K}})$ if and only if
$\mathsf{U}(y, z) \in \ker(\widetilde{\mathsf{D}})$.  Writing
$\mathsf{U}(y, z) = (y + \mathsf{A}^{-1}\mathsf{B}\,z,\; z)$
and noting that
$\ker(\widetilde{\mathsf{D}})
 = \ker(\mathsf{A}) \times \ker(\mathsf{S})
 = \{0\} \times \ker(\mathsf{S})$,
we obtain $(y, z) \in \ker(\widehat{\mathsf{K}})$ if and only if
$z \in \ker(\mathsf{S})$ and
$y = -\mathsf{A}^{-1}\mathsf{B}\,z$.  Hence the map $z \mapsto \del[1]{-\mathsf{A}^{-1}\mathsf{B}\,z,\; z}$ is a bijection from
$\ker(\mathsf{S})$ onto $\ker(\widehat{\mathsf{K}})$ and the result now follows from $\ker(\mathsf{K}) = J\del[1]{\ker(\widehat{\mathsf{K}})}$.

\medskip
\noindent\textup{(c)}\;
It suffices to show that $\mathsf{K}$ is an isomorphism if and only if $\mathsf{S}$ is an isomorphism, which immediately follows from the fact that $A$ is an isomorphism and the decomposition \eqref{eq:block_fredholm_LDU}.

\end{proof}

\begin{proposition}[Fredholm stability under finite-codimensional
restriction]\label{prop:finite_rank_perturb}
Let $\mathcal{Z}$ be a Banach space and let
$\mathcal{Y} \subseteq \mathcal{Z}$ be a closed subspace with
$\operatorname{codim}(\mathcal{Y}) < \infty$.  Let
$\iota_{\mathcal{Y}} : \mathcal{Y} \to \mathcal{Z}$ denote the
canonical injection.  If
$\mathsf{T} \in \mathcal{L}( \mathcal{Z},\mathcal{Z}^\star)$ is Fredholm with
$\operatorname{ind}(\mathsf{T}) = 0$, then
$\widehat{\mathsf{T}}
 := \iota_{\mathcal{Y}}^\star  \mathsf{T}  \iota_{\mathcal{Y}}
 \in \mathcal{L}( \mathcal{Y}, \mathcal{Y}^\star)$
is Fredholm with $\operatorname{ind}(\widehat{\mathsf{T}}) = 0$.
\end{proposition}

\begin{proof}
Since $\mathcal{Y}$ is closed and of finite codimension, there
exists a finite-dimensional subspace
$\mathcal{W} \subset \mathcal{Z}$ such that
$\mathcal{Z} = \mathcal{Y} \oplus \mathcal{W}$ topologically.
Denote by $\iota_{\mathcal{W}} : \mathcal{W} \to \mathcal{Z}$ the
canonical injection and define
\begin{align*}
  J : \mathcal{Y} \times \mathcal{W} \to \mathcal{Z},
  \qquad
  J(y,w) := \iota_{\mathcal{Y}}\, y + \iota_{\mathcal{W}}\, w,
\end{align*}
which is an isomorphism because the direct sum is topological.
Let $\Phi : \mathcal{Y}^\star \times \mathcal{W}^\star
\to (\mathcal{Y} \times \mathcal{W})^\star$ be the canonical
identification and $\tilde{\jmath}_{\mathcal{Y}}, \tilde{\jmath}_{\mathcal{W}}$ the
coordinate injections of $\mathcal{Y} \times \mathcal{W}$
(\emph{cf}. \Cref{rem:notation}). The operator $\widetilde{\mathsf{T}} := \Phi^{-1} J^\star \mathsf{T} J
  : \mathcal{Y} \times \mathcal{W}
  \to \mathcal{Y}^\star \times \mathcal{W}^\star$ is congruent to $\mathsf{T}$, and since
$J \tilde{\jmath}_{\mathcal{Y}} = \iota_{\mathcal{Y}}$ and
$J \tilde{\jmath}_{\mathcal{W}} = \iota_{\mathcal{W}}$, it has the following block representation:
\begin{align*}
  \widetilde{\mathsf{T}}
  = \begin{bmatrix}
      \iota_{\mathcal{Y}}^\star  \mathsf{T}  \iota_{\mathcal{Y}}
        & \iota_{\mathcal{Y}}^\star \mathsf{T}\, \iota_{\mathcal{W}} \\
      \iota_{\mathcal{W}}^\star \mathsf{T}\, \iota_{\mathcal{Y}}
        & \iota_{\mathcal{W}}^\star \mathsf{T}\, \iota_{\mathcal{W}}
    \end{bmatrix}.
\end{align*}
As $J$, $J^\star$, and $\Phi$ are isomorphisms and $\mathsf{T}$ is
Fredholm of index~$0$, 
$\widetilde{\mathsf{T}}$ is Fredholm with
$\operatorname{ind}(\widetilde{\mathsf{T}}) = 0$
by~\cite[Theorem~XI.3.7]{conway2007course}.
Equip $\mathcal{W}$ with any inner product and let
$\mathcal{R} : \mathcal{W} \to \mathcal{W}^\star$ denote the
associated Riesz isomorphism.  Then
\begin{align*}
  \begin{bmatrix}
    \widehat{\mathsf{T}} & 0 \\
    0 & \mathcal{R}
  \end{bmatrix}
  = \widetilde{\mathsf{T}}
  - \underbrace{
    \begin{bmatrix}
      0
        & \iota_{\mathcal{Y}}^\star \mathsf{T}\, \iota_{\mathcal{W}} \\
      \iota_{\mathcal{W}}^\star \mathsf{T}\, \iota_{\mathcal{Y}}
        & \iota_{\mathcal{W}}^\star \mathsf{T}\, \iota_{\mathcal{W}}
          - \mathcal{R}
    \end{bmatrix}
    }_{\mathsf{F}},
\end{align*}
Each block of $\mathsf{F}$ has either a finite-dimensional domain or
codomain and is therefore finite-rank.  Hence the operator
$\widehat{\mathsf{T}} \oplus \mathcal{R}
 = \widetilde{\mathsf{T}} - \mathsf{F}$
is a finite-rank (and thus compact) perturbation of a Fredholm operator and hence Fredholm
of index~$0$
by~\cite[Theorem~XI.3.11]{conway2007course}. 
Since $\mathcal{R}$ is
an isomorphism (hence Fredholm of index~$0$), it follows
from~\cite[Proposition~XI.3.4(c)]{conway2007course} that
$\widehat{\mathsf{T}}$ is Fredholm with
$\operatorname{ind}(\widehat{\mathsf{T}}) = 0$.
\end{proof}

\end{document}